\documentclass{IEEEtran}
\usepackage{cite}
\usepackage{url}
\usepackage{amsmath,amssymb,amsfonts}
\usepackage{multirow}
\usepackage{amsthm}
\usepackage{algorithmic}
\usepackage{graphicx}
\usepackage{textcomp}
\usepackage{epsfig}
\usepackage{algorithmic}
\usepackage{epstopdf}
\usepackage[ruled]{algorithm}
\usepackage{enumerate}
\usepackage{float}
\usepackage{graphics,graphicx}
\usepackage{subfigure}
\usepackage{array}
\usepackage{mathtools}
\usepackage[justification=centering]{caption}

\DeclareGraphicsExtensions{.png}
\graphicspath{{./Pics/}}
\usepackage{color}
\usepackage{xspace}

\def\BibTeX{{\rm B\kern-.05em{\sc i\kern-.025em b}\kern-.08em
		T\kern-.1667em\lower.7ex\hbox{E}\kern-.125emX}}

\newtheorem{lemma}{Lemma}
\newtheorem{theorem}{Theorem}

\newtheorem{rem}{Remark}
\newtheorem{assumption}{Assumption}
\newtheorem{proposition}{Proposition}
\newtheorem{corollary}{Corollary}
\newtheorem{example}{Example}

\begin{document}

\title{\LARGE \bf Distributed Optimization with Coupling Constraints}

\author{Xuyang Wu, He Wang, and Jie Lu\thanks{X. Wu is with the Division of Decision and Control Systems, KTH Royal Institute of Technology, SE-100 44 Stockholm, Sweden. Email: {\tt xuyangw@kth.se}.}
\thanks{H. Wang and J. Lu are with the School of Information Science and Technology, ShanghaiTech University, 201210 Shanghai, China. Email: {\tt \{wanghe, lujie\}@shanghaitech.edu.cn}.}}
\maketitle

\begin{abstract}
In this paper, we develop a novel distributed algorithm for addressing convex optimization with both nonlinear inequality and linear equality constraints, where the objective function can be a general nonsmooth convex function and all the constraints can be fully coupled. Specifically, we first separate the constraints into three groups, and design two primal-dual methods and utilize a virtual-queue-based method to handle each group of the constraints independently. Then, we integrate these three methods in a strategic way, leading to an integrated primal-dual proximal (IPLUX) algorithm, and enable the distributed implementation of IPLUX. We show that IPLUX achieves an $O(1/k)$ rate of convergence in terms of optimality and feasibility, which is stronger than the convergence results of the state-of-the-art distributed algorithms for convex optimization with coupling nonlinear constraints. Finally, IPLUX exhibits competitive practical performance in the simulations.
\end{abstract}


% =============================================================================================================================
%
% Introduction
%
% =============================================================================================================================
\section{Introduction}\label{sec:introduction}

This paper is intended to develop a distributed algorithm for the following convex composite optimization problem with coupling nonlinear inequality and linear equality constraints:
\begin{equation}\label{eq:proboriginal}
\begin{split}
\underset{x_1\in \mathbb{R}^{d_1},\ldots,x_n\in \mathbb{R}^{d_n}}{\operatorname{minimize}}~~&\sum_{i=1}^n (f_i(x_i)+h_i(x_i))\\
\operatorname{subject~to}~~~~&\sum_{i=1}^n \bar{g}_i^\ell(x_i) \le 0,\quad\forall \ell=1,\ldots,\bar{p},\\
&\sum_{i=1}^n\bar{A}_ix_i=\bar{b}.\\
\end{split}
\end{equation}
In problem~\eqref{eq:proboriginal}, the decision variables $x_i\in\mathbb{R}^{d_i}$ $\forall i=1,\ldots,n$ are possibly of different dimensions. In the objective function, each $f_i:\mathbb{R}^{d_i}\rightarrow \mathbb{R}$ is a convex and differentiable function whose gradient is Lipschitz continuous, and each $h_i:\mathbb{R}^{d_i}\rightarrow \mathbb{R}\cup\{+\infty\}$ is a convex function that can be non-differentiable. Since $h_i$ is allowed to take the value of $+\infty$, it may include an indicator function with respect to any closed convex set $X_i\subseteq\mathbb{R}^{d_i}$, so that a set constraint $x_i\in X_i$ can be implicitly imposed. Each of the $\bar{p}$ nonlinear inequality constraints in \eqref{eq:proboriginal} is described by the sum of $n$ convex functions $\bar{g}_i^\ell:\mathbb{R}^{d_i}\rightarrow\mathbb{R}$, $\forall i=1,\ldots,n$, where each $\bar{g}_i^\ell$ is associated with the variable $x_i$. The linear equality constraint is encoded by the matrices $\bar{A}_i\in \mathbb{R}^{\bar{m}\times d_i}$ $\forall i=1,\ldots,n$ and the vector $\bar{b}\in \mathbb{R}^{\bar{m}}$. %\textcolor{blue}{We assume that there is at least one optimal solution to problem~\eqref{eq:proboriginal}.}

\subsection{Motivating Examples}\label{ssec:examples}

Problem~\eqref{eq:proboriginal} arises in many real applications and is often of huge size. Typically, we distribute the problem data to a network of agents and employ the agents to cooperatively address the entire problem. Presented below are several widely-studied distributed optimization problems that are specialized from \eqref{eq:proboriginal}.

\subsubsection{Consensus Optimization}\label{sssec:consensus}

Consensus optimization is prevalent in multi-agent systems \cite{Nedic15}, which requires all the agents to reach a consensus that minimizes their total costs subject to their individual constraints. This problem is usually formulated as $\operatorname{minimize}_{z\in\cap_{i=1}^nX_i}\sum_{i=1}^n \varphi_i(z)$, where each $\varphi_i$ is a convex function and each $X_i\subseteq\mathbb{R}^d$ is a closed convex set. It can also be rewritten as
\begin{align*}
\underset{x_1,\ldots,x_n\in\mathbb{R}^d}{\operatorname{minimize}}~~~~&\sum_{i=1}^n (\varphi_i(x_i)+\mathcal{I}_{X_i}(x_i))\displaybreak[0]\\
\operatorname{subject~to}~~~& x_1=x_2=\cdots=x_n,
\end{align*}
where $\mathcal{I}_{X_i}(x_i)$ is the indicator function with respect to $X_i$, i.e., $\mathcal{I}_{X_i}(x_i)=0$ if $x_i\in X_i$ and $+\infty$ otherwise. Clearly, it is in the form of \eqref{eq:proboriginal} with only the linear equality constraint.

\subsubsection{Distributed Control}
Suppose we conduct model predictive control (MPC) on a large-scale control system consisting of $n$ interconnected subsystems. The subsystems jointly derive their control actions by iteratively solving convex optimization problems in the following form \cite{Giselsson13}: 
\begin{align*}
\underset{y_i,z_i,\;\forall i=1,\ldots,n}{\operatorname{minimize}}~~&\sum_{i=1}^n\Bigl(\frac{1}{2}y_i^TH_iy_i+w_i^Ty_i+\gamma\|z_i\|_1\Bigr)\displaybreak[0]\\
\operatorname{subject~to}~~~& \sum_{i=1}^n P_{ji}y_i\le q_j,\quad\forall j=1,\ldots,n,\displaybreak[0]\\
&\sum_{i=1}^n A_{ji}'y_i = b_j',\quad\forall j=1,\ldots,n,\displaybreak[0]\\
&\Bigl(\sum_{i=1}^nA_{ji}''y_i\Bigr)-b_j'' = z_j,\quad\forall j=1,\ldots,n.
\end{align*}
Here, $H_i$ is a positive semidefinite matrix, $w_i$ is a given vector, and $\gamma$ is a nonnegative weight. Each $y_i$ consists of the states and the control variables associated with subsystem $i$. In addition, $z_i$ $\forall i=1,\ldots,n$ are auxiliary variables, which, together with the last group of constraints, are used for $\ell_1$-regularization. %transform the originally non-separable $\ell_1$-regularization term into the above separable one in the objective function.
The first group of constraints describe the coupling relations among the subsystems, which are often sparsely connected, and the second group of constraints characterize the system model. Observe that such a distributed MPC problem is also an example of \eqref{eq:proboriginal}.	
%Furthermore, the $n$ subsystems are often assumed to be sparsely connected, i.e., for each subsystem $i$, there exist quite a few $j$'s such that the corresponding $P_{ji}$, $A_{ji}'$, and $A_{ji}''$ are zero. 

\subsubsection{Resource Allocation}\label{sssec:resource}
Many resource allocation problems can be cast in the form of \eqref{eq:proboriginal}. For example, the economic dispatch problem \cite{Yang2016} requires $n$ power generators to determine the minimum-cost power generations by solving
%such that the total power generation cost is minimized, the global power demand is met, and the local generation limits are satisfied:
\begin{align*}
\underset{x_1,\ldots,x_n\in\mathbb{R}}{\operatorname{minimize}}~~~&~\sum_{i=1}^n \varphi_i(x_i)\displaybreak[0]\\
\operatorname{subject~to}~~&~\sum_{i=1}^n x_i = D,\displaybreak[0]\\
~~&~x_i^{\min}\le x_i\le x_i^{\max},\quad\forall i=1,\ldots,n.
\end{align*}
Here, $\varphi_i(x_i)$ represents the power generation cost of generator $i$ when its power generation is $x_i$, $D$ is the global power demand across the entire power network, and $0\le x_i^{\min}<x_i^{\max}<\infty$ indicates the local generation limit of generator $i$. It is straightforward to see that the above problem is also a particular form of \eqref{eq:proboriginal}, where the constraint $x_i^{\min}\le x_i\le x_i^{\max}$ can be either viewed as inequality constraints or merged into the objective function via the indicator function with respect to the closed interval $[x_i^{\min}, x_i^{\max}]$.

\subsection{Literature Review}\label{ssec:literature}

To date, a large volume of distributed optimization algorithms have been proposed (e.g., \cite{Nedic10,Nedic15,Nedic17,ShiW15,ShiW15a,QuG19,Koshal11,Wu19,XiaoL06b,Lakshmanan08,Giselsson13,Yang2016,Nedic17a,ChangTH14,Falsone17,Notarnicola20,Liang19,Liang19a,Wu19a}), which allow a network of agents to address various convex optimization problems only through one-hop interactions. Most of these algorithms (e.g., \cite{Nedic10,Giselsson13, Nedic15, Nedic17, ShiW15, ShiW15a, Yang2016,QuG19, Wu19, XiaoL06b, Lakshmanan08, Koshal11, Nedic17a, Wu19a}) are only guaranteed to solve \emph{linearly}-constrained problems. For example, \cite{Nedic10, Nedic15, Nedic17, ShiW15, ShiW15a, QuG19, Wu19, Wu19a} focus on the consensus optimization problem introduced in Section~\ref{sssec:consensus}, \cite{XiaoL06b, Lakshmanan08, Yang2016,Nedic17a} study resource allocation problems exemplified in Section~\ref{sssec:resource}, and \cite{Giselsson13,Koshal11} tackle problems with linear inequality constraints.

Despite the prevalence of convex optimization problems with \emph{nonlinear} inequality constraints in practice, relatively few methods manage to solve such problems in a distributed setting. It is even more challenging when the nonlinear constraints have a coupling structure. Nevertheless, the consensus-based primal-dual perturbation method \cite{ChangTH14}, the dual-decomposition-based method \cite{Falsone17}, and the relaxation and successive distributed decomposition (RSDD) method \cite{Notarnicola20} are able to handle coupling nonlinear inequality constraints, which achieve asymptotic convergence with diminishing step-sizes, provided that the constraint set is compact. Moreover, the distributed primal-dual gradient method \cite{Liang19} and the distributed dual subgradient method \cite{Liang19a} can deal with both nonlinear inequality and linear equality constraints like problem~\eqref{eq:proboriginal}, where \cite{Liang19} ensures asymptotic convergence and \cite{Liang19a} establishes an $O(\ln k/\sqrt{k})$ convergence rate, provided that the objective function is smooth on the constraint set.

\subsection{Contributions and Paper Organization}

In this paper, we propose a distributed algorithm for solving \eqref{eq:proboriginal}. To this end, we first reformulate \eqref{eq:proboriginal} by decoupling the nonlinear constraints, leading to an equivalent problem that contains three types of constraints, namely, the sparsely-coupled linear equality constraints, the densely/fully-coupled linear equality constraints, and the sparsely-coupled nonlinear inequality constraints. To address this equivalent problem, we develop two new primal-dual proximal methods that are inspired from the Method of Multipliers \cite{Boyd11} and the P-EXTRA algorithm \cite{ShiW15a}, and are intended for the sparse and dense linear equality constraints, respectively. Then, we strategically integrate these two methods with the virtual-queue-based algorithm in \cite{YuH17} which targets at the sparse nonlinear inequality constraints. We also approximate the smooth components in the objective function for enhancing the computational efficiency. The above process constructs an \underline{i}ntegrated \underline{p}rima\underline{l}-d\underline{u}al pro\underline{x}imal algorithm, referred to as IPLUX. We further provide an approach to implementing IPLUX in a distributed fashion over an underlying undirected graph. The contributions of this paper are highlighted as follows:
\begin{enumerate}
\item IPLUX is a novel distributed algorithm for solving the constrained convex optimization problem~\eqref{eq:proboriginal} that generalizes a number of common distributed optimization problems such as those described in Section~\ref{ssec:examples}. 
\item IPLUX achieves optimality and feasibility at a rate of $O(1/k)$, where $k$ is the number of iterations. This is stronger than the convergence results of the existing distributed methods \cite{ChangTH14,Falsone17,Notarnicola20,Liang19,Liang19a} that are also guaranteed to solve convex optimization problems with \emph{coupling nonlinear} inequality constraints. In particular, the $O(1/k)$ rate of IPLUX is faster than the $O(\ln k/\sqrt{k})$ rate in \cite{Liang19a} and is obtained under less restrictive assumptions. Moreover, \cite{ChangTH14,Falsone17,Notarnicola20,Liang19} only prove asymptotic convergence. 
\item We illustrate the superior practical convergence performance of IPLUX via a couple of numerical examples.
\end{enumerate}

The outline of the paper is as follows: Section~\ref{sec:probform} transforms problem~\eqref{eq:proboriginal} into an equivalent form, which facilitates the development of IPLUX in Section~\ref{sec:algdevelop}. Section~\ref{sec:convanal} is dedicated to the convergence analysis, and Section~\ref{sec:numericalexample} presents the simulation results. Section~\ref{sec:conclusion} then concludes the paper. 

A preliminary conference version of this paper can be found in \cite{WuX20}, which contains no proofs and addresses a special case of problem~\eqref{eq:proboriginal} by assuming the constraints to be sparse. Accordingly, this paper generalizes both the algorithm development and the convergence analysis in \cite{WuX20}.

\subsection{Notation} 
For any set $X\subseteq\mathbb{R}^n$, $\operatorname{rel\;int} X$ represents its relative interior. We use $\{\cdot,\cdot\}$ to denote an unordered pair, $\|\cdot\|$ the Euclidean vector norm, and $\|\cdot\|_1$ the $\ell_1$ vector norm. For any $x,y\in \mathbb{R}^n$, $\max\{x,y\}\in\mathbb{R}^n$ is the element-wise maximum of $x$ and $y$. %Given $x_1,\ldots,x_m\in\mathbb{R}^n$, $\mathbf{x}=[(x_1)^T,\ldots,(x_m)^T]^T\in\mathbb{R}^{mn}$ means the vector obtained by stacking $x_1,\ldots,x_m$. 
In addition, we let $I_n$ and $\mathbf{O}_n$ denote the $n\times n$ identity matrix and zero matrix, respectively. Moreover, $\mathbf{1}_n$ and $\mathbf{0}_n$ are the $n$-dimensional all-one vector and zero vector. In case we omit the subscripts, these matrices and vectors are of proper dimensions by default. For any matrix $A\in\mathbb{R}^{m\times n}$, $\operatorname{Range}(A)$ is the range of $A$ and $\|A\|_2$ is the spectral norm of $A$. If $A=A^T\in \mathbb{R}^{n\times n}$ is positive semidefinite, we define $A^\dag$ as its Moore-Penrose inverse and $\|\mathbf{x}\|_A=\sqrt{\mathbf{x}^TA\mathbf{x}}$ for all $\mathbf{x}\in \mathbb{R}^n$. For any two matrices $A\in \mathbb{R}^{m\times n}$ and $B\in \mathbb{R}^{m'\times n'}$, $\operatorname{diag}(A,B)\in \mathbb{R}^{(m+m')\times (n+n')}$ is the block diagonal matrix with $A$ and $B$ sequentially constituting its diagonal blocks. For any function $\varphi:\mathbb{R}^n\rightarrow\mathbb{R}$, $\partial \varphi(x)$ denotes the subdifferential (i.e., the set of subgradients) of $\varphi$ at $x\in\mathbb{R}^n$, and $\operatorname{dom}(\varphi):=\{x\in\mathbb{R}^n|~\varphi(x)<+\infty\}$ is the domain of $\varphi$. 

% =============================================================================================================================
%
% Problem Reformulation
%
% =============================================================================================================================
\section{Problem Reformulation}\label{sec:probform}

In this section, we reformulate problem~\eqref{eq:proboriginal} in order to facilitate the algorithm design.

\subsection{Constraint Separation}\label{ssec:separation}

In many real-world problems, part of the constraints may have a sparse structure, i.e., some of the $\bar{g}_i^{\ell}$'s and some rows of the $\bar{A}_i$'s may be zero. In such cases, we separate the sparse constraints from the dense ones, and rewrite \eqref{eq:proboriginal} as follows:

\begin{equation}\label{eq:prob}
\begin{split}
\underset{x_i\in \mathbb{R}^{d_i},\;\forall i\in \mathcal{V}}{\operatorname{minimize}}~&\sum_{i\in \mathcal{V}} (f_i(x_i)+h_i(x_i))\\
\operatorname{subject~to}~&\sum_{i\in \mathcal{V}} g_i(x_i) \le \mathbf{0}_{\tilde{p}},\;\;\,\quad\quad\qquad\qquad (\text{2a})\\
&\sum_{i\in \mathcal{V}} A_ix_i=\sum_{i\in\mathcal{V}}b_i,\;\;\;\qquad\qquad\quad (\text{2b})\\
&\sum_{j\in S_i^{\operatorname{in}}} g_{ij}^s(x_j)\le \mathbf{0}_{p_i},~\forall i\in \mathcal{V}^{\operatorname{in}},\quad (\text{2c})\\
&\sum_{j\in S_i^{\operatorname{eq}}} A_{ij}^sx_j= b_i^s,~\forall i\in \mathcal{V}^{\operatorname{eq}},\qquad\!(\text{2d})
\end{split}
\end{equation}
where $\mathcal{V}=\{1,\ldots,n\}$. The constraints (2a) and (2b) represent the densely-coupled constraints with few zero entries, where each $g_i:\mathbb{R}^{d_i}\rightarrow \mathbb{R}^{\tilde{p}}$ is a vector-valued convex function (i.e., each of its entry is a convex function), $A_i\in \mathbb{R}^{\tilde{m}\times d_i}$, and $b\in \mathbb{R}^{\tilde{m}}$. Note that $\tilde{p}$ and $\tilde{m}$ can be $0$, which means the absence of (2a) and (2b). The constraints (2c) and (2d) represent the sparse constraints with an abundance of zero entries. Here, $\mathcal{V}^{\operatorname{in}}$ and $\mathcal{V}^{\operatorname{eq}}$ are subsets of $\mathcal{V}$, which could be empty so that (2c) and (2d) are also allowed to be absent. In addition, $S_i^{\operatorname{in}}$ and $S_i^{\operatorname{eq}}$ are proper subsets of $\mathcal{V}$ (which do not necessarily contain $i$) and often, $|S_i^{\operatorname{in}}|,|S_i^{\operatorname{eq}}|\ll n$. For convenience, we additionally define $S_i^{\operatorname{in}}=\emptyset$ $\forall i\in\mathcal{V}-\mathcal{V}^{\operatorname{in}}$ and $S_i^{\operatorname{eq}}=\emptyset$ $\forall i\in\mathcal{V}-\mathcal{V}^{\operatorname{eq}}$. Moreover, each $g_{ij}^s:\mathbb{R}^{d_j}\rightarrow \mathbb{R}^{p_i}$ is convex, $A_{ij}^s\in \mathbb{R}^{m_i\times d_j}$, and $b_i^s\in \mathbb{R}^{m_i}$. Note that $\tilde{p}+\sum_{i\in\mathcal{V}^{\operatorname{in}}} p_i=\bar{p}$ and $\tilde{m}+\sum_{i\in\mathcal{V}^{\operatorname{eq}}} m_i=\bar{m}$, where $\bar{p}$ and $\bar{m}$ are the dimensions of the inequality and equality constraints in problem~\eqref{eq:proboriginal}, respectively. 

There are multiple ways of casting the constraints in \eqref{eq:proboriginal} into (2a)--(2d), as is illustrated in the following example.

\begin{example}\label{example:densesparse}
\rm
Consider problem~\eqref{eq:proboriginal} with $n=4$, $d_1=d_2=d_3=d_4=1$, and the following constraints:
\begin{align*}
\begin{array}{llllllll}
\bar{g}_1^1(x_1)&+&\bar{g}_2^1(x_2)&+&\bar{g}_3^1(x_3)&+&\bar{g}_4^1(x_4)&\le 0,\\
\bar{g}_1^2(x_1)& & &+&\bar{g}_3^2(x_3)&+&\bar{g}_4^2(x_4)&\le 0,\\
\begin{bmatrix}1\\0\\0\end{bmatrix}\!x_1&+&\begin{bmatrix}2\\1\\3\end{bmatrix}\!x_2&+&\begin{bmatrix}0\\0\\4\end{bmatrix}\!x_3&+&\begin{bmatrix}0\\0\\0\end{bmatrix}\!x_4&=\begin{bmatrix}3\\1\\0\end{bmatrix}.
\end{array}
\end{align*}
Below, we exemplify two different ways of transforming these constraints into (2a)--(2d).

\textbf{Way 1}: Let $g_i(x_i)=\bar{g}_i^1(x_i)$ $\forall i=1,\ldots,4$ in (2a), and let (2b) be absent (i.e., $\tilde{m}=0$). Then, set $\mathcal{V}^{\operatorname{in}}=\{1\}$, $S_1^{\operatorname{in}}=\{1,3,4\}$, and $g_{1j}^s=\bar{g}_j^2$ $\forall j\in S_1^{\operatorname{in}}$ in (2c). For (2d), choose $\mathcal{V}^{\operatorname{eq}}=\{1,2\}$, $S_1^{\operatorname{eq}}=\{1,2\}$, $A_{11}^s=\begin{bmatrix}1\\0\end{bmatrix}$, $A_{12}^s=\begin{bmatrix}2\\1\end{bmatrix}$, $b_1^s=\begin{bmatrix}3\\1\end{bmatrix}$, $S_2^{\operatorname{eq}}=\{2,3\}$, $A_{22}^s=3$, $A_{23}^s=4$, and $b_2^s=0$.

\textbf{Way 2}: Let $g_i(x_i)=\begin{bmatrix}\bar{g}_i^1(x_i)\\ \bar{g}_i^2(x_i)\end{bmatrix}$ $\forall i=1,3,4$ and $g_2(x_2)=\begin{bmatrix}\bar{g}_2^1(x_2)\\ 0\end{bmatrix}$ in (2a), and let both (2b) and (2c) be absent (i.e., $\tilde{m}=0$ and $\mathcal{V}^{\operatorname{in}}=\emptyset$). For (2d), let $\mathcal{V}^{\operatorname{eq}}=\{1,3\}$, $S_1^{\operatorname{eq}}=\{1,2\}$, $A_{11}^s=1$, $A_{12}^s=2$, $b_1^s=3$, $S_3^{\operatorname{eq}}=\{2,3\}$, $A_{32}^s=\begin{bmatrix}1\\3\end{bmatrix}$, $A_{33}^s=\begin{bmatrix}0\\4\end{bmatrix}$, and $b_3^s=\begin{bmatrix}1\\0\end{bmatrix}$.
\end{example}

The above separation of constraints is \emph{not mandatory}. In other words, we can always let $\mathcal{V}^{\operatorname{in}}$ and $\mathcal{V}^{\operatorname{eq}}$ be empty, so that (2a)--(2b) are identical to the constraints in \eqref{eq:proboriginal}. However, we may take advantage of such sparsity to reduce the variable dimension of our proposed algorithm (cf. Remark~\ref{rem:dimension}).

\subsection{Underlying Graph}\label{ssec:graph}

To solve problem~\eqref{eq:prob} in a distributed way, we employ $n$ agents and let them cooperate over an undirected graph $\mathcal{G}$ induced from the coupling structure of (2a)--(2d). 

We let $\mathcal{V}$ be the vertex set of $\mathcal{G}$. Suppose each node (i.e., agent) $i\in\mathcal{V}$ knows $S_i^{\operatorname{in}}$, $S_i^{\operatorname{eq}}$, and whether it belongs to $S_j^{\operatorname{in}}$ and $S_j^{\operatorname{eq}}$ for any other $j\in\mathcal{V}$. Define
\begin{align*}
\mathcal{N}_i^s\!=\!\{j\in\mathcal{V}\!-\!\{i\}|~j\!\in\!S_i^{\operatorname{in}}\cup S_i^{\operatorname{eq}}\text{ or }i\!\in\! S_j^{\operatorname{in}}\cup S_j^{\operatorname{eq}}\},\;\forall i\in\mathcal{V}.
\end{align*}
Clearly, for any $i,j\in\mathcal{V}$, $i\in\mathcal{N}_j^s$ if and only if $j\in\mathcal{N}_i^s$. Also, define $\mathcal{E}^s=\{\{i,j\}|~i\in\mathcal{V},\;j\in\mathcal{N}_i^s\}$.

We then construct the edge set $\mathcal{E}$ of $\mathcal{G}$ according to the following three cases. Without loss of generality, we assume at least one of the constraints (2a)--(2d) exists, because otherwise the problem is completely decoupled and the nodes can independently solve it without any interactions.

\emph{Case 1}: When the sparse constraints (2c)(2d) are absent (i.e., $\mathcal{V}^{\operatorname{in}}=\mathcal{V}^{\operatorname{eq}}=\emptyset$), $\mathcal{G}$ can be \emph{any} undirected, connected graph. 

\emph{Case 2}: When the dense constraints (2a)(2b) are absent (i.e., $\tilde{m}=\tilde{p}=0$), we let $\mathcal{E}=\mathcal{E}^s$. It is possible that $\mathcal{G}=(\mathcal{V},\mathcal{E})$ is unconnected in this case, which, due to the absence of (2a)--(2b), implies that problem~\eqref{eq:proboriginal} can be decoupled into several independent subproblems, and each subproblem corresponds to a connected component in $\mathcal{G}$.% Thus, without loss of generality, we assume $\mathcal{G}$ is connected for this case.

\emph{Case 3}: When the problem contains both sparse and dense constraints (i.e., $\mathcal{V}^{\operatorname{in}}\cup\mathcal{V}^{\operatorname{eq}}\neq\emptyset$ and $\tilde{m}+\tilde{p}>0$), we let $\mathcal{E}=\mathcal{E}^s$ if the graph $(\mathcal{V},\mathcal{E}^s)$ is connected. Otherwise, we obtain $\mathcal{E}$ by artificially adding proper links to $\mathcal{E}^s$ until the resulting graph is connected.

With the undirected graph $\mathcal{G}=(\mathcal{V},\mathcal{E})$ derived from the above three cases, we suppose each node $i\in\mathcal{V}$ can directly communicate with its neighbors in $\mathcal{N}_i:=\{j|~\{i,j\}\!\in\!\mathcal{E}\}\!\supseteq\!\mathcal{N}_i^s$. 

Fig.~\ref{fig:network} illustrates the underlying graph $\mathcal{G}=(\mathcal{V},\mathcal{E})$ induced from the problem in Example~\ref{example:densesparse} applied with two different ways of constraint separation. For Way~1 in Example~\ref{example:densesparse}, $\mathcal{E}^s=\{\{1,3\},\{1,4\},\{1,2\},\{2,3\}\}$. According to the above \emph{Case~3}, as $(\mathcal{V},\mathcal{E}^s)$ is connected, we can directly set $\mathcal{E}=\mathcal{E}^s$. In Way~2, $\mathcal{E}^s=\{\{1,2\},\{2,3\}\}$, and $(\mathcal{V},\mathcal{E}^s)$ is unconnected. Again, from \emph{Case~3}, we artificially add a link, say, $\{1,4\}$, to $\mathcal{E}^s$ and obtain $\mathcal{E}=\{\{1,2\},\{2,3\},\{1,4\}\}$, which guarantees that $(\mathcal{V},\mathcal{E})$ is connected.

\begin{figure}
	\centering 
	\subfigure[Way~1 in Example~\ref{example:densesparse}]{ 
		\label{fig:subfig:a} %% label for first subfigure 
		\includegraphics[width=1.2in]{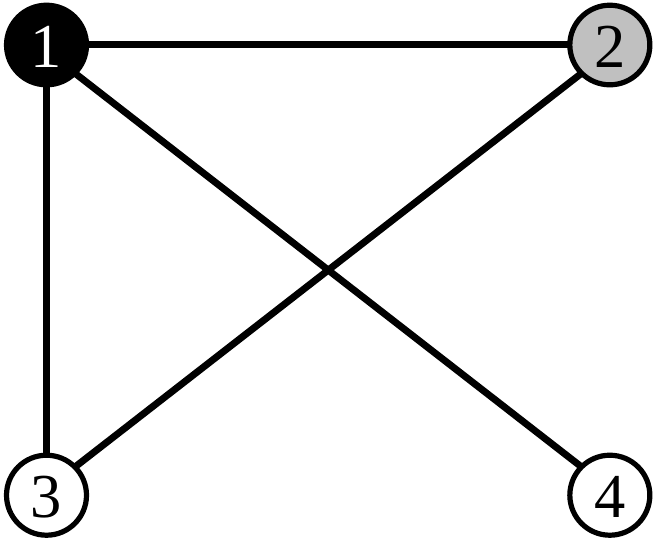}}
	\qquad
	\subfigure[Way~2 in Example~\ref{example:densesparse}]{ 
		\label{fig:subfig:b} %% label for second subfigure 
		\includegraphics[width=1.2in]{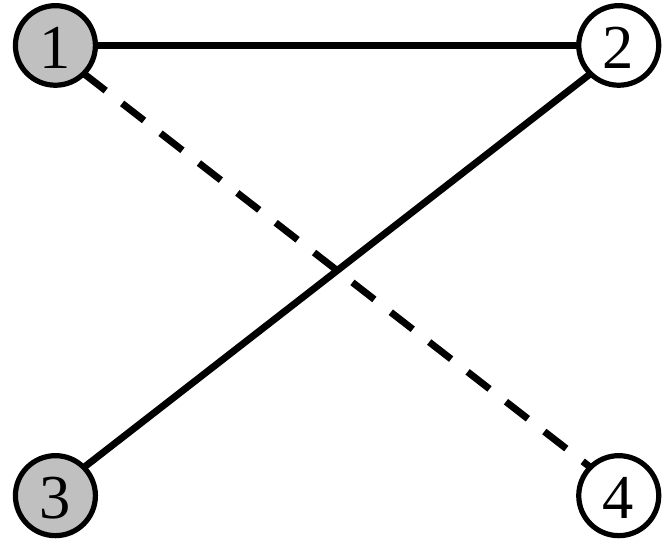}} 
\caption{Underlying interaction graph $\mathcal{G}$ (Black nodes belong to $\mathcal{V}^{\operatorname{in}}\cap\mathcal{V}^{\operatorname{eq}}$, gray nodes are in $\mathcal{V}^{\operatorname{eq}}$ but not in $\mathcal{V}^{\operatorname{in}}$, and white nodes are in neither $\mathcal{V}^{\operatorname{in}}$ nor $\mathcal{V}^{\operatorname{eq}}$. Solid lines represent links in $\mathcal{E}^s$ and dashed lines represent links in $\mathcal{E}-\mathcal{E}^s$.)}
\label{fig:network} %% label for entire figure 
\end{figure}

Next, we distribute the problem data of \eqref{eq:prob} to every node in $\mathcal{G}$. Assigned to each node $i\in\mathcal{V}$ are (i) the local objective functions $f_i:\mathbb{R}^{d_i}\rightarrow \mathbb{R}$ and $h_i:\mathbb{R}^{d_i}\rightarrow \mathbb{R}\cup\{+\infty\}$, (ii) the constraint functions $g_i:\mathbb{R}^{d_i}\rightarrow \mathbb{R}^{\tilde{p}}$ and $g_{ji}^s:\mathbb{R}^{d_i}\rightarrow \mathbb{R}^{p_j}$ for all $j\in\mathcal{V}^{\operatorname{in}}$ such that $i\in S_j^{\operatorname{in}}$, and (iii) the matrices $A_i\in \mathbb{R}^{\tilde{m}\times d_i}$ and $A_{ji}^s\in \mathbb{R}^{m_j\times d_i}$ for all $j\in\mathcal{V}^{\operatorname{eq}}$ such that $i\in S_j^{\operatorname{eq}}$, as well as the vector $b_i\in \mathbb{R}^{\tilde{m}}$. Moreover, each node $i\in\mathcal{V}^{\operatorname{eq}}$ is additionally associated with the vector $b_i^s\in \mathbb{R}^{m_i}$.

Our goal is to design a distributed algorithm for solving problem~\eqref{eq:prob} over the underlying graph $\mathcal{G}$, where each node only has access to its own problem data and only interacts with its neighbors, so that problem~\eqref{eq:prob} can be solved in an efficient and scalable way. 

\subsection{Equivalent Problem}\label{ssec:probtrans}

This subsection transforms \eqref{eq:prob} into an equivalent form.

We first introduce a set of auxiliary variables $t_i\in \mathbb{R}^{\tilde{p}}$ $\forall i\in \mathcal{V}$, and equivalently convert the constraint (2a) to
\begin{align}
&g_i(x_i)\le t_i,\quad\forall i\in\mathcal{V},\label{eq:g<=t}\displaybreak[0]\\
&\sum_{i\in \mathcal{V}} t_i=\mathbf{0}_{\tilde{p}}.\label{eq:sumt=0}
\end{align} 
By doing so, we decouple the dense nonlinear inequality constraint (2a) into $n$ independent nonlinear constraints in \eqref{eq:g<=t} at the cost of additionally imposing the linear equality constraint \eqref{eq:sumt=0}, which, intuitively, may be easier to tackle. Then, we attempt to rewrite \eqref{eq:prob} in terms of the new decision variables
\begin{align}
y_i=[x_i^T, t_i^T]^T\in\mathbb{R}^{d_i+\tilde{p}},\quad\forall i\in\mathcal{V}.\label{eq:yi=xiti}
\end{align}

For the sake of compactness, let $\mathbf{y}=[y_1^T, \ldots, y_n^T]^T\in \mathbb{R}^N$, where $N=n\tilde{p}+\sum_{i\in \mathcal{V}} d_i$. We thus express the objective function in \eqref{eq:prob} by $\Phi(\mathbf{y}):=f(\mathbf{y})+h(\mathbf{y})$, with
\begin{align*}
f(\mathbf{y})=\sum_{i\in \mathcal{V}} f_i(x_i), \quad h(\mathbf{y})=\sum_{i\in \mathcal{V}} h_i(x_i).
\end{align*}

Next, we group together all the inequality constraints including \eqref{eq:g<=t} and the sparse inequality constraint (2c) by defining
\begin{equation*}
G_i(\mathbf{y})=\begin{cases}
\begin{bmatrix}
g_i(x_i)-t_i\\
\sum_{j\in S_i^{\operatorname{in}}} g_{ij}^s(x_j)
\end{bmatrix}, & \text{if }i\in\mathcal{V}^{\operatorname{in}},\\
g_i(x_i)-t_i, & \text{otherwise,}
\end{cases}
\end{equation*}
for each $i\in\mathcal{V}$. Observe that
\begin{equation*}
\mathbf{G}(\mathbf{y}):=\begin{bmatrix}
G_1(\mathbf{y})\\
\vdots\\
G_n(\mathbf{y})
\end{bmatrix}\le \mathbf{0}_{n\tilde{p}+\underset{i\in\mathcal{V}^{\operatorname{in}}}{\sum} p_i}
\end{equation*}
is identical to \eqref{eq:g<=t} and (2c).

Subsequently, we express \eqref{eq:sumt=0} and the dense equality constraint (2b) in a compact form. To do so, we introduce the following notations: For each $i\in\mathcal{V}$, let $B_i=\operatorname{diag}(A_i,I_{\tilde{p}})\in\mathbb{R}^{(\tilde{m}+\tilde{p})\times(d_i+\tilde{p})}$ and $c_i=[b_i^T,\mathbf{0}_{\tilde{p}}^T]^T\in\mathbb{R}^{\tilde{m}+\tilde{p}}$. Then, let $\mathbf{B}=\operatorname{diag}(B_1,\ldots,B_n)\in\mathbb{R}^{n(\tilde{m}+\tilde{p})\times N}$ and $\mathbf{c}=[c_1^T,\ldots,c_n^T]^T\in\mathbb{R}^{n(\tilde{m}+\tilde{p})}$. Note that
\begin{align}
(\mathbf{1}_n\!\otimes\! I_{\tilde{m}+\tilde{p}})^T(\mathbf{B}\mathbf{y}\!-\!\mathbf{c})\!=\!\sum_{i\in\mathcal{V}}(B_iy_i\!-\!c_i)\!=\!\begin{bmatrix}\sum\limits_{i\in\mathcal{V}}(A_ix_i\!-\!b_i)\\ \sum_{i\in\mathcal{V}}t_i\end{bmatrix}.\nonumber
\end{align}
Hence, $(\mathbf{1}_n\!\otimes\! I_{\tilde{m}+\tilde{p}})^T(\mathbf{B}\mathbf{y}\!-\!\mathbf{c})\!=\!\mathbf{0}_{\tilde{m}+\tilde{p}}$ is exactly \eqref{eq:sumt=0} and (2b).

It remains to describe the sparse equality constraint (2d) in terms of $\mathbf{y}$. To this end, define $B_{ij}^s\in\mathbb{R}^{m_i\times (d_j+\tilde{p})}$ $\forall i\in\mathcal{V}^{\operatorname{eq}}$ $\forall j\in\mathcal{V}$ as follows:
\begin{equation*}
B_{ij}^s=
\begin{cases}
[A_{ij}^s,\mathbf{O}_{m_i\times\tilde{p}}], & \text{if }j\in S_i^{\operatorname{eq}},\\
\mathbf{O}_{m_i\times (d_j\!+\tilde{p})}, & \text{otherwise.}
\end{cases}
\end{equation*}
Then, let $B_i^s=\begin{bmatrix}B_{i1}^s, \ldots, B_{in}^s\end{bmatrix}\in\mathbb{R}^{m_i\times N}$ $\forall i\in\mathcal{V}^{\operatorname{eq}}$. Note that for each $i\in\mathcal{V}^{\operatorname{eq}}$, $B_{ij}^sy_j$ is equal to $A_{ij}^sx_j$ if $j\in S_i^{\operatorname{eq}}$ and is equal to $\mathbf{0}_{m_i}$ otherwise. Hence, $B_i^s\mathbf{y}-b_i^s=(\sum_{j\in\mathcal{V}}B_{ij}^sy_j)-b_i^s=(\sum_{j\in S_i^{\operatorname{eq}}} A_{ij}^sx_j)-b_i^s$. It follows that (2d) can be written as $\mathbf{B}^s\mathbf{y}=\mathbf{c}^s$, where $\mathbf{B}^s\in\mathbb{R}^{(\sum_{i\in\mathcal{V}^{\operatorname{eq}}}m_i)\times N}$ and $\mathbf{c}^s\in\mathbb{R}^{\sum_{i\in\mathcal{V}^{\operatorname{eq}}}m_i}$ are given by
\begin{equation*}
\mathbf{B}^s=\begin{bmatrix}
B_{\zeta_1}^s\\ \vdots \\ B_{\zeta_a}^s\end{bmatrix},\quad
\mathbf{c}^s=\begin{bmatrix}
b_{\zeta_1}^s\\ \vdots \\ b_{\zeta_a}^s
\end{bmatrix},\quad \{\zeta_1,\ldots,\zeta_a\}=\mathcal{V}^{\operatorname{eq}}.
\end{equation*}
In other words, $\mathbf{B}^s$ and $\mathbf{c}^s$ are obtained by vertically stacking $B_i^s$ and $b_i^s$ $\forall i\in\mathcal{V}^{\operatorname{eq}}$, respectively.	

Combining the above, we derive the following equivalent form of problem~\eqref{eq:prob}:
\begin{equation}\label{eq:originalprobleminy}
\begin{split}
\underset{\mathbf{y}\in \mathbb{R}^N}{\operatorname{minimize}} ~~&~ \Phi(\mathbf{y})=f(\mathbf{y})+h(\mathbf{y})\\
\operatorname{subject~to}~&~\mathbf{G}(\mathbf{y})\le\mathbf{0}_{n\tilde{p}+\underset{{i\in\mathcal{V}^{\operatorname{in}}}}{\sum} p_i},\\
&~(\mathbf{1}_n\otimes I_{\tilde{m}+\tilde{p}})^T(\mathbf{B}\mathbf{y}-\mathbf{c}) = \mathbf{0}_{\tilde{m}+\tilde{p}},\\
&~\mathbf{B}^s\mathbf{y}=\mathbf{c}^s.
\end{split}
\end{equation}
% =============================================================================================================================
%
% Algorithm
%
% =============================================================================================================================
\section{Algorithm}\label{sec:algdevelop}

In this section, we propose a novel distributed algorithm for solving problem~\eqref{eq:originalprobleminy} that is equivalent to \eqref{eq:prob} and \eqref{eq:proboriginal}.

We adopt a separation principle in our algorithm design. Specifically, we develop two new primal-dual proximal methods and utilize a prior virtual-queue-based method to minimize $\Phi(\mathbf{y})$ subject to each of the three constraints in \eqref{eq:originalprobleminy}, respectively. Then, we deliberately unite these three methods, so that \eqref{eq:originalprobleminy} with all the three constraints can be solved.

\subsection{Proximal Algorithm for Sparse Equality Constraints}\label{ssec:proximalsparse}

We tentatively focus on the sparse linear equality constraint $\mathbf{B}^s\mathbf{y}=\mathbf{c}^s$ and consider the following special case of \eqref{eq:originalprobleminy}:
\begin{equation}\label{eq:sparseeqprob}
	\begin{split}
		\underset{\mathbf{y}\in\mathbb{R}^N}{\operatorname{minimize}}~~&~\Phi(\mathbf{y})\\
		\operatorname{subject~to} ~&~\mathbf{B}^s\mathbf{y}=\mathbf{c}^s.
	\end{split}
\end{equation}

We endeavor to solve problem~\eqref{eq:sparseeqprob} in a distributed fashion. To this end, we try to apply the Method of Multipliers \cite{Boyd11} to \eqref{eq:sparseeqprob}, which has the following primal-dual dynamics:
\begin{align}
\mathbf{y}(k+1) \in& \operatorname{\arg\;\min}_{\mathbf{y}\in \mathbb{R}^N} \Phi(\mathbf{y})+\langle \tilde{\mathbf{v}}(k), \mathbf{B}^s\mathbf{y}\rangle\nonumber\displaybreak[0]\\
&+\frac{\gamma}{2}\|\mathbf{B}^s\mathbf{y}-\mathbf{c}^s\|^2,\label{eq:MMprimal}\\
\tilde{\mathbf{v}}(k+1) =& \tilde{\mathbf{v}}(k)+\gamma(\mathbf{B}^s\mathbf{y}(k+1)-\mathbf{c}^s),\quad\forall k\ge0.\label{eq:MMdual}
\end{align}
Here, $\tilde{\mathbf{v}}(0)\in\mathbb{R}^{\sum_{i\in\mathcal{V}^{\operatorname{eq}}}m_i}$ is arbitrarily chosen and $\gamma>0$. 

Although \eqref{eq:sparseeqprob} has a sparse structure, the primal update \eqref{eq:MMprimal} in the Method of Multipliers is still a \emph{centralized} operation since $\|\mathbf{B}^s\mathbf{y}-\mathbf{c}^s\|^2$ mingles the coordinates of the optimization variable $\mathbf{y}$. To decentralize the primal update, we add a proximal term $\frac{\gamma}{2}(\mathbf{y}-\mathbf{y}(k))^T(\lambda^2 I_N-(\mathbf{B}^s)^T\mathbf{B}^s)(\mathbf{y}-\mathbf{y}(k))$, $\lambda>0$ to the minimized function on the right-hand side of \eqref{eq:MMprimal}, where $\mathbf{y}(0)$ is arbitrarily given. It can be shown that
\begin{align*}
&\|\mathbf{B}^s\mathbf{y}-\mathbf{c}^s\|^2+(\mathbf{y}-\mathbf{y}(k))^T(\lambda^2I_N-(\mathbf{B}^s)^T\mathbf{B}^s)(\mathbf{y}-\mathbf{y}(k))\displaybreak[0]\\
&=\lambda^2\|\mathbf{y}-\mathbf{y}(k)+\frac{1}{\lambda^2}(\mathbf{B}^s)^T(\mathbf{B}^s\mathbf{y}(k)-\mathbf{c}^s)\|^2+C(k),
\end{align*}
where $C(k)=\|\mathbf{B}^s\mathbf{y}(k)-\mathbf{c}^s\|^2-\frac{1}{\lambda^2}\|(\mathbf{B}^s)^T(\mathbf{B}^s\mathbf{y}(k)-\mathbf{c}^s)\|^2$ is independent of $\mathbf{y}$. As a result, the primal update \eqref{eq:MMprimal} is changed to the following:
\begin{equation}\label{eq:sparsexupdate}
\begin{split}
&\mathbf{y}(k+1)=\operatorname{\arg\;\min}_{\mathbf{y}\in \mathbb{R}^N}~\Phi(\mathbf{y})+\langle\tilde{\mathbf{v}}(k),\mathbf{B}^s\mathbf{y}\rangle\\
&+\frac{\gamma\lambda^2}{2}\|\mathbf{y}-\mathbf{y}(k)+\frac{1}{\lambda^2}(\mathbf{B}^s)^T(\mathbf{B}^s\mathbf{y}(k)-\mathbf{c}^s)\|^2.
\end{split}
\end{equation}

To reduce the computational cost of \eqref{eq:sparsexupdate} and \eqref{eq:MMdual}, we additionally apply the linear transformation $\mathbf{v}(k)=(\mathbf{B}^s)^T\tilde{\mathbf{v}}(k)\in\mathbb{R}^N$ $\forall k\ge 0$, which results in
\begin{align}
&\mathbf{y}(k+1) = \operatorname{\arg\;\min}_{\mathbf{y}\in \mathbb{R}^N}~\Phi(\mathbf{y})+\Delta_1^k(\mathbf{y}),\label{eq:sparseyupdatetildev}\\
&\mathbf{v}(k+1)=\mathbf{v}(k)+\gamma(\mathbf{B}^s)^T(\mathbf{B}^s\mathbf{y}(k+1)-\mathbf{c}^s),\label{eq:sparsetildevupdate}
\end{align}
where $\Delta_1^k(\mathbf{y}):=\langle \mathbf{v}(k), \mathbf{y}\rangle+\frac{\gamma\lambda^2}{2}\|\mathbf{y}-\mathbf{y}(k)+\frac{1}{\lambda^2}(\mathbf{B}^s)^T(\mathbf{B}^s\mathbf{y}(k)-\mathbf{c}^s)\|^2$. Note that the computations involving $\mathbf{B}^s$ in \eqref{eq:sparsexupdate} and \eqref{eq:MMdual} include calculating $(\mathbf{B}^s)^T\tilde{\mathbf{v}}(k)$, $(\mathbf{B}^s)^T(\mathbf{B}^s\mathbf{y}(k)-\mathbf{c}^s)$, and $\mathbf{B}^s\mathbf{y}(k+1)$, yet \eqref{eq:sparseyupdatetildev} and \eqref{eq:sparsetildevupdate} only need to calculate $(\mathbf{B}^s)^T(\mathbf{B}^s\mathbf{y}(k+1)-\mathbf{c}^s)$, as $(\mathbf{B}^s)^T(\mathbf{B}^s\mathbf{y}(k)-\mathbf{c}^s)$ can be inherited from the last iteration. 

Due to the above linear transformation, we require $\mathbf{v}(k)\in\operatorname{Range}((\mathbf{B}^s)^T)$ for each $k\ge0$. This can be guaranteed by choosing $\mathbf{v}(0)\in \operatorname{Range}((\mathbf{B}^s)^T)$, because \eqref{eq:sparsetildevupdate} indicates that $\mathbf{v}(k+1)-\mathbf{v}(k)\in\operatorname{Range}((\mathbf{B}^s)^T)$ $\forall k\ge0$. 
	
The updates \eqref{eq:sparseyupdatetildev}--\eqref{eq:sparsetildevupdate} as well as the initialization of arbitrary $\mathbf{y}(0)\in\mathbb{R}^N$ and $\mathbf{v}(0)\in \operatorname{Range}((\mathbf{B}^s)^T)$ constitute a new primal-dual proximal algorithm for solving problems with sparse linear equality constraints in the form of \eqref{eq:sparseeqprob}. It can be viewed as a proximal variation of the Method of Multipliers, and can be executed in a decentralized manner, which will be discussed in Section~\ref{ssec:distributed}.

\subsection{Proximal Algorithm for Dense Equality Constraints}\label{ssec:proximaldense}

We switch to problem~\eqref{eq:originalprobleminy} with only the densely-coupled linear equality constraint, i.e., 
\begin{equation}\label{eq:equalityprob}
\begin{split}
\underset{\mathbf{y}\in\mathbb{R}^N}{\operatorname{minimize}}~~&~ \Phi(\mathbf{y})\\
\operatorname{subject~to} ~&~ (\mathbf{1}_n\otimes I_{\tilde{m}+\tilde{p}})^T(\mathbf{B}\mathbf{y}-\mathbf{c}) =\mathbf{0}_{\tilde{m}+\tilde{p}}.
\end{split}
\end{equation}

Instead of directly addressing \eqref{eq:equalityprob}, we tackle the Lagrange dual problem of \eqref{eq:equalityprob}, which is given by
\begin{equation}\label{eq:equalitydualprobnoncopyvariable}
\underset{u\in \mathbb{R}^{\tilde{m}+\tilde{p}}}{\operatorname{maximize}}\min_{\mathbf{y}\in\mathbb{R}^N} \Phi(\mathbf{y})+\langle u, (\mathbf{1}_n\otimes I_{\tilde{m}+\tilde{p}})^T(\mathbf{B}\mathbf{y}-\mathbf{c})\rangle.
\end{equation}
For each $i\in\mathcal{V}$, let $u_i\in \mathbb{R}^{\tilde{m}+\tilde{p}}$ be a local copy of the dual variable $u\in\mathbb{R}^{\tilde{m}+\tilde{p}}$ in \eqref{eq:equalitydualprobnoncopyvariable}, and let $\mathbf{u}=[u_1^T, \ldots, u_n^T]^T\in\mathbb{R}^{n(\tilde{m}+\tilde{p})}$. Thus, the dual problem~\eqref{eq:equalitydualprobnoncopyvariable} can be rewritten as
\begin{equation}\label{eq:equalitydualproblem}
\begin{split}
\underset{\mathbf{u}\in \mathbb{R}^{n(\tilde{m}+\tilde{p})}}{\operatorname{maximize}}~~&~ D_e(\mathbf{u}):=\min_{\mathbf{y}\in \mathbb{R}^N} \Phi(\mathbf{y})+\langle\mathbf{u},\mathbf{B}\mathbf{y}-\mathbf{c}\rangle\\
\operatorname{subject~to} ~&~ u_1 = \ldots = u_n.
\end{split}
\end{equation}

To solve the equivalent dual problem \eqref{eq:equalitydualproblem}, we develop another proximal algorithm based on the P-EXTRA algorithm \cite{ShiW15a}. P-EXTRA is a distributed first-order method for nonsmooth convex consensus optimization. According to \cite{Wu19a}, when we apply P-EXTRA to solve \eqref{eq:equalitydualproblem}, it can be expressed in the following form:
\begin{align*}
\mathbf{u}(k+1)=& \operatorname{\arg\;\min}_{\mathbf{u}\in\mathbb{R}^{n(\tilde{m}+\tilde{p})}} -D_e(\mathbf{u})+\langle\mathbf{u}, H^{\frac{1}{2}}\tilde{\mathbf{z}}(k)\rangle\displaybreak[0]\\
&+\frac{\rho}{2}\|\mathbf{u}-W\mathbf{u}(k)\|^2,\\
\tilde{\mathbf{z}}(k+1) =& \tilde{\mathbf{z}}(k)+\rho H^{\frac{1}{2}}\mathbf{u}(k+1).
\end{align*}
Here, $\mathbf{u}(0)\in\mathbb{R}^{n(\tilde{m}+\tilde{p})}$ is arbitrarily selected and $\tilde{\mathbf{z}}(0)=\rho H^{\frac{1}{2}}\mathbf{u}(0)$. In addition, $\rho>0$, $W=P^W\otimes I_{\tilde{m}+\tilde{p}}$, and $H=P^H\otimes I_{\tilde{m}+\tilde{p}}$, where $P^W,P^H\in \mathbb{R}^{n\times n}$ satisfy the assumption below.

\begin{assumption}\label{asm:weightmatrix}
The matrices $P^W$, $P^H$ satisfy the following:
\begin{enumerate}[(a)]
\item $[P^W]_{ij}=[P^H]_{ij}=0$, $\forall i\in \mathcal{V}$, $\forall j\notin \mathcal{N}_i\cup\{i\}$.
\item $P^W$ and $P^H$ are symmetric and positive semidefinite.
\item $P^W\mathbf{1}_n=\mathbf{1}_n$, $\operatorname{Null}(P^H)=\operatorname{span}(\mathbf{1}_n)$.
\item $P^W+P^H\preceq I_n$.
\end{enumerate}
\end{assumption}

Assumption \ref{asm:weightmatrix}(a) is used to enable distributed implementation and Assumption \ref{asm:weightmatrix}(b)--(d) are for ensuring the convergence of P-EXTRA. The connectivity of the graph $\mathcal{G}$ makes Assumption~\ref{asm:weightmatrix} easily hold. For example, we may let $P^W=\frac{I_n+P'}{2}$ and $P^H=\frac{I_n-P'}{2}$ for some $P'\in \mathbb{R}^{n\times n}$ such that $[P']_{ij}=[P']_{ji}>0$ $\forall \{i,j\}\in \mathcal{E}$, $[P']_{ii}=1-\sum_{j\in \mathcal{N}_i} [P']_{ij}>0$ $\forall i\in \mathcal{V}$, and $[P']_{ij}=0$ otherwise. Assumption~\ref{asm:weightmatrix} also yields the following properties of $W$ and $H$:
\begin{align}
&W=W^T\succeq \mathbf{O},\;H=H^T\succeq \mathbf{O},\label{eq:WHsymmetricity}\displaybreak[0]\\
&W(\mathbf{1}_n\otimes I_{\tilde{m}+\tilde{p}})=\mathbf{1}_n\otimes I_{\tilde{m}+\tilde{p}},\label{eq:W1is1}\displaybreak[0]\\
&H(\mathbf{1}_n\otimes I_{\tilde{m}+\tilde{p}})=\mathbf{O},\label{eq:H1=0}\displaybreak[0]\\
&\operatorname{Range}(H)\!=\!\{\mathbf{y}\in \mathbb{R}^{n(\tilde{m}+\tilde{p})}|(\mathbf{1}_n\otimes I_{\tilde{m}+\tilde{p}})^T\mathbf{y}\!=\!\mathbf{0}\},\label{eq:rangeH}\displaybreak[0]\\
&W+H\preceq I_{n(\tilde{m}+\tilde{p})}.\label{eq:WHsumsmallerthanid}
\end{align}

For the purpose of distributed design later, we further set $\mathbf{z}(k)=H^{\frac{1}{2}}\tilde{\mathbf{z}}(k)$ $\forall k\ge 0$, leading to
\begin{align}
\mathbf{u}(k+1)=& \operatorname{\arg\;\min}_{\mathbf{u}\in \mathbb{R}^{n(\tilde{m}+\tilde{p})}} -D_e(\mathbf{u})+\langle\mathbf{u}, \mathbf{z}(k)\rangle\nonumber\displaybreak[0]\\
&+\frac{\rho}{2}\|\mathbf{u}-W\mathbf{u}(k)\|^2,\label{eq:pextrauupdate}\\
\mathbf{z}(k+1) =& \mathbf{z}(k)+\rho H\mathbf{u}(k+1),\label{eq:pextrazupdate}
\end{align}
where $\mathbf{z}(0)=\rho H\mathbf{u}(0)$. 

Nevertheless, neither P-EXTRA nor its equivalent form \eqref{eq:pextrauupdate}--\eqref{eq:pextrazupdate} can be applied to solve \eqref{eq:equalitydualproblem} because the dual function $D_e$ generally cannot be expressed in closed form. To overcome this, we propose an implementable approximation of P-EXTRA for solving \eqref{eq:equalitydualproblem}. 

First of all, with the same initialization and parameter settings as P-EXTRA described above, for each $k\ge0$, let
\begin{align}
\mathbf{y}(k+1)\in&\operatorname{\arg\;\min}_{\mathbf{y}\in\mathbb{R}^N} \Phi(\mathbf{y})+\frac{1}{2\rho}\|\mathbf{B}\mathbf{y}-\mathbf{c}\|^2\nonumber\displaybreak[0]\\
&+\langle W\mathbf{u}(k)-\frac{1}{\rho}\mathbf{z}(k), \mathbf{B}\mathbf{y}-\mathbf{c}\rangle,\label{eq:originalyintheeqalg}\\
\mathbf{u}(k+1)=&W\mathbf{u}(k)+\frac{1}{\rho}(\mathbf{B}\mathbf{y}(k+1)-\mathbf{c}-\mathbf{z}(k)),\label{eq:relationofdualandprimal}
\end{align}
and $\mathbf{z}(k)$ be updated as \eqref{eq:pextrazupdate}. In fact, the method described by \eqref{eq:originalyintheeqalg}, \eqref{eq:relationofdualandprimal}, and \eqref{eq:pextrazupdate} is equivalent to P-EXTRA described by \eqref{eq:pextrauupdate}--\eqref{eq:pextrazupdate}, provided that $\mathbf{y}(k+1)$ in \eqref{eq:originalyintheeqalg} is well-defined. 

Below, we verify such an equivalent relation. To do so, we tentatively assume $\mathbf{y}(k+1)$ in \eqref{eq:originalyintheeqalg} exists, and will address this issue shortly. Note that \eqref{eq:originalyintheeqalg} is identical to
\begin{equation*}
	-\mathbf{B}^T(W\mathbf{u}(k)+\frac{1}{\rho}(\mathbf{B}\mathbf{y}(k+1)-\mathbf{c}-\mathbf{z}(k)))\in \partial \Phi(\mathbf{y}(k+1)).
\end{equation*}
Then, substituting \eqref{eq:relationofdualandprimal} into the above gives
\begin{equation*}
	-\mathbf{B}^T\mathbf{u}(k+1) \in \partial \Phi(\mathbf{y}(k+1)),
\end{equation*}
which indicates
\begin{equation}\label{eq:yk1isminimizer}
\mathbf{y}(k+1)\in \operatorname{\arg\;\min}_{\mathbf{y}\in \mathbb{R}^{N}} \Phi(\mathbf{y})+\langle \mathbf{u}(k+1), \mathbf{B}\mathbf{y}-\mathbf{c}\rangle.
\end{equation}
In addition, \eqref{eq:yk1isminimizer} implies $\mathbf{B}\mathbf{y}(k+1)-\mathbf{c}\in\partial D_e(\mathbf{u}(k+1))$ \cite{Bertsekas99}. On the other hand, \eqref{eq:relationofdualandprimal} results in $\mathbf{B}\mathbf{y}(k+1)-\mathbf{c} = \mathbf{z}(k)+\rho(\mathbf{u}(k+1)-W\mathbf{u}(k))$. Consequently, 
\begin{align*}
\mathbf{z}(k)+\rho(\mathbf{u}(k+1)-W\mathbf{u}(k))\in \partial D_e(\mathbf{u}(k+1)), 
\end{align*}
which is equivalent to \eqref{eq:pextrauupdate}. Therefore, $(\mathbf{z}(k),\mathbf{u}(k))$ generated by \eqref{eq:originalyintheeqalg}, \eqref{eq:relationofdualandprimal}, and \eqref{eq:pextrazupdate} is exactly the same as $(\mathbf{z}(k),\mathbf{u}(k))$ generated by \eqref{eq:pextrauupdate}--\eqref{eq:pextrazupdate}, yet \eqref{eq:originalyintheeqalg}, \eqref{eq:relationofdualandprimal}, and \eqref{eq:pextrazupdate} do not rely on $D_e$ and thus do not require that $D_e$ be in closed form. 

Apart from that, the additional variable $\mathbf{y}(k)$ updated as in \eqref{eq:originalyintheeqalg} is more than an auxiliary variable. It can also be used as an online approximate primal solution to problem~\eqref{eq:equalityprob}, while \eqref{eq:pextrauupdate}--\eqref{eq:pextrazupdate} cannot provide such a real-time primal approximation. To see this, note that for any optimum $\mathbf{u}^\star$ of \eqref{eq:equalitydualproblem}, $\mathbf{y}^\star\in \operatorname{\arg\;\min}_{\mathbf{y}} \Phi(\mathbf{y})+\langle\mathbf{u}^\star, \mathbf{B}\mathbf{y}-\mathbf{c}\rangle$ is an optimum of \eqref{eq:equalityprob}. Thus, from \eqref{eq:yk1isminimizer}, $\mathbf{y}(k)$ is able to converge to a primal optimum of \eqref{eq:equalityprob}, as long as $\mathbf{u}(k)$ converges to a dual optimum. Since $\mathbf{u}(k)$ results from P-EXTRA, it is guaranteed to converge to an optimum of \eqref{eq:equalitydualproblem} under proper conditions \cite{ShiW15a}.

It remains to address the issue that \eqref{eq:originalyintheeqalg} may not be well-posed as the convex optimization problem on its right-hand side may not have an optimal solution. To this end, we modify \eqref{eq:originalyintheeqalg} by adding a proximal term $\frac{\alpha_1}{2}\|\mathbf{y}-\mathbf{y}(k)\|^2$, $\alpha_1>0$ to its right-hand side, which gives
\begin{equation}\label{eq:yintheeqalg}
\mathbf{y}(k+1)=\operatorname{\arg\;\min}_{\mathbf{y}\in\mathbb{R}^N} \Phi(\mathbf{y})+\Delta_2^k(\mathbf{y}),
\end{equation}
where $\mathbf{y}(0)$ can be arbitrary and $\Delta_2^k(\mathbf{y}):=\frac{1}{2\rho}\|\mathbf{B}\mathbf{y}-\mathbf{c}\|^2+\langle W\mathbf{u}(k)-\frac{1}{\rho}\mathbf{z}(k), \mathbf{B}\mathbf{y}-\mathbf{c}\rangle+\frac{\alpha_1}{2}\|\mathbf{y}-\mathbf{y}(k)\|^2$. This ensures the unique existence of $\mathbf{y}(k+1)$.

The above proximal algorithm originating from P-EXTRA for solving the dual problem of \eqref{eq:equalityprob} can be summarized as follows: Given any $\mathbf{u}(0)\in\mathbb{R}^{n(\tilde{m}+\tilde{p})}$ and $\mathbf{y}(0)\in\mathbb{R}^N$, we set $\mathbf{z}(0)=\rho H\mathbf{u}(0)$ and then successively update $\mathbf{y}(k+1)$, $\mathbf{u}(k+1)$, and $\mathbf{z}(k+1)$ for each $k\ge0$ according to \eqref{eq:yintheeqalg}, \eqref{eq:relationofdualandprimal}, and \eqref{eq:pextrazupdate}, respectively. The entire process can be executed in a fully distributed way over the graph $\mathcal{G}$, which will be discussed in Section~\ref{ssec:distributed}.

\subsection{Virtual-Queue Algorithm for Sparse Inequality Constraints}\label{ssec:virtualqueue}

Now we only keep the inequality constraint for problem~\eqref{eq:originalprobleminy}:
\begin{equation}\label{eq:inequalityprobiny}
\begin{split}
\underset{\mathbf{y}\in\mathbb{R}^N}{\operatorname{minimize}}~~&~ \Phi(\mathbf{y})\\
\operatorname{subject~to} ~&~ \mathbf{G}(\mathbf{y})\le\mathbf{0}_{n\tilde{p}+\underset{{i\in\mathcal{V}^{\operatorname{in}}}}{\sum} p_i}.
\end{split}
\end{equation}
We may utilize the virtual-queue-based algorithm \cite{YuH17} to solve \eqref{eq:inequalityprobiny}. The virtual-queue-based algorithm can be viewed as a modification of the dual subgradient algorithm, with an additional proximal term in the primal update and an approximate dual update that resembles a queuing equation. It takes the following form when applied to problem~\eqref{eq:inequalityprobiny}: Arbitrarily pick $\mathbf{y}(0)\in\operatorname{dom}(h)$ and set $\mathbf{q}(0) = \max\{-\mathbf{G}(\mathbf{y}(0)),\mathbf{0}_{n\tilde{p}+\sum_{{i\in\mathcal{V}^{\operatorname{in}}}} p_i}\}$. Then, for any $k\ge 0$,
\begin{align}
\!\!\mathbf{y}(k\!+\!1) \!=&\! \operatorname{\arg\;\min}_{\mathbf{y}\in \mathbb{R}^N} \Phi(\mathbf{y})+\Delta_3^k(\mathbf{y}),\label{eq:virtualqueueiny}\\
\!\!\mathbf{q}(k\!+\!1) \!=& \max\{-\mathbf{G}(\mathbf{y}(k+1)), \mathbf{q}(k)\!+\!\mathbf{G}(\mathbf{y}(k+1))\},\label{eq:virtualqueueq}
\end{align}
where $\Delta_3^k(\mathbf{y}):=\langle \mathbf{q}(k)+\mathbf{G}(\mathbf{y}(k)),\mathbf{G}(\mathbf{y})\rangle+\frac{\alpha_2}{2}\|\mathbf{y}-\mathbf{y}(k)\|^2$, $\alpha_2>0$. 

Recall from Section~\ref{ssec:probtrans} that the inequality constraint in \eqref{eq:inequalityprobiny} is sparse, which would enable the distributed implementation of the above updates (cf. Section~\ref{ssec:distributed}). Indeed, although the virtual-queue-based algorithm is targeted at inequality-constrained convex problems, it can also deal with equality constraints by means of transforming each equality constraint into two inequalities. However, this would significantly increase the number of constraints and thus the algorithm complexity. Moreover, the virtual-queue-based algorithm would rely on centralized coordination when it comes to the dense constraints in problem~\eqref{eq:originalprobleminy}. 

\subsection{Integrated Primal-Dual Proximal Algorithm}\label{ssec:alg}

The three constraints in problem~\eqref{eq:originalprobleminy} are considered separately so far---We have proposed two new proximal algorithms in Sections~\ref{ssec:proximalsparse} and~\ref{ssec:proximaldense}, and presented the virtual-queue-based algorithm \cite{YuH17} in Section~\ref{ssec:virtualqueue}. Each of them is primal-dual by nature and copes with only one constraint in \eqref{eq:originalprobleminy}. 

To unite these three primal-dual algorithms in order to solve \eqref{eq:originalprobleminy}, we keep their variables $\mathbf{v}(k)$, $\mathbf{u}(k)$, $\mathbf{z}(k)$, and $\mathbf{q}(k)$, which are updated via \eqref{eq:sparsetildevupdate}, \eqref{eq:relationofdualandprimal}, \eqref{eq:pextrazupdate}, and \eqref{eq:virtualqueueq}, respectively. Their initializations are also the same as those given in Sections~\ref{ssec:proximalsparse}, \ref{ssec:proximaldense}, and \ref{ssec:virtualqueue}. Then, note that these algorithms have a common primal variable $\mathbf{y}(k)$, yet they update $\mathbf{y}(k)$ differently via \eqref{eq:sparseyupdatetildev}, \eqref{eq:yintheeqalg}, and \eqref{eq:virtualqueueiny}.

Below, we unify these different updates of $\mathbf{y}(k)$. Note that \eqref{eq:sparseyupdatetildev}, \eqref{eq:yintheeqalg}, and \eqref{eq:virtualqueueiny} all involve minimizing the objective function $\Phi(\mathbf{y})$ of \eqref{eq:originalprobleminy} plus a few other terms denoted by $\Delta_i^k(\mathbf{y})$, $i=1,2,3$. Thus, we let $\mathbf{y}(0)$ be any vector in $\operatorname{dom}(h)$ and 
\begin{align}
\mathbf{y}(k+1) = \operatorname{\arg\;\min}_{\mathbf{y}\in \mathbb{R}^N}~\Phi(\mathbf{y})+\Delta^k(\mathbf{y}),\quad\forall k\ge0,\label{eq:proxfrogyupdate}
\end{align}
where $\Delta^k(\mathbf{y})=\Delta_1^k(\mathbf{y})+\Delta—_2^k(\mathbf{y})+\Delta_3^k(\mathbf{y})$. Note that \eqref{eq:proxfrogyupdate} is able to uniquely determine $\mathbf{y}(k+1)$ given the last iterates $\mathbf{y}(k)$, $\mathbf{v}(k)$, $\mathbf{u}(k)$, $\mathbf{z}(k)$, and $\mathbf{q}(k)$.

Furthermore, in order to reduce the computational complexity of \eqref{eq:proxfrogyupdate}, we replace the smooth component $f(\mathbf{y})$ in $\Phi(\mathbf{y})$ with the sum of its first-order approximation at $\mathbf{y}(k)$ (i.e., $f(\mathbf{y}(k))+\langle \nabla f(\mathbf{y}(k)), \mathbf{y}-\mathbf{y}(k)\rangle$) and a proximal term $\frac{\alpha_3}{2}\|\mathbf{y}-\mathbf{y}(k)\|^2$, $\alpha_3>0$, so that \eqref{eq:proxfrogyupdate} is modified to
\begin{equation}\label{eq:finalupdatey}
\begin{split}
&\mathbf{y}(k+1)=\operatorname{\arg\;\min}_{\mathbf{y}\in\mathbb{R}^N} \mathcal{L}^k(\mathbf{y}),\quad\forall k\ge0,
\end{split}
\end{equation}
where $\mathcal{L}^k(\mathbf{y})=\langle \nabla f(\mathbf{y}(k)), \mathbf{y}-\mathbf{y}(k)\rangle+h(\mathbf{y})+\Delta^k(\mathbf{y})+\frac{\alpha_3}{2}\|\mathbf{y}-\mathbf{y}(k)\|^2$. Note that \eqref{eq:finalupdatey} ensures $\mathbf{y}(k+1)\in\operatorname{dom}(h)$. Compared to \eqref{eq:proxfrogyupdate}, the primal update \eqref{eq:finalupdatey} is less computationally costly, especially when $h(\mathbf{y})\equiv 0$ and $\mathbf{G}(\mathbf{y})$ is simple.

Combining the above, we construct the following \emph{\uppercase{i}ntegrated \uppercase{p}rima\uppercase{l}-d\uppercase{u}al pro\uppercase{x}imal algorithm}, referred to as IPLUX: Given arbitrary $\mathbf{y}(0)\in\operatorname{dom}(h)$ and $\mathbf{u}(0)\in\mathbb{R}^{n(\tilde{m}+\tilde{p})}$, we initialize $\mathbf{v}(0)\in \operatorname{Range}((\mathbf{B}^s)^T)$, $\mathbf{z}(0)=\rho H\mathbf{u}(0)$, and $\mathbf{q}(0) = \max\{-\mathbf{G}(\mathbf{y}(0)),\mathbf{0}_{n\tilde{p}+\sum_{{i\in\mathcal{V}^{\operatorname{in}}}} p_i}\}$. Subsequently, for each $k\ge0$, we compute $\mathbf{y}(k+1)$, $\mathbf{v}(k+1)$, $\mathbf{u}(k+1)$, $\mathbf{z}(k+1)$, and $\mathbf{q}(k+1)$ according to \eqref{eq:finalupdatey}, \eqref{eq:sparsetildevupdate}, \eqref{eq:relationofdualandprimal}, \eqref{eq:pextrazupdate}, and \eqref{eq:virtualqueueq} with $\gamma,\lambda,\rho,\alpha_1,\alpha_2,\alpha_3>0$.

\begin{rem}
If problem~\eqref{eq:prob} does not contain the sparse equality constraint (2d), we may simply remove $\Delta_1^k(\mathbf{y})$ from $\mathcal{L}^k(\mathbf{y})$ in \eqref{eq:finalupdatey} and remove $\mathbf{v}(k)$ from IPLUX. When the dense constraints (2a) and (2b) do not exist, then $\Delta_2^k(\mathbf{y})$ and the variables $\mathbf{u}(k)$ and $\mathbf{z}(k)$ can be eliminated. Once the inequality constraints (2a) and (2c) are absent, then $\Delta_3^k(\mathbf{y})$ and $\mathbf{q}(k)$ are unnecessary.
\end{rem}

\begin{rem}\label{rem:dimension}
The dimensions of the variables in IPLUX add up to $2n\tilde{m}+5n\tilde{p}+2\sum_{i\in\mathcal{V}}d_i+\sum_{i\in\mathcal{V}^{\operatorname{in}}}p_i$. If we do not apply the constraint separation in Section~\ref{ssec:separation} to problem~\eqref{eq:proboriginal}, namely, all the sparse constraints in \eqref{eq:prob} are treated as dense constraints, then the total variable dimension of IPLUX becomes $2n\tilde{m}+4n\tilde{p}+\sum_{i\in\mathcal{V}}d_i+2n\sum_{i\in\mathcal{V}^{\operatorname{eq}}}m_i+4n\sum_{i\in\mathcal{V}^{\operatorname{in}}}p_i$. Therefore, the constraint separation reduces the variable dimension of IPLUX once $2n\sum_{i\in\mathcal{V}^{\operatorname{eq}}}m_i+(4n-1)\sum_{i\in\mathcal{V}^{\operatorname{in}}}p_i-n\tilde{p}-\sum_{i\in\mathcal{V}}d_i>0$, which implies that the constraints in \eqref{eq:proboriginal} contain plenty of sparsity and the decision variables are of moderate sizes relative to the number of sparse constraints.
\end{rem}

\subsection{Distributed Implementation}\label{ssec:distributed}

Although the dynamics of IPLUX have been provided in Section~\ref{ssec:alg}, the distributed implementation of IPLUX is not straightforward yet and will be discussed in this subsection. 

We first partition the variables of IPLUX as follows:
\begin{align*}
&\mathbf{y}(k)=[y_1^T(k),\ldots,y_n^T(k)]^T,\displaybreak[0]\\
&\mathbf{v}(k)=[v_1^T(k),\ldots,v_n^T(k)]^T,\;\mathbf{u}(k)=[u_1^T(k),\ldots,u_n^T(k)]^T,\displaybreak[0]\\
&\mathbf{z}(k)=[z_1^T(k),\ldots,z_n^T(k)]^T,\;\mathbf{q}(k)=[q_1^T(k),\ldots,q_n^T(k)]^T.
\end{align*}
Here, for each $i\in\mathcal{V}$, $y_i(k)\in \mathbb{R}^{d_i+\tilde{p}}$, $v_i(k)\in \mathbb{R}^{d_i+\tilde{p}}$, $u_i(k)\in \mathbb{R}^{\tilde{m}+\tilde{p}}$, and $z_i(k)\in \mathbb{R}^{\tilde{m}+\tilde{p}}$. Additionally, $q_i(k)\in \mathbb{R}^{p_i+\tilde{p}}$ if $i\in\mathcal{V}^{\operatorname{in}}$ and $q_i(k)\in \mathbb{R}^{\tilde{p}}$ otherwise. 

We allocate the variables $y_i(k)$, $v_i(k)$, $u_i(k)$, $z_i(k)$, and $q_i(k)$ to each node $i\in\mathcal{V}$. Similar to \eqref{eq:yi=xiti}, we write each $y_i(k)$ as $y_i(k)=[x_i^T(k),t_i^T(k)]^T$, with $x_i(k)\in\mathbb{R}^{d_i}$ and $t_i(k)\in\mathbb{R}^{\tilde{p}}$. Then, we further partition $v_i(k)$, $u_i(k)$, and $z_i(k)$ as follows: $v_i(k)=[(v_i^x(k))^T, (v_i^t(k))^T]^T$, $u_i(k)=[(u_i^x(k))^T, (u_i^t(k))^T]^T$, and $z_i(k)=[(z_i^x(k))^T, (z_i^t(k))^T]^T$. It will be shown shortly that $v_i^x(k)\in\mathbb{R}^{d_i}$, $u_i^x(k)\in\mathbb{R}^{\tilde{m}}$, and $z_i^x(k)\in\mathbb{R}^{\tilde{m}}$ are associated with $x_i(k)$, and that $v_i^t(k)\in\mathbb{R}^{\tilde{p}}$, $u_i^t(k)\in\mathbb{R}^{\tilde{p}}$, and $z_i^t(k)\in\mathbb{R}^{\tilde{p}}$ are associated with $t_i(k)$. In addition, for all $i\in\mathcal{V}^{\operatorname{in}}$, we let $q_i(k) = [(q_i'(k))^T, (q_i''(k))^T]^T$, with $q_i'(k)\in\mathbb{R}^{\tilde{p}}$ and $q_i''(k)\in\mathbb{R}^{p_i}$. For each $i\in\mathcal{V}-\mathcal{V}^{\operatorname{in}}$, we set $q_i(k)=q_i'(k)\in\mathbb{R}^{\tilde{p}}$ for convenience.

In addition to the above local variables, we let each node $i\in\mathcal{V}$ maintain two auxiliary variables defined as
\begin{align}
&r_i(k)=\sum_{j\in\mathcal{V}^{\operatorname{eq}}: i\in S_j^{\operatorname{eq}}} (A_{ji}^s)^T\Bigl(\bigl(\sum_{\ell\in S_j^{\operatorname{eq}}} A_{j\ell}^sx_\ell(k)\bigr)-b_j^s\Bigr),\label{eq:ri}\displaybreak[0]\\
&s_i(k)=G_i(\mathbf{y}(k)),\quad\forall k\ge0.\label{eq:si}
\end{align} 
These two variables will play a role in decentralizedly updating $y_i(k)$, $v_i(k)$, and $q_i(k)$. Note that if node $i$ does not belong to any $S_j^{\operatorname{eq}}$ $\forall j\in\mathcal{V}^{\operatorname{eq}}$, then $r_i(k)\equiv\mathbf{0}_{d_i}$, which means node $i$ does not need to update $r_i(k)$. We keep $r_i(k)$ for every node $i$ for the sake of simplifying the notation. Moreover, similar to $q_i(k)$, we write $s_i(k) = [(s_i'(k))^T, (s_i''(k))^T]^T$ for each $i\in\mathcal{V}^{\operatorname{in}}$, where $s_i'(k)\in\mathbb{R}^{\tilde{p}}$ and $s_i''(k)\in\mathbb{R}^{p_i}$, and set $s_i(k)=s_i'(k)\in\mathbb{R}^{\tilde{p}}$ for all $i\in\mathcal{V}-\mathcal{V}^{\operatorname{in}}$.

In the sequel, we investigate the initializations and updates of each node $i$'s local variables $y_i(k)$, $v_i(k)$, $u_i(k)$, $z_i(k)$, $q_i(k)$, $r_i(k)$, and $s_i(k)$ over the graph $\mathcal{G}$ in Section~\ref{ssec:graph}. 

Table~\ref{table:distributedcalculation} presents the distributed updates of $r_i(k)$ and $s_i(k)$. It can be seen that (r1), (r2), and (s2) only require a subset of the nodes to interact with part of their neighbors. For instance, according to (s2), node $i\in\mathcal{V}$ interacts with the nodes in the set $\{\ell\in\mathcal{V}^{\operatorname{in}}-\{i\}|~i\in S_\ell^{\operatorname{in}}\}$. This set is a subset of $\mathcal{N}_i$ and is possibly empty. If the set is empty, then node $i$ does not need to send out anything, but it will receive information from some of its neighbors when $S_i^{\operatorname{in}}-\{i\}\subseteq\mathcal{N}_i$ is nonempty.

\begin{table}[ht]
	\centering
	\caption{Distributed updates of $r_i(k)$ and $s_i(k)$.}
	\label{table:distributedcalculation}
	\bgroup
	\def\arraystretch{1.5}
	\begin{tabular}{|p{0.45\textwidth}|}
	\hline \textbf{Module of updating $r_i(k)$ $\forall i\in\mathcal{V}$ $\forall k\ge0$}:
		\begin{enumerate}[(r1)]
			\item Each node $i\in\mathcal{V}$ sends $A_{ji}^sx_i(k)$ to every node $j\in\{\ell \in\mathcal{V}^{\operatorname{eq}}-\{i\}|~i\in S_\ell^{\operatorname{eq}}\}\subseteq\mathcal{N}_i$.
			\item Each node $i\in\mathcal{V}^{\operatorname{eq}}$ computes $\tilde{r}_i^k:=(\sum_{j\in S_i^{\operatorname{eq}}} A_{ij}^sx_j(k))-b_i^s$ and sends $\tilde{r}_i^k$ to every node $j\in S_i^{\operatorname{eq}}-\{i\}\subseteq\mathcal{N}_i$.
			\item Each node $i\in \mathcal{V}$ sets $r_i(k)=\sum_{j\in\mathcal{V}^{\operatorname{eq}}: i\in S_j^{\operatorname{eq}}} (A_{ji}^s)^T\tilde{r}_j^k$.
		\end{enumerate}\\
		\hline \textbf{Module of updating $s_i(k)$ $\forall i\in\mathcal{V}$ $\forall k\ge0$}:
		\begin{enumerate}[(s1)]
		  \item Each node $i\in\mathcal{V}$ sets $s_i'(k)=g_i(x_i(k))-t_i(k)$.
			\item Each node $i\in\mathcal{V}$ sends $g_{ji}^s(x_i(k))$ to every node $j\in\{\ell\in\mathcal{V}^{\operatorname{in}}-\{i\}|~i\in S_\ell^{\operatorname{in}}\}\subseteq\mathcal{N}_i$.
			\item Each node $i\in\mathcal{V}^{\operatorname{in}}$ sets $s_i''(k)=\sum_{j\in S_i^{\operatorname{in}}} g_{ij}^s(x_j(k))$.	  
		\end{enumerate}\\
	\hline
	\end{tabular}
	\egroup
\end{table}

Subsequently, we discuss the initializations of the remaining variables. Recall from Section~\ref{ssec:alg} that we allow $\mathbf{u}(0)$ to be arbitrary, and require $\mathbf{y}(0)\in\operatorname{dom}(h)$, $\mathbf{v}(0)\in\operatorname{Range}((\mathbf{B}^s)^T)$, $\mathbf{z}(0)=\rho H\mathbf{u}(0)$, and $\mathbf{q}(0)=\max\{-\mathbf{G}(\mathbf{y}(0)),\mathbf{0}\}$. To realize such initializations, we arbitrarily pick $u_i(0)\in \mathbb{R}^{\tilde{m}+\tilde{p}}$, $x_i(0)\in\operatorname{dom}(h_i)$, and $t_i(0)\in\mathbb{R}^{\tilde{p}}$. Besides, we simply set $v_i(0)=\mathbf{0}_{d_i+\tilde{p}}$ $\forall i\in\mathcal{V}$, so that $\mathbf{v}(0)=\mathbf{0}_N$ is guaranteed to be in $\operatorname{Range}((\mathbf{B}^s)^T)$. Additionally, due to the structure of $H$ in Section~\ref{ssec:proximaldense}, $\mathbf{z}(0)=\rho H\mathbf{u}(0)$ is equivalent to $z_i(0)=\rho\sum_{j\in\mathcal{N}_i\cup\{i\}} [P^H]_{ij}u_j(0)$ $\forall i\in\mathcal{V}$, which can be initialized by letting each node $i$ exchange $u_i(0)$ with all the neighbors\footnote{Since $u_i(0)$ can be arbitrary, if we choose $u_i(0)=\mathbf{0}_{\tilde{m}+\tilde{p}}$ $\forall i\in\mathcal{V}$, then each node $i$ can directly set $z_i(0)=\mathbf{0}_{\tilde{m}+\tilde{p}}$ by default.}. Lastly, $\mathbf{q}(0)$ can be found by the nodes in parallel via $q_i(0) = \max\{-s_i(0), \mathbf{0}\}$ $\forall i\in\mathcal{V}$.

To update each $v_i(k)$ through \eqref{eq:sparsetildevupdate}, note from \eqref{eq:ri} that
\begin{align}
&[(\mathbf{B}^s)^T(\mathbf{B}^s\mathbf{y}(k)-\mathbf{c}^s)]^T=[(r_1(k))^T, \mathbf{0}_{\tilde{p}}^T,\ldots,(r_n(k))^T, \mathbf{0}_{\tilde{p}}^T].\label{eq:BByc=r0r0}
\end{align}
This, along with \eqref{eq:sparsetildevupdate} and $v_i(0)=\mathbf{0}_{d_i+\tilde{p}}$, implies that 
\begin{align}
v_i^t(k)\equiv\mathbf{0}_{\tilde{p}},\quad\forall i\in\mathcal{V}.\label{eq:vit=0}
\end{align}
Hence, $v_i^t(k)$ takes no effect in the algorithm. Indeed, each node $i$ only needs to maintain $v_i^x(k)$ instead of the entire vector $v_i(k)$, which can be updated via
\begin{align}
v_i^x(k+1)=v_i^x(k)+\gamma r_i(k+1).\label{eq:vi1distributedimplementation}
\end{align}
In addition, note that \eqref{eq:relationofdualandprimal}, \eqref{eq:pextrazupdate}, and \eqref{eq:virtualqueueq} can be decomposed as follows: For each $i\in\mathcal{V}$,
\begin{align}
u_i(k+1)=& \frac{1}{\rho}(B_iy_i(k+1)\!-\!c_i\!-\!z_i(k))\!+\!\!\!\!\!\sum_{j\in \mathcal{N}_i\cup\{i\}}\!\!\!\!\! [P^W]_{ij}u_j(k),\label{eq:udistributedimplementation}\\
z_i(k+1)=& z_i(k)+\rho\sum_{j\in \mathcal{N}_i\cup\{i\}}[P^H]_{ij}u_j(k+1),\label{eq:zdistributedimplementation}\\
q_i(k+1)=& \max\{-s_i(k+1), q_i(k)+s_i(k+1)\}.\label{eq:Qdistributedimplementation}
\end{align}
From the above, the computations of $v_i^x(k+1)$ (or equivalently $v_i(k+1)$), $u_i(k+1)$, $z_i(k+1)$, and $q_i(k+1)$ are all local operations, which only require each node to acquire $u_j(k)$ and $u_j(k+1)$ from every neighbor $j\in\mathcal{N}_i$.

Finally, we decompose $\mathcal{L}^k(\mathbf{y})$ in \eqref{eq:finalupdatey} in order to describe the update of $y_i(k)=[x_i^T(k),t_i^T(k)]^T$ $\forall i\in\mathcal{V}$. By using \eqref{eq:si}, we obtain $\langle \mathbf{q}(k)+\mathbf{G}(\mathbf{y}(k)), \mathbf{G}(\mathbf{y})\rangle=\sum_{i\in \mathcal{V}}\langle q_i'(k)+s_i'(k), g_i(x_i)-t_i\rangle+\sum_{i\in\mathcal{V}^{\operatorname{in}}}\langle q_i''(k)+s_i''(k),\sum_{j\in S_i^{\operatorname{in}}}g_{ij}^s(x_j)\rangle=\sum_{i\in \mathcal{V}}\Bigl(\langle q_i'(k)+s_i'(k), g_i(x_i)-t_i\rangle+\sum_{j\in\mathcal{V}^{\operatorname{in}}:i\in S_j^{\operatorname{in}}}\langle q_j''(k)+s_j''(k),g_{ji}^s(x_i)\rangle\Bigr).$ It then follows from \eqref{eq:BByc=r0r0} and \eqref{eq:vit=0} that $\mathcal{L}^k(\mathbf{y})$ can be decomposed as
\begin{align*}
\mathcal{L}^k(\mathbf{y})=\sum_{i\in \mathcal{V}} (\mathcal{L}_i^k(x_i)+\tilde{\mathcal{L}}_i^k(t_i)).
\end{align*}
For each $k\ge0$ and $i\in\mathcal{V}$, $\mathcal{L}_i^k(x_i)$ is a function of $x_i$ and $\tilde{\mathcal{L}}_i^k(t_i)$ is a function of $t_i$, which depend on the iterates of node $i$ and its neighbors at iteration $k$. Specifically, $\mathcal{L}_i^k(x_i)=\langle\nabla f_i(x_i(k)),x_i-x_i(k)\rangle+h_i(x_i)+\langle v_i^x(k),x_i\rangle+\frac{\gamma\lambda^2}{2}\|x_i-x_i(k)+\frac{1}{\lambda^2}r_i(k)\|^2+\frac{1}{2\rho}\|A_ix_i-b_i\|^2+\langle\bigl(\sum_{j\in\mathcal{N}_i\cup\{i\}}[P^W]_{ij}u_j^x(k)\bigr)-\frac{1}{\rho}z_i^x(k),A_ix_i-b_i\rangle+\langle q_i'(k)+s_i'(k),g_i(x_i)\rangle+\Bigl(\sum_{j\in\mathcal{V}^{\operatorname{in}}:i\in S_j^{\operatorname{in}}}\langle q_j''(k)+s_j''(k),g_{ji}^s(x_i)\rangle\Bigr)+\frac{\alpha}{2}\|x_i-x_i(k)\|^2$ and $\tilde{\mathcal{L}}_i^k(t_i)=\frac{\gamma\lambda^2+\alpha}{2}\|t_i-t_i(k)\|^2+\frac{1}{2\rho}\|t_i\|^2+\langle\bigl(\sum_{j\in\mathcal{N}_i\cup\{i\}}[P^W]_{ij}u_j^t(k)\bigr)-\frac{1}{\rho}z_i^t(k),t_i\rangle-\langle q_i'(k)+s_i'(k),t_i\rangle$, where $\alpha=\alpha_1+\alpha_2+\alpha_3>0$ and $\rho,\gamma,\lambda>0$.

As a result, the update \eqref{eq:finalupdatey} of $\mathbf{y}(k)$ is equivalent to letting each node $i\in\mathcal{V}$ implement
\begin{align}
x_i(k+1)&= \operatorname{\arg\;\min}_{x_i\in \mathbb{R}^{d_i}} \mathcal{L}_i^k(x_i),\label{eq:updateofxi}\\
t_i(k+1)&= \operatorname{\arg\;\min}_{t_i\in \mathbb{R}^{\tilde{p}}} \tilde{\mathcal{L}}_i^k(t_i).\label{eq:timinimization}
\end{align}
Furthermore, since \eqref{eq:timinimization} involves solving a convex quadratic program, it can be identically expressed as
\begin{align}
&t_i(k+1)=\frac{1}{1/\rho+\gamma\lambda^2+\alpha}\label{eq:updateofti}\displaybreak[0]\\
&\cdot\!\!\Bigl(\!\!(\gamma\lambda^2\!\!+\!\alpha)t_i(k)\!-\!\!\bigl(\!\!\!\!\!\!\sum_{j\in\mathcal{N}_i\cup\{i\}}\!\!\!\!\!\!\![P^W]_{ij}u_j^t(k)\bigr)\!\!+\!\frac{1}{\rho}z_i^t(k)\!+\!q_i'(k)\!+\!s_i'(k)\!\Bigr).\nonumber
\end{align}
Observe from \eqref{eq:updateofxi} and \eqref{eq:updateofti} that each node $i$ is able to locally compute $x_i(k+1)$ and $t_i(k+1)$ by collecting 
$u_j(k)=[(u_j^x(k))^T,(u_j^t(k))^T]^T$ from every neighbor $j\in\mathcal{N}_i$ and $q_j''(k)+s_j''(k)$ from each selected neighbor $j\in\{\ell\in\mathcal{V}^{\operatorname{in}}-\{i\}|~i\in S_{\ell}^{\operatorname{in}}\}\subseteq\mathcal{N}_i$.

Combining the above, Algorithm~\ref{alg:IPLUX} describes the distributed implementation of IPLUX developed in Section~\ref{ssec:alg}. To execute Algorithm~\ref{alg:IPLUX}, all the nodes need to agree on the values of $\gamma,\lambda,\rho,\alpha>0$. Also, according to \cite{ShiW15a}, each node $i$ manages to locally determine $[P^W]_{ij}$ and $[P^H]_{ij}$ $\forall j\in\mathcal{N}_{i}\cup\{i\}$ such that Assumption~\ref{asm:weightmatrix} holds. 

{
	\renewcommand{\baselinestretch}{1.05}
	\begin{algorithm} [!htb]
		\caption{\small Distributed Implementation of IPLUX}
		\label{alg:IPLUX}
		\begin{algorithmic}[1]
			\small
			\STATE Each node $i\in\mathcal{V}$ maintains $x_i(k)$, $t_i(k)$, $v_i^x(k)$, $u_i(k)$, $z_i(k)$, $q_i(k)$, $r_i(k)$, and $s_i(k)$.
			\STATE \textbf{Initialization:}
			\STATE Each node $i\in\mathcal{V}$ sets $v_i^x(0)=\mathbf{0}_{d_i}$ and arbitrarily selects $x_i(0)\in\operatorname{dom}(h_i)$, $t_i(0)\in\mathbb{R}^{\tilde{p}}$, and $u_i(0)\in\mathbb{R}^{\tilde{m}+\tilde{p}}$.
			\STATE Each node $i\in\mathcal{V}$ sends $u_i(0)$ to all $j\in \mathcal{N}_i$.
			\STATE Each node $i\in\mathcal{V}$ sets $z_i(0)=\rho\sum_{j\in\mathcal{N}_i\cup\{i\}}[P^H]_{ij}u_j(0)$.
			\STATE Each node $i\in\mathcal{V}$ computes $r_i(0)$ and $s_i(0)$ according to Table~\ref{table:distributedcalculation}.
			\STATE Each node $i\in\mathcal{V}$ sets $q_i(0) = \max\{-s_i(0), \mathbf{0}\}$.
			\FOR{ $k=0,1,\ldots$}
			\STATE Each node $i\!\in\!\mathcal{V}^{\operatorname{in}}$ sends $q_i''(k)\!+\!s_i''(k)$ to all $j\!\in\! S_i^{\operatorname{in}}\!-\{i\}\!\subseteq\!\mathcal{N}_i$.
			\STATE Each node $i\in\mathcal{V}$ computes $x_i(k+1)$ and $t_i(k+1)$ according to \eqref{eq:updateofxi} and \eqref{eq:updateofti}, respectively.
			\STATE Each node $i\in\mathcal{V}$ computes $r_i(k+1)$ and $s_i(k+1)$ according to Table~\ref{table:distributedcalculation}.
			\STATE Each node $i\in \mathcal{V}$ computes $v_i^x(k+1)$, $u_i(k+1)$, and $q_i(k+1)$ according to \eqref{eq:vi1distributedimplementation}, \eqref{eq:udistributedimplementation}, and \eqref{eq:Qdistributedimplementation}, respectively.
			\STATE Each node $i\in\mathcal{V}$ sends $u_i(k+1)$ to all $j\in\mathcal{N}_i$.
			\STATE Each node $i\in \mathcal{V}$ computes $z_i(k+1)$ according to \eqref{eq:zdistributedimplementation}.
			\ENDFOR
		\end{algorithmic}
	\end{algorithm}
}

\begin{rem}
The three algorithms in Sections~\ref{ssec:proximalsparse}, \ref{ssec:proximaldense}, and \ref{ssec:virtualqueue} for solving special forms of problem~\eqref{eq:originalprobleminy} can be specialized from a variant of IPLUX with \eqref{eq:finalupdatey} replaced by \eqref{eq:proxfrogyupdate}. Since the distributed realization of \eqref{eq:proxfrogyupdate} is very close to that of \eqref{eq:finalupdatey}, these three algorithms are also distributed.
\end{rem}

\section{Convergence Analysis}\label{sec:convanal}

In this section, we analyze the convergence behavior of IPLUX under the following problem assumptions.

\begin{assumption}\label{asm:prob}
Problem \eqref{eq:prob} satisfies the following:
\begin{enumerate}[(a)]
\item Each $f_i(x_i)$ and $h_i(x_i)$, $i\in \mathcal{V}$ are convex on $\operatorname{dom}(h_i)$, where $\operatorname{dom}(h_i)$ is convex.
\item Each $f_i(x_i)$, $i\in \mathcal{V}$ is differentiable on $\operatorname{dom}(h_i)$ and its gradient $\nabla f_i(x_i)$ is Lipschitz continuous with Lipschitz constant $L_f>0$ on $\operatorname{dom}(h_i)$.
\item Each $g_i(x_i)$, $i\in \mathcal{V}$ is convex on $\operatorname{dom}(h_i)$ and Lipschitz continuous with Lipschitz constant $L_g>0$ on $\operatorname{dom}(h_i)$.
\item Each $g_{ij}^s(x_j)$, $i\in \mathcal{V}^{\operatorname{in}}$, $j\in S_i^{\operatorname{in}}$ is convex on $\operatorname{dom}(h_j)$ and Lipschitz continuous with Lipschitz constant $L_{g^s}>0$ on $\operatorname{dom}(h_j)$.
\item There exists at least one optimal solution $\mathbf{x}^\star=[(x_1^\star)^T, \ldots, (x_n^\star)^T]^T$ to problem \eqref{eq:prob}.
\item There exist $\tilde{x}_i\in \operatorname{rel\;int}(\operatorname{dom}(h_i))$ $\forall i\in \mathcal{V}$ such that $\sum_{i\in \mathcal{V}} g_i(\tilde{x}_i)<\mathbf{0}_{\tilde{p}}$, $\sum_{i\in \mathcal{V}} A_i\tilde{x}_i=\sum_{i\in \mathcal{V}}b_i$, $\sum_{j\in S_i^{\operatorname{in}}} g_{ij}^s(\tilde{x}_j)<\mathbf{0}_{p_i}$ $\forall i\in \mathcal{V}^{\operatorname{in}}$, and $\sum_{j\in S_i^{\operatorname{eq}}} A_{ij}^s\tilde{x}_j= b_i^s$ $\forall i\in \mathcal{V}^{\operatorname{eq}}$.
\end{enumerate}
\end{assumption}

Assumption \ref{asm:prob} allows each local objective function $f_i+h_i$ in problem~\eqref{eq:prob} to be a general convex function on $\operatorname{dom}(h_i)$. The Lipschitz continuity in Assumptions~\ref{asm:prob}(b)--(d) only needs to hold on the domains of the $h_i$'s. If the domains happen to be compact, then such Lipschitz continuity can be guaranteed by continuous differentiability. For simplicity, we let the Lipschitz constants $L_f,L_g,L_{g^s}$ be uniform for all the $\nabla f_i(x_i)$'s, $g_i(x_i)$'s, and $g_{ij}^s(x_j)$'s, respectively. Assumption~\ref{asm:prob}(e) ensures that problem~\eqref{eq:prob} is solvable. In addition, Assumption \ref{asm:prob}(f) is the Slater's condition, which is used to guarantee the strong duality between problem~\eqref{eq:prob} and its Lagrange dual \cite{Bertsekas99}. 

The following proposition shows that the equivalent form \eqref{eq:originalprobleminy} of problem~\eqref{eq:prob} enjoys several favorable properties owing to Assumption \ref{asm:prob}.

\begin{proposition}\label{prop:propertyofproby}
Suppose Assumption \ref{asm:prob} holds. Then, problem~\eqref{eq:originalprobleminy} has the following properties:
\begin{enumerate}[(a)]
\item $f(\mathbf{y})$, $h(\mathbf{y})$, and $\Phi(\mathbf{y})$ are convex on $\operatorname{dom}(h)$.
\item $\nabla f(\mathbf{y})$ is Lipschitz continuous with Lipschitz constant $L_f$ on $\operatorname{dom}(h)$.
\item $\mathbf{G}(\mathbf{y})$ is Lipschitz continuous with Lipschitz constant $L>0$ on $\operatorname{dom}(h)$, where $L:=\Bigl((\max_{i\in\mathcal{V}}\sum_{j\in\mathcal{V}^{\operatorname{in}}:i\in S_j^{\operatorname{in}}}|S_j^{\operatorname{in}}|)L_{g^s}^2+1+L_g^2\Bigr)^{\frac{1}{2}}$.
\item There exists at least one optimal solution $\mathbf{y}^\star\in\operatorname{dom}(h)$ to problem \eqref{eq:originalprobleminy}.
\item Strong duality holds between \eqref{eq:originalprobleminy} and its Lagrange dual.
\end{enumerate}
\end{proposition}

\begin{proof}
See Appendix~\ref{ssec:proofofproLipschitz}.
\end{proof}

Next, we derive a couple of auxiliary lemmas towards establishing the convergence results. Let $\mathbf{y}(k)$, $\mathbf{v}(k)$, $\mathbf{u}(k)$, $\mathbf{z}(k)$, and $\mathbf{q}(k)$ be generated by IPLUX described in Section~\ref{ssec:alg} with parameters $\gamma>0$, $\lambda>0$, $\rho>0$, and $\alpha=\alpha_1+\alpha_2+\alpha_3>0$. The first lemma bounds the accumulative difference between the objective value and the optimal value of problem~\eqref{eq:originalprobleminy} with the aid of Proposition~\ref{prop:propertyofproby}.

\begin{lemma}\label{lemma:boundedsumF}
Suppose Assumptions~\ref{asm:weightmatrix} and~\ref{asm:prob} hold. If $\alpha\ge L_f+L^2$ and $\lambda\ge\|\mathbf{B}^s\|_2$, then for any $k\ge 1$, 
\begin{equation}\label{eq:sumPhiykystarupperbound}
\sum_{\ell=1}^k (\Phi(\mathbf{y}(\ell))-\Phi(\mathbf{y}^\star)) \le R(0)-R(k),
\end{equation}
where $R(k)=\frac{1}{2\rho}\|\mathbf{z}(k)-\mathbf{z}^\star\|_{H^\dag}^2+\frac{\rho}{2}\|\mathbf{u}(k)\|_W^2+\frac{1}{2\gamma}\|\mathbf{v}(k)\|_{((\mathbf{B}^s)^T\mathbf{B}^s)^\dag}^2+\frac{\gamma\lambda^2+\alpha}{2}\|\mathbf{y}(k)-\mathbf{y}^\star\|^2+\frac{1}{2}\|\mathbf{q}(k)\|^2-\frac{1}{2}\|\mathbf{G}(\mathbf{y}(k))\|^2\ge0$ $\forall k\ge 0$ and $\mathbf{z}^\star=\mathbf{B}\mathbf{y}^\star-\mathbf{c}$.
\end{lemma}

\begin{proof}
See Appendix~\ref{ssec:proofoflemmaboundedsumF}.
\end{proof}

For each $k\ge1$, let $\bar{\mathbf{y}}(k):=\frac{1}{k}\sum_{\ell=1}^k \mathbf{y}(\ell)$ be the running average of $\mathbf{y}(\ell)$ from $\ell=1$ to $\ell=k$. The second lemma connects the constraint violations regarding problem~\eqref{eq:originalprobleminy} at $\bar{\mathbf{y}}(k)$ with the iterates $\mathbf{q}(k), \mathbf{u}(k), \mathbf{v}(k)$ of IPLUX. 

\begin{lemma}\label{lemma:feasibility}
Suppose Assumptions~\ref{asm:weightmatrix} and~\ref{asm:prob} hold. For any $k\ge 1$,
	\begin{align}
	&\mathbf{G}(\bar{\mathbf{y}}(k))\le \frac{\mathbf{q}(k)}{k},\label{eq:Gbaryklemma}\\
	&(\mathbf{1}_n\!\otimes\! I_{\tilde{m}+\tilde{p}})^T(\mathbf{B}\bar{\mathbf{y}}(k)-\mathbf{c})=\frac{\rho}{k}(\mathbf{1}_n\!\otimes\! I_{\tilde{m}+\tilde{p}})^T(\mathbf{u}(k)\!-\!\mathbf{u}(0)),\label{eq:Bbaryklemma}\\
	&\|\mathbf{B}^s\bar{\mathbf{y}}(k)-\mathbf{c}^s\|= \frac{\|\mathbf{v}(k)-\mathbf{v}(0)\|_{((\mathbf{B}^s)^T\mathbf{B}^s)^{\dag}}}{\gamma k}.\label{eq:Bsbaryklemma}
	\end{align}
\end{lemma}
\begin{proof}
	See Appendix \ref{sec:prooflemmafeasibility}.
\end{proof}

From Lemma~\ref{lemma:feasibility}, if $\mathbf{q}(k)$, $(\mathbf{1}_n\otimes I_{\tilde{m}+\tilde{p}})^T(\mathbf{u}(k)-\mathbf{u}(0))$, and $\|\mathbf{v}(k)-\mathbf{v}(0)\|_{((\mathbf{B}^s)^T\mathbf{B}^s)^{\dag}}$ are bounded, then $\bar{\mathbf{y}}(k)$ would reach the feasibility of problem~\eqref{eq:originalprobleminy} at a rate of $O(1/k)$. As is shown in Theorem~\ref{theo:funcval} below, this is indeed true, and the optimal value of problem~\eqref{eq:originalprobleminy} can also be achieved at an $O(1/k)$ rate. 

To introduce Theorem~\ref{theo:funcval}, we let $(\mathbf{q}^\star, u^\star, \tilde{\mathbf{v}}^\star)$ be a geometric multiplier for problem~\eqref{eq:originalprobleminy}, where $\mathbf{q}^\star\ge\mathbf{0}_{n\tilde{p}+\sum_{{i\in\mathcal{V}^{\operatorname{in}}}} p_i}$, $u^\star\in\mathbb{R}^{\tilde{m}+\tilde{p}}$, and $\tilde{\mathbf{v}}^\star\in\mathbb{R}^{\sum_{i\in\mathcal{V}^{\operatorname{eq}}}m_i}$. Owing to Proposition \ref{prop:propertyofproby}(e),
\begin{align}
&\Phi(\mathbf{y}^\star)=\min_{\mathbf{y}\in \mathbb{R}^N}~\Phi(\mathbf{y})+\langle\mathbf{q}^\star,\mathbf{G}(\mathbf{y})\rangle\nonumber\displaybreak[0]\\
&\quad+\langle u^\star,(\mathbf{1}_n\otimes I_{\tilde{m}+\tilde{p}})^T(\mathbf{B}\mathbf{y}-\mathbf{c})\rangle+\langle\tilde{\mathbf{v}}^\star,\mathbf{B}^s\mathbf{y}-\mathbf{c}^s\rangle.\label{eq:ystarminimum}
\end{align}
Then, we present our main results in Theorem~\ref{theo:funcval}.

\begin{theorem}\label{theo:funcval}
Suppose Assumptions~\ref{asm:weightmatrix} and~\ref{asm:prob} hold. Also suppose $\alpha\ge L_f+L^2$ and $\lambda\ge\|\mathbf{B}^s\|_2$. Then, for any $k\ge 1$,
\begin{align}
	&\mathbf{G}(\bar{\mathbf{y}}(k))\le \frac{V_{\mathbf{q}}}{k}\mathbf{1}_{n\tilde{p}+\underset{i\in\mathcal{V}^{\operatorname{in}}}{\sum} p_i},\label{eq:Gerror}\\
	&\|(\mathbf{1}\otimes I_{\tilde{m}+\tilde{p}})^T(\mathbf{B}\bar{\mathbf{y}}(k)-\mathbf{c})\|\le\frac{\rho V_{\mathbf{u}}}{k},\label{eq:Berror}\\
	&\|\mathbf{B}^s\bar{\mathbf{y}}(k)-\mathbf{c}^s\|\le\frac{V_{\mathbf{v}}}{\gamma k},\label{eq:Bserror}\\
	&-\frac{R_{\Phi}}{k}\le\Phi(\bar{\mathbf{y}}(k))-\Phi(\mathbf{y}^\star)\le\frac{R_{\Phi}'}{k},\label{eq:theofuncvalbound}
\end{align}
where $V_{\mathbf{q}},V_{\mathbf{u}},V_{\mathbf{v}},R_{\Phi},R_{\Phi}'\in[0,\infty)$ are given by
\begin{align*}
V_{\mathbf{q}}=&2(\|\mathbf{q}^\star\|+\sqrt{C}),\displaybreak[0]\\
V_{\mathbf{u}}=&\sqrt{n}(\|\mathbf{u}(0)\|_W+\sqrt{n}\|u^\star\|+\sqrt{2C/\rho}),\displaybreak[0]\\
V_{\mathbf{v}}=&\|\mathbf{v}(0)\|_{((\mathbf{B}^s)^T\mathbf{B}^s)^{\dag}}+\|\tilde{\mathbf{v}}^\star\|+\sqrt{2\gamma C},\displaybreak[0]\\
R_{\Phi}=&V_{\mathbf{q}}\mathbf{1}_{n\tilde{p}+\sum_{i\in\mathcal{V}^{\operatorname{in}}} p_i}^T\mathbf{q}^\star+\rho V_{\mathbf{u}}\|u^\star\|+V_{\mathbf{v}}\|\tilde{\mathbf{v}}^\star\|/\gamma,\displaybreak[0]\\
R_{\Phi}'=&R(0),\displaybreak[0]\\
C=&R(0)\!+\!\rho\sqrt{n}\|u^\star\|\!\cdot\!\|\mathbf{u}(0)\|_W\!+\!\frac{1}{\gamma}\|\tilde{\mathbf{v}}^\star\|\!\cdot\!\|\mathbf{v}(0)\|_{((\mathbf{B}^s)^T\mathbf{B}^s)^{\dag}}\displaybreak[0]\\
&+\frac{\rho n}{2}\|u^\star\|^2+\frac{1}{2\gamma}\|\tilde{\mathbf{v}}^\star\|^2+\|\mathbf{q}^\star\|^2+\frac{1}{2}\|\mathbf{G}(\mathbf{y}^\star)\|^2.
\end{align*}
\end{theorem}

\begin{proof}
See Appendix \ref{ssec:proofoffuncval}.
\end{proof}

Theorem \ref{theo:funcval} states that IPLUX achieves the optimality and feasibility of problem~\eqref{eq:originalprobleminy} at an $O(1/k)$ rate. The following corollary traces back to problem~\eqref{eq:prob} and provides $O(1/k)$ bounds on the constraint violations and the objective error with respect to problem~\eqref{eq:prob} at $\bar{x}_i(k)=\frac{1}{k}\sum_{\ell=1}^k x_i(\ell)$ $\forall i\in\mathcal{V}$ $\forall k\ge 1$, where the $x_i(k)$'s are generated in Algorithm~\ref{alg:IPLUX}.

\begin{corollary}\label{cor:originalprobconv}
Suppose all the conditions in Theorem~\ref{theo:funcval} hold. Then, for any $k\ge 1$,
\begin{align}
&\sum_{i\in\mathcal{V}}g_i(\bar{x}_i(k))\le\frac{nV_{\mathbf{q}}+\rho V_{\mathbf{u}}}{k}\mathbf{1}_{\tilde{p}},\displaybreak[0]\label{eq:globineqconv}\\
&\sum_{j\in S_i^{\operatorname{in}}} g_{ij}^s(\bar{x}_j(k))\le \frac{V_{\mathbf{q}}}{k}\mathbf{1}_{p_i},\quad\forall i\in\mathcal{V}^{\operatorname{in}},\label{eq:sparseineqconv}\\
&\|\sum_{i\in\mathcal{V}}(A_i\bar{x}_i(k)-b_i)\|\le\frac{\rho V_{\mathbf{u}}}{k},\displaybreak[0]\label{eq:globeqconv}\\
&\|\Bigl(\sum_{j\in S_i^{\operatorname{eq}}}A_{ij}^s\bar{x}_j(k)\Bigr)-b_i^s\|\le\frac{V_{\mathbf{v}}}{\gamma k},\quad\forall i\in\mathcal{V}^{\operatorname{eq}},\label{eq:sparseeqconv}\\
&-\frac{R_{\Phi}}{k}\!\le\!\sum_{i\in \mathcal{V}} ((f_i+h_i)(\bar{x}_i(k))\!-\!(f_i+h_i)(x_i^\star))\!\le\!\frac{R_{\Phi}'}{k},\label{eq:funcvalbound}
\end{align}
where the constants are given in Theorem~\ref{theo:funcval}.
\end{corollary}

\begin{proof}
See Appendix \ref{ssec:proofofcororiginprob}.
\end{proof}

\subsection{Comparison with Related Works}

This subsection compares IPLUX and the prior distributed optimization methods \cite{ChangTH14,Falsone17,Notarnicola20,Liang19,Liang19a} in their assumptions and convergence results, whereby \cite{ChangTH14,Falsone17,Notarnicola20} can also tackle nonlinear inequality constraints like (2a) and (2c) in problem~\eqref{eq:prob} and \cite{Liang19,Liang19a} manage to handle both nonlinear inequality and linear equality constraints like (2a)--(2d). Note that they all assume the problem to be convex. Moreover, \cite{Notarnicola20,Liang19,Liang19a}, like IPLUX, consider undirected, fixed graphs, while \cite{ChangTH14,Falsone17} allow for time-varying directed graphs.

From Table~\ref{table:comparison}, \cite{ChangTH14,Liang19,Liang19a} require that the objective function be smooth on a closed convex set, i.e., each $h_i$ in \eqref{eq:prob} be equal to an indicator function with respect to that set, yet IPLUX and \cite{Falsone17,Notarnicola20} can address problems with general nonsmooth objective functions. In addition, \cite{ChangTH14,Falsone17,Notarnicola20,Liang19a} assume that the problem is constrained by a compact set, while IPLUX and \cite{Liang19} do not need such compactness. When it comes to the nonlinear inequality constraint functions, IPLUX requires a slightly stronger condition than \cite{Falsone17,Notarnicola20}. Specifically, IPLUX and \cite{Liang19a} need Lipschitz continuity of the inequality constraint functions, which is relaxed to local Lipschitz continuity in \cite{Liang19}. However, unlike \cite{ChangTH14,Liang19}, IPLUX does not require the inequality constraint functions to have Lipschitz gradients. Furthermore, the Slater's condition is necessary for all these algorithms to establish their convergence results.

Among the algorithms in Table~\ref{table:comparison}, only IPLUX and \cite{Liang19a} derive non-asymptotic convergence results. It can be seen that the $O(1/k)$ convergence rate of IPLUX is much faster than the $O(\ln k/\sqrt{k})$ rate in \cite{Liang19a}, and IPLUX indeed imposes weaker assumptions than \cite{Liang19a}. Finally, the algorithms in \cite{ChangTH14,Falsone17,Notarnicola20,Liang19a} adopt diminishing step-sizes, which may cause slow convergence in practice.

\begin{table*}[tb]
\centering
\caption{{\upshape Comparison in problem assumptions and convergence results. Here, $\surd$ means the assumption is required.}}\label{table:comparison}
\vspace*{0.1in}
\begin{tabular}{|c|c|c|c|c|c|c|c|c|c|}
\hline
Algorithm & smooth & compact & Lipschitz continuous & smooth & Slater's &convergence\\
& objective & domain & constraint function & constraint function & condition & result\\
\hline
primal-dual perturbation method \cite{ChangTH14} & $\surd$ & $\surd$ & & $\surd$ &$\surd$ &asymptotic\\
\hline
dual-decomposition-based method \cite{Falsone17} & & $\surd$ & & & $\surd$ & asymptotic\\
\hline
RSDD \cite{Notarnicola20} & & $\surd$ & & & $\surd$ & asymptotic\\
\hline
primal-dual gradient method \cite{Liang19} & $\surd$ & & $\surd$ (local) & $\surd$ (local) & $\surd$ & asymptotic\\
\hline
dual subgradient method \cite{Liang19a} & $\surd$ & $\surd$ & $\surd$ & & $\surd$ & $O(\ln k/\sqrt{k})$\\
\hline
IPLUX & & & $\surd$ & & $\surd$ & $O(1/k)$\\
\hline
\end{tabular}
\end{table*}

\section{Numerical Examples}\label{sec:numericalexample}

In this section, we demonstrate the practical performance of IPLUX in solving distributed convex optimization problems with both nonlinear inequality and linear equality constraints.

\subsection{Smooth Problem}\label{ssec:smooth}

We first consider the following problem:
\begin{align}
\underset{x_1,\ldots,x_n\in\mathbb{R}^d}{\operatorname{minimize}} ~&~ \sum_{i=1}^n\! \left(x_i^TP_ix_i+Q_i^Tx_i+\mathcal{I}_{X_i}(x_i)\right)\displaybreak[0]\nonumber\\
\operatorname{subject~to}~&~ \sum_{i=1}^n (\|x_i-a_i'\|^2-c_i') \le 0,\displaybreak[0]\nonumber\\
&~\sum_{i=1}^n A_ix_i = \mathbf{0}_{\tilde{m}},\displaybreak[0]\nonumber\\
&~\sum_{j\in S_i^{\text{in}}} (\|x_j-a_{ij}''\|^2-c_{ij}'')\le 0,\quad\forall i\in\mathcal{V}^{\text{in}},\displaybreak[0]\nonumber\\
&~\sum_{j\in S_i^{\text{eq}}} A_{ij}^sx_j= \mathbf{0}_{m^s},\quad\forall i\in\mathcal{V}^{\text{eq}}.\label{eq:simuprob}
\end{align}
In the objective function, each $P_i\in\mathbb{R}^{d\times d}$ is symmetric positive semidefinite, $Q_i\in \mathbb{R}^d$, and $\mathcal{I}_{X_i}$ is the indicator function with respect to the set $X_i=\{x\in\mathbb{R}^d|~\|x-a_i\|^2\le c_i\}$ with $a_i\in\mathbb{R}^d$ and $c_i>\|a_i\|^2$, so that we set each $f_i(x_i)=x_i^TP_ix_i+Q_i^Tx_i$ and each $h_i(x_i)=\mathcal{I}_{X_i}(x_i)$. In the constraints, each $a_i',a_{ij}''\in \mathbb{R}^d$, $c_i',c_{ij}''\in\mathbb{R}$, $A_i\in\mathbb{R}^{\tilde{m}\times d}$, $A_{ij}^s\in\mathbb{R}^{m^s\times d}$, and $\mathcal{V}^{\operatorname{in}},S_i^{\operatorname{in}},\mathcal{V}^{\operatorname{eq}},S_i^{\operatorname{eq}}$ are nonempty subsets of $\mathcal{V}=\{1,\ldots,n\}$. We also require $\sum_{i=1}^n c_i'> \sum_{i=1}^n \|a_i'\|^2$ and $\sum_{j\in S_i^{\text{in}}}c_{ij}''>\sum_{j\in S_i^{\text{in}}}\|a_{ij}''\|^2$ $\forall i\in\mathcal{V}^{\text{in}}$. Apparently, problem~\eqref{eq:simuprob} meets Assumption \ref{asm:prob}(a). Moreover, since $\operatorname{dom}(h_i)=X_i$ $\forall i\in\mathcal{V}$ are compact, Assumption \ref{asm:prob}(b)-(e) hold. Besides, Assumption~\ref{asm:prob}(f) is satisfied with $\tilde{x}_i=\mathbf{0}_d$ $\forall i\in \mathcal{V}$.

Since problem~\eqref{eq:simuprob} also obeys the assumptions in \cite{Liang19, Liang19a}, below we simulate the convergence performance of IPLUX, the distributed primal-dual gradient method in \cite{Liang19}, and the distributed dual subgradient method in \cite{Liang19a} when solving \eqref{eq:simuprob}. The simulation settings are as follows: Let $n=30$, $d = 5$, $\tilde{m}=3$, and $m^s=2$. We randomly generate $\mathcal{V}^{\operatorname{in}}$, $S_i^{\operatorname{in}}$ $\forall i\in\mathcal{V}^{\operatorname{in}}$, $\mathcal{V}^{\operatorname{eq}}$, and $S_i^{\operatorname{eq}}$ $\forall i\in\mathcal{V}^{\operatorname{eq}}$ satisfying $|\mathcal{V}^{\operatorname{in}}|=|\mathcal{V}^{\operatorname{eq}}|=15$, $|S_i^{\operatorname{in}}|=4$ $\forall i\in\mathcal{V}^{\operatorname{in}}$, and $|S_i^{\operatorname{eq}}|=4$ $\forall i\in\mathcal{V}^{\operatorname{eq}}$. The remaining problem data is also randomly generated under the conditions imposed above. The underlying interaction graph $\mathcal{G}$ is obtained according to \emph{Case~3} in Section \ref{ssec:graph}, and is shown in Fig.~\ref{fig:simugraph}. All the algorithm parameters are fine-tuned to achieve best possible convergence performance. Moreover, we let the three algorithms start from the same initial primal iterates for fairness.

\begin{figure}[ht]
	\centering
	\includegraphics[scale=0.45]{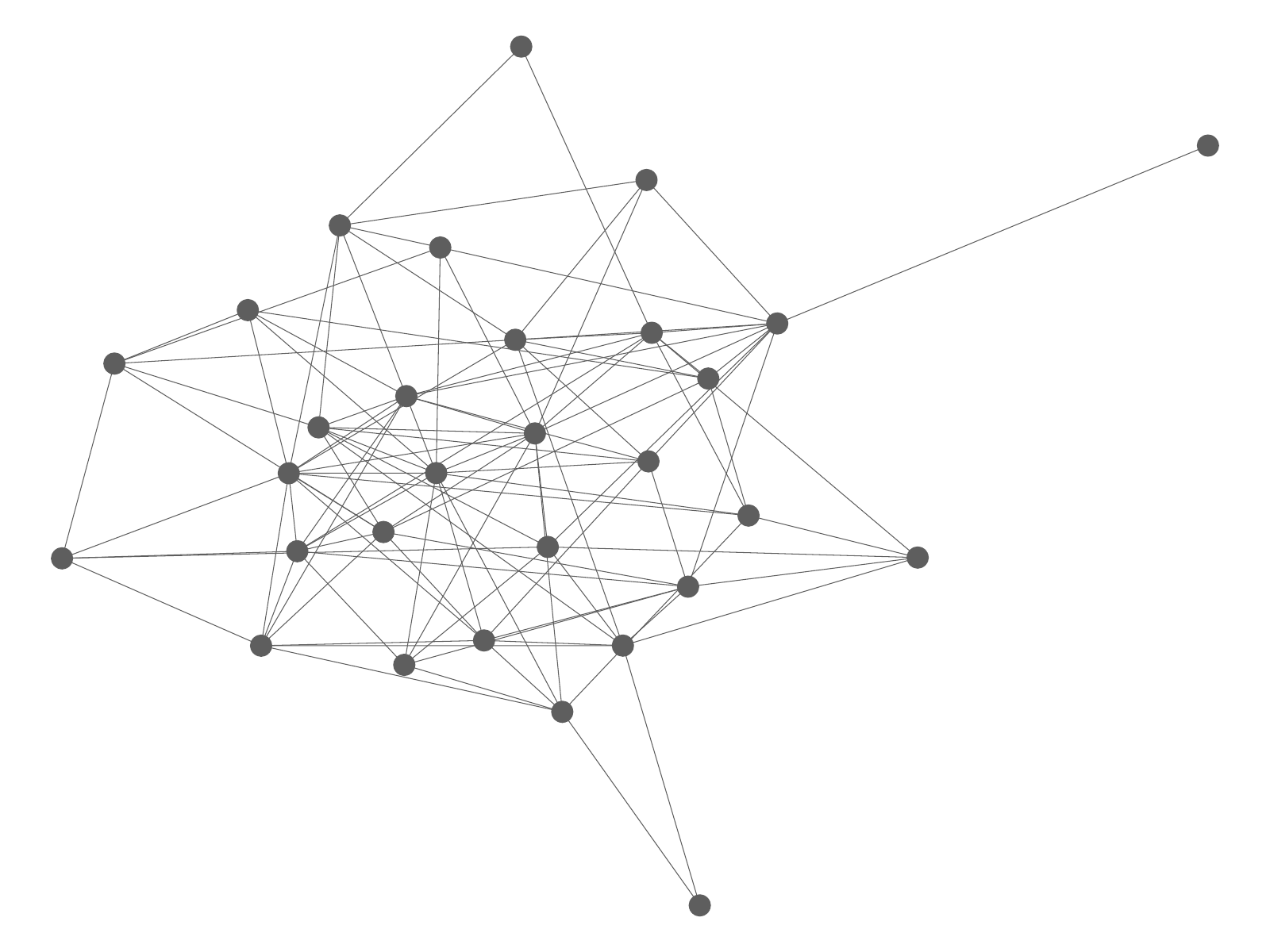}
	\caption{Underlying interaction graph $\mathcal{G}$ induced from problems~\eqref{eq:simuprob} and~\eqref{eq:simuprobnonsmooth}.}
	\label{fig:simugraph}
\end{figure}

We plot the constraint violations and the objective errors generated by the aforementioned three algorithms during $2000$ iterations in Fig.~\ref{fig:smooth}. Here, we use $w_i(k)$ $\forall i\in\mathcal{V}$ to denote the variables in these three algorithms that are theoretically guaranteed to reach primal optimality. Specifically, for the distributed primal-dual algorithm \cite{Liang19}, $w_i(k)$ is the primal iterate associated with node $i\in\mathcal{V}$ at iteration $k$. For the distributed dual subgradient method \cite{Liang19a} and our proposed IPLUX, $w_i(k)$ is the running average of the primal iterates associated with node $i\in\mathcal{V}$ (i.e., $\bar{x}_i(k)$ defined above Corollary~\ref{cor:originalprobconv} with $\bar{x}_i(0)=x_i(0)$). The constraint violation $e_1(k)+e_2(k)+e_3(k)+e_4(k)$ consists of four parts, i.e., $e_1(k)=\max\{\sum_{i=1}^n (\|w_i(k)-a_i'\|^2-c_i'), 0\}$, $e_2(k)=\|\sum_{i=1}^n A_iw_i(k)\|$, $e_3(k)=\sum_{i\in\mathcal{V}^{\operatorname{in}}}\max\{\sum_{j\in S_i^{\operatorname{in}}} (\|w_j(k)-a_{ij}''\|^2-c_{ij}''), 0\}$, and $e_4(k)=\sum_{i\in\mathcal{V}^{\operatorname{eq}}}\|\sum_{j\in S_i^{\operatorname{eq}}}A_{ij}^sw_j(k)\|$. The objective error is defined as $|\sum_{i=1}^n \left(w_i(k)^TP_iw_i(k)\!+\!Q_i^Tw_i(k)\!+\!\mathcal{I}_{X_i}(w_i(k))\right)-\Phi^\star|$, where $\Phi^\star$ is the optimal value of \eqref{eq:simuprob} calculated by CVXPY \cite{Diamond16}. Observe from Fig.~\ref{fig:smooth} that compared to the two existing methods in \cite{Liang19, Liang19a}, IPLUX presents a prominent advantage in convergence speed and accuracy with respect to both optimality and feasibility.

\begin{figure}[!ht]
	\centering
	\subfigure[Constraint violation.]{
		\includegraphics[width=8cm,height=4.8cm]{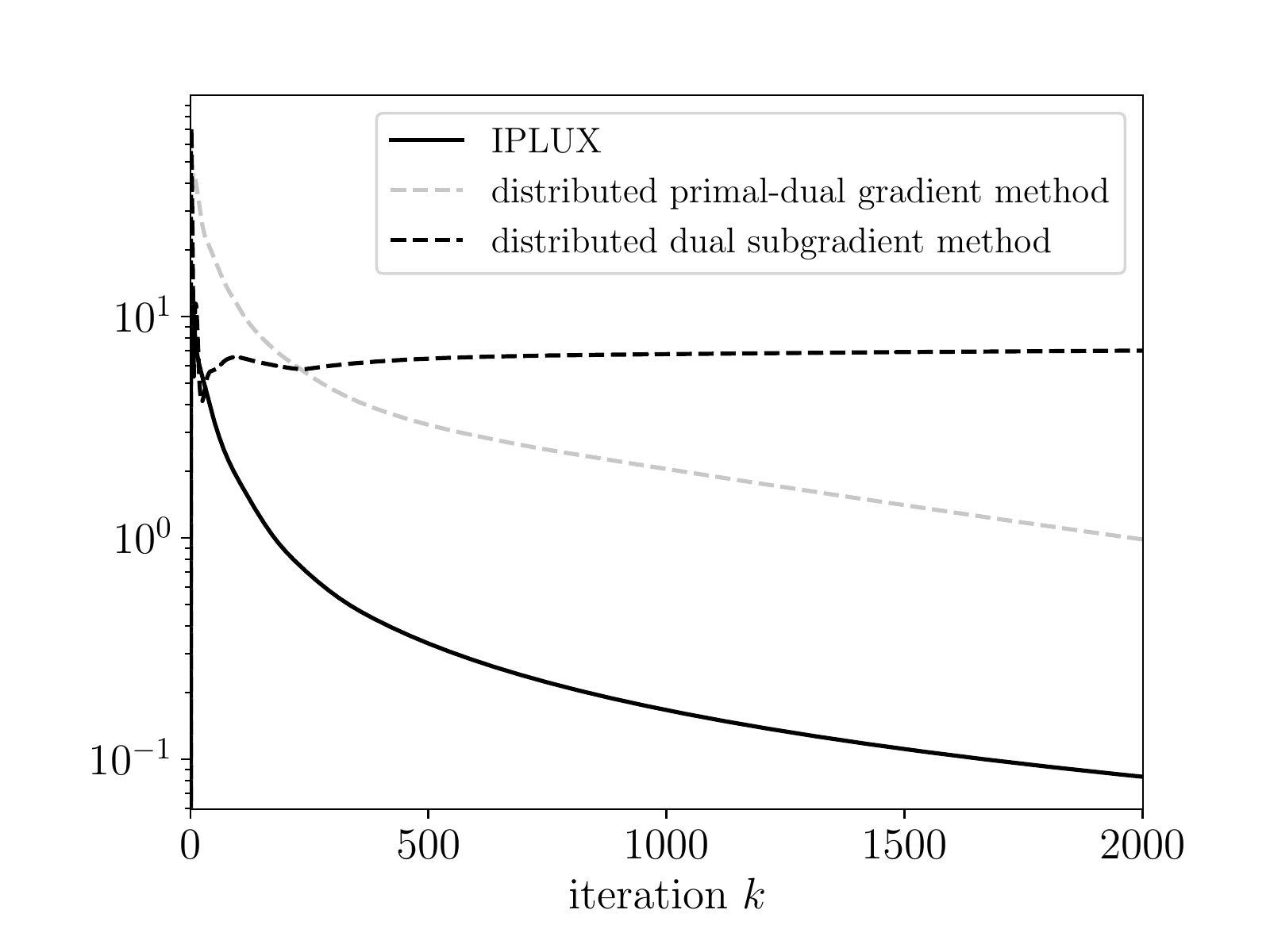}\label{fig:smoothfeasibility}}
	\vfill
	\subfigure[Objective error.]{
		\includegraphics[width=8cm,height=4.8cm]{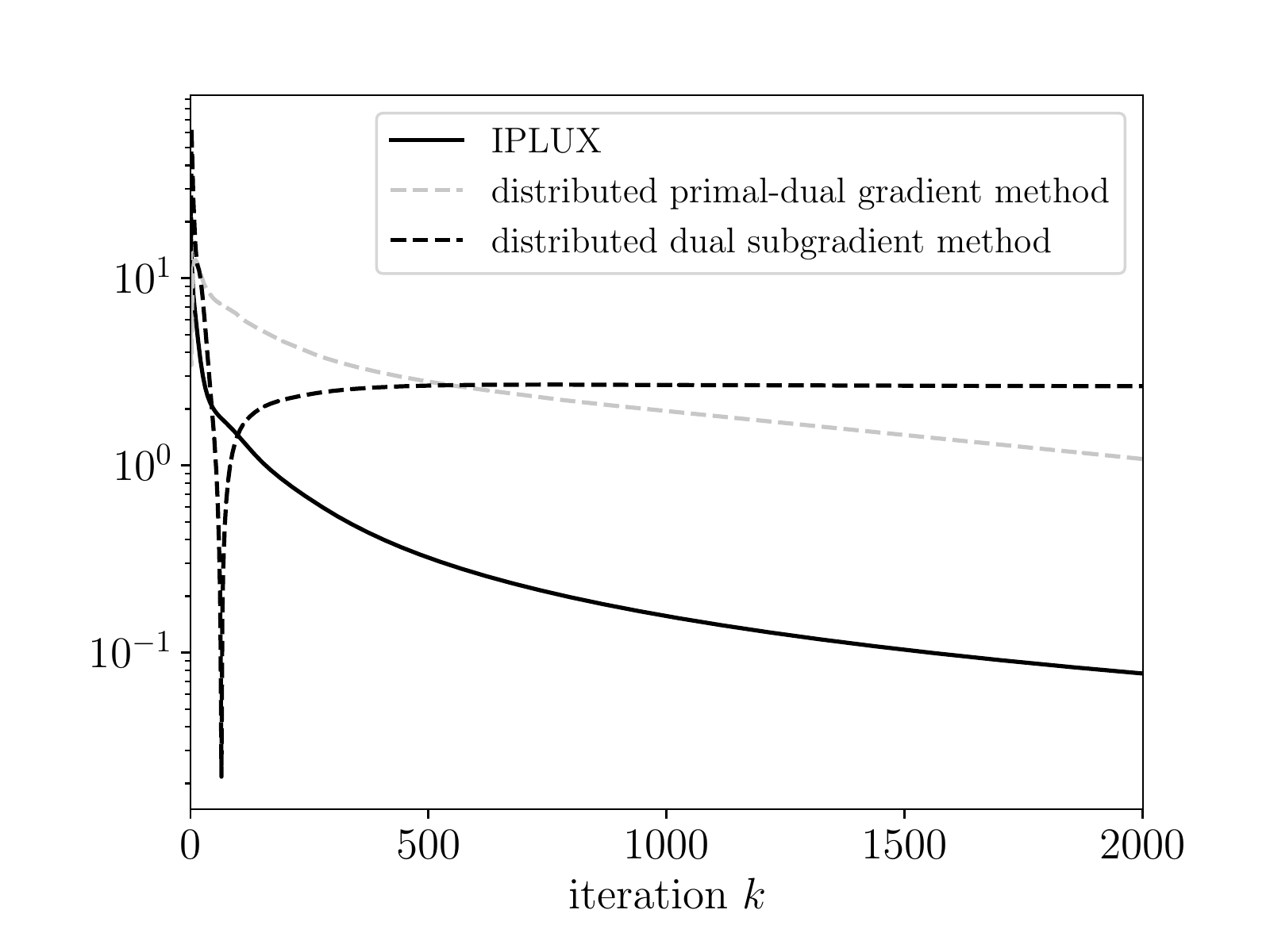}\label{fig:smoothfuncval}}
	\caption{Convergence performance of IPLUX, the distributed primal-dual gradient method \cite{Liang19}, and the distributed dual subgradient method \cite{Liang19a} in solving problem~\eqref{eq:simuprob}.}
	\label{fig:smooth}
\end{figure}

\subsection{Nonsmooth Problem}\label{ssec:nonsmooth}

Now we add the $\ell_1$-norm of each $x_i$ to the objective function of problem~\eqref{eq:simuprob}, and thus consider the following problem:
\begin{align}
\underset{x_1,\ldots,x_n\in\mathbb{R}^d}{\operatorname{minimize}} ~&~ \sum_{i=1}^n\! \left(x_i^TP_ix_i+Q_i^Tx_i+\|x_i\|_1+\mathcal{I}_{X_i}(x_i)\right)\displaybreak[0]\nonumber\\
\operatorname{subject~to}~&~ \text{All the constraints in \eqref{eq:simuprob}}.\label{eq:simuprobnonsmooth}
\end{align}
Since the existing methods in \cite{Liang19, Liang19a} have not been proven to achieve the optimality of \eqref{eq:simuprobnonsmooth}, we only implement IPLUX for this numerical example. All the settings are the same as those in Section~\ref{ssec:smooth}, except that each $h_i(x_i)$ is changed to $h_i(x_i)=\|x_i\|_1+\mathcal{I}_{X_i}(x_i)$.

Fig.~\ref{fig:nonsmooth} illustrates that the constraint violation and the objective error at $\bar{x}_i(k)$ $\forall i\in\mathcal{V}$ in solving \eqref{eq:simuprobnonsmooth} vanish quickly, validating the convergence of IPLUX established in Section~\ref{sec:convanal}.

\begin{figure}[!ht]
\centering
\includegraphics[width=8cm,height=4.8cm]{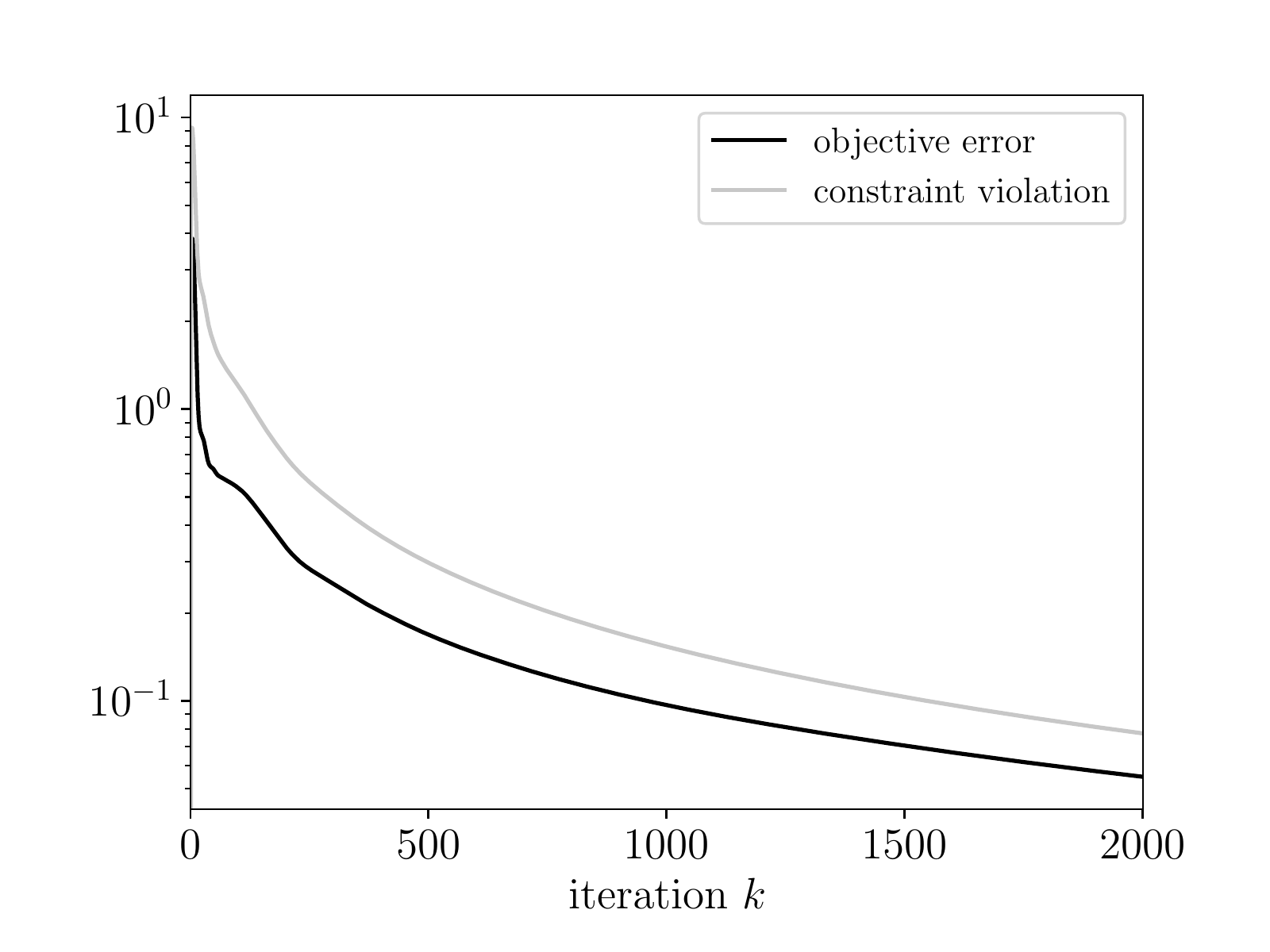}
\caption{Convergence performance of IPLUX in solving problem~\eqref{eq:simuprobnonsmooth}.}
\label{fig:nonsmooth}
\end{figure}

\section{Conclusion}\label{sec:conclusion}
We have developed the integrated primal-dual proximal (IPLUX) algorithm to solve nonsmooth convex optimization problems with both nonlinear inequality and linear equality constraints in a distributed way. Moreover, we have shown that both the objective error and the constraint violation of IPLUX dissipate at $O(1/k)$ rates. IPLUX is able to address more general problems than most prior distributed optimization algorithms, achieve stronger convergence results than the existing methods that also admit coupling nonlinear inequality constraints, and possess superior practical performance in solving constrained distributed convex optimization problems.

\appendix
\subsection{Proof of Proposition \ref{prop:propertyofproby}}\label{ssec:proofofproLipschitz}

Pick any $\mathbf{y}, \mathbf{y}'\in \operatorname{dom}(h)$. Then, they can be written as $\mathbf{y}=[y_1^T, \ldots, y_n^T]^T$ and $\mathbf{y}'=[(y_1')^T, \ldots, (y_n')^T]^T$, where for each $i\in\mathcal{V}$, $y_i=[x_i^T, t_i^T]^T$ and $y_i'=[(x_i')^T, (t_i')^T]^T$, with $x_i, x_i'\in \operatorname{dom}(h_i)$ and $t_i, t_i'\in \mathbb{R}^{\tilde{p}}$. 

First of all, since $f(\mathbf{y})=\sum_{i\in\mathcal{V}}f_i(x_i)$ and $h(\mathbf{y})=\sum_{i\in\mathcal{V}}h_i(x_i)$, Assumption~\ref{asm:prob}(a) leads to the convexity of $f$ and $h$ on $\operatorname{dom}(h)$. Since $\Phi(\mathbf{y})=f(\mathbf{y})+h(\mathbf{y})$, $\Phi$ is also convex on $\operatorname{dom}(h)$. Therefore, property (a) is satisfied.

Due to Assumption~\ref{asm:prob}(b), $\|\nabla f(\mathbf{y})-\nabla f(\mathbf{y}')\|^2=\sum_{i\in \mathcal{V}} \|\nabla f_i(x_i)-\nabla f_i(x_i')\|^2\le L_f^2\sum_{i\in \mathcal{V}} \|x_i-x_i'\|^2\le L_f^2\|\mathbf{y}-\mathbf{y}'\|^2$, i.e., property (b) holds. 

To prove property (c), note from the definition of $\mathbf{G}$ that
\begin{align}
&\|\mathbf{G}(\mathbf{y})-\mathbf{G}(\mathbf{y}')\|^2=\sum_{i\in \mathcal{V}}\|(g_i(x_i)-t_i)-(g_i(x_i')-t_i')\|^2\nonumber\displaybreak[0]\\
&\qquad\qquad+\sum_{i\in\mathcal{V}^{\operatorname{in}}}\|\sum_{j\in S_i^{\operatorname{in}}} (g_{ij}^s(x_j)-g_{ij}^s(x_j'))\|^2.\label{eq:Lipschitzpropeq1}
\end{align}
Note that for each $i\in\mathcal{V}$, $\|(g_i(x_i)-t_i)-(g_i(x_i')-t_i')\|^2=\|g_i(x_i)-g_i(x_i')\|^2+\|t_i-t_i'\|^2-2\langle g_i(x_i)-g_i(x_i'), t_i-t_i'\rangle\le(1+\frac{1}{L_g^2})\|g_i(x_i)\!-\!g_i(x_i')\|^2+(1+L_g^2)\|t_i-t_i'\|^2$. Then, using Assumption~\ref{asm:prob}(c),
\begin{align}
&\sum_{i\in \mathcal{V}}\|(g_i(x_i)-t_i)-(g_i(x_i')-t_i')\|^2\nonumber\displaybreak[0]\\
\le&\sum_{i\in \mathcal{V}}\Bigl((1+\frac{1}{L_g^2})\cdot L_g^2\|x_i-x_i'\|^2+(1+L_g^2)\|t_i\!-\!t_i'\|^2\Bigr)\nonumber\displaybreak[0]\\
=&(1+L_g^2)\|\mathbf{y}-\mathbf{y}'\|^2.\label{eq:sumgtgt<=1Lyy}
\end{align}
On the other hand, from Assumption~\ref{asm:prob}(d),
\begin{align}
&\sum_{i\in \mathcal{V}^{\operatorname{in}}}\Big\|\sum_{j\in S_i^{\operatorname{in}}}(g_{ij}^s(x_j)-g_{ij}^s(x_j'))\Big\|^2\nonumber\displaybreak[0]\\
\le&\sum_{i\in \mathcal{V}^{\operatorname{in}}}|S_i^{\operatorname{in}}|\sum_{j\in S_i^{\operatorname{in}}}\| g_{ij}^s(x_j)-g_{ij}^s(x_j')\|^2\nonumber\displaybreak[0]\\
\le&\sum_{i\in \mathcal{V}^{\operatorname{in}}}|S_i^{\operatorname{in}}|\sum_{j\in S_i^{\operatorname{in}}}L_{g^s}^2\|x_j-x_j'\|^2\nonumber\displaybreak[0]\\
=&\sum_{i\in\mathcal{V}}\sum_{j\in\mathcal{V}^{\operatorname{in}}:i\in S_j^{\operatorname{in}}}|S_j^{\operatorname{in}}|L_{g^s}^2\|x_i-x_i'\|^2\nonumber\displaybreak[0]\\
\le&\Bigl(\max_{i\in \mathcal{V}}\sum_{j\in\mathcal{V}^{\operatorname{in}}:i\in S_j^{\operatorname{in}}}|S_j^{\operatorname{in}}|\Bigr)L_{g^s}^2\|\mathbf{y}-\mathbf{y}'\|^2.\label{eq:sumsumgg<=maxCLyy}
\end{align}
Incorporating \eqref{eq:sumgtgt<=1Lyy} and \eqref{eq:sumsumgg<=maxCLyy} into \eqref{eq:Lipschitzpropeq1} leads to property (c).

Next, we prove property (d) by constructing an optimal solution $\mathbf{y}^\star$ to problem \eqref{eq:originalprobleminy}. To do so, let $\mathbf{x}^\star=[(x_1^\star)^T,\ldots,(x_n^\star)^T]^T$ be an optimum of \eqref{eq:prob}, which exists due to Assumption~\ref{asm:prob}(e). Then, let $t_i^\star=g_i(x_i^\star)-\frac{1}{n}\sum_{j\in \mathcal{V}} g_j(x_j^\star)$ $\forall i\in \mathcal{V}$, and let $\mathbf{y}^\star=[(y_1^\star)^T, \ldots, (y_n^\star)^T]^T$ with $y_i^\star=[(x_i^\star)^T, (t_i^\star)^T]^T$ $\forall i\in \mathcal{V}$. Since $\mathbf{x}^\star$ satisfies constraint (2a) in \eqref{eq:prob}, we have $t_i^\star\ge g_i(x_i^\star)$. Also observe that $\sum_{i\in\mathcal{V}}t_i^\star=\mathbf{0}_{\tilde{p}}$. It follows that $\mathbf{y}^\star$ satisfies \eqref{eq:g<=t}, \eqref{eq:sumt=0}, and (2b)-(2d) in problem~\eqref{eq:prob}. As is shown in Section~\ref{ssec:probtrans}, these constraints are equivalent to the constraints in \eqref{eq:originalprobleminy}, i.e., $\mathbf{y}^\star$ is feasible to problem~\eqref{eq:originalprobleminy}. Moreover, because $\Phi(\mathbf{y}^\star)=\sum_{i\in \mathcal{V}} (f_i(x_i^\star)+h_i(x_i^\star))$ which means $\mathbf{y}^\star$ attains the optimal value of problem~\eqref{eq:prob} and because of the equivalence between \eqref{eq:prob} and \eqref{eq:originalprobleminy}, $\mathbf{y}^\star$ is an optimum to \eqref{eq:originalprobleminy}.

Finally, we prove property (e) by showing that the Slater's condition holds for problem~\eqref{eq:originalprobleminy}. Let $\tilde{\mathbf{x}}=[(\tilde{x}_1)^T,\ldots,(\tilde{x}_n)^T]^T$ be a Slater's point as is described in Assumption~\ref{asm:prob}(f). Then, let $\tilde{t}_i=g_i(\tilde{x}_i)-\frac{1}{n}\sum_{j\in \mathcal{V}} g_j(\tilde{x}_j)$ $\forall i\in \mathcal{V}$, $\tilde{y}_i=[(\tilde{x}_i)^T, (\tilde{t}_i)^T]^T$ $\forall i\in \mathcal{V}$, and $\tilde{\mathbf{y}}=[(\tilde{y}_1)^T, \ldots, (\tilde{y}_n)^T]^T$. Similar to the proof of property (d) above, it can be shown that $\tilde{\mathbf{y}}\in \operatorname{rel\;int}(\operatorname{dom}(h))$ is a Slater's point of \eqref{eq:originalprobleminy}, so that property (e) holds \cite{Bertsekas99}.

\subsection{Proof of Lemma \ref{lemma:boundedsumF}}\label{ssec:proofoflemmaboundedsumF}

To prove \eqref{eq:sumPhiykystarupperbound}, it suffices to show that
\begin{equation}\label{eq:FandRk}
\Phi(\mathbf{y}(\ell+1))-\Phi(\mathbf{y}^\star) \le R(\ell)-R(\ell+1),\quad\forall \ell\ge 0.
\end{equation}
Below, we first derive an upper bound on $\Phi(\mathbf{y}(\ell+1))-\Phi(\mathbf{y}^\star)=f(\mathbf{y}(\ell+1))+h(\mathbf{y}(\ell+1))-f(\mathbf{y}^\star)-h(\mathbf{y}^\star)$, and then show that this upper bound cannot exceed $R(\ell)-R(\ell+1)$ due to the condition on $\alpha$. 

For convenience, we define $\tilde{\Delta}_1^\ell(\mathbf{y})=\langle\mathbf{v}(\ell), \mathbf{y}\rangle+\frac{\gamma\lambda^2}{2}\|\mathbf{y}-\mathbf{y}(\ell)+\frac{1}{\lambda^2}(\mathbf{B}^s)^T(\mathbf{B}^s\mathbf{y}(\ell)-\mathbf{c}^s)\|^2$, $\tilde{\Delta}_2^\ell(\mathbf{y})=\langle W\mathbf{u}(\ell)-\frac{1}{\rho}\mathbf{z}(\ell), \mathbf{B}\mathbf{y}-\mathbf{c}\rangle+\frac{1}{2\rho}\|\mathbf{B}\mathbf{y}-\mathbf{c}\|^2$, and $\tilde{\Delta}_3^\ell(\mathbf{y})=\langle\mathbf{q}(\ell)+\mathbf{G}(\mathbf{y}(\ell)),\mathbf{G}(\mathbf{y})\rangle+\frac{\alpha}{2}\|\mathbf{y}-\mathbf{y}(\ell)\|^2$ for each $\ell\ge 0$. Thus, $\mathcal{L}^{\ell}(\mathbf{y})$ defined below \eqref{eq:finalupdatey} can be rewritten as 
\begin{align}\label{eq:L=nablafyyhsumtildeDelta}
\mathcal{L}^{\ell}(\mathbf{y})=\langle \nabla f(\mathbf{y}(\ell)), \mathbf{y}-\mathbf{y}(\ell)\rangle+h(\mathbf{y})+\sum_{i=1}^3\tilde{\Delta}_i^\ell(\mathbf{y}).
\end{align}
Note that $\mathcal{L}^\ell(\mathbf{y})-\frac{1}{2\rho}\|\mathbf{B}\mathbf{y}-\mathbf{c}\|^2$ is strongly convex with convexity parameter $\alpha+\gamma\lambda^2>0$. Also note that $\mathbf{0}\in \partial \mathcal{L}^\ell(\mathbf{y}(\ell+1))$ because of \eqref{eq:finalupdatey}. It follows that
\begin{align}
&(\mathcal{L}^\ell(\mathbf{y}^\star)-\frac{1}{2\rho}\|\mathbf{B}\mathbf{y}^\star-\mathbf{c}\|^2)\nonumber\displaybreak[0]\\
&-(\mathcal{L}^\ell(\mathbf{y}(\ell+1))-\frac{1}{2\rho}\|\mathbf{B}\mathbf{y}(\ell+1)-\mathbf{c}\|^2)\nonumber\displaybreak[0]\\
\ge&\langle -\frac{1}{\rho}\mathbf{B}^T(\mathbf{B}\mathbf{y}(\ell+1)-\mathbf{c}), \mathbf{y}^\star-\mathbf{y}(\ell+1)\rangle\nonumber\displaybreak[0]\\
&+\frac{\alpha+\gamma\lambda^2}{2}\|\mathbf{y}^\star-\mathbf{y}(\ell+1)\|^2.\label{eq:L12rhoBycL12rhoByc<=}
\end{align}
From Proposition~\ref{prop:propertyofproby}(b), $\nabla f$ is Lipschitz continuous on $\operatorname{dom}(h)$. Also, $f$ is convex on $\operatorname{dom}(h)$ due to Assumption~\ref{asm:prob}(a), and $\mathbf{y}(\ell),\mathbf{y}(\ell+1),\mathbf{y}^\star\in\operatorname{dom}(h)$. It follows that
\begin{align*}
f(\mathbf{y}(\ell+1)) \le & f(\mathbf{y}(\ell))+\langle \nabla f(\mathbf{y}(\ell)), \mathbf{y}(\ell+1)-\mathbf{y}(\ell)\rangle\\
&+\frac{L_f}{2}\|\mathbf{y}(\ell+1)-\mathbf{y}(\ell)\|^2,\displaybreak[0]\\
-f(\mathbf{y}^\star)\le& -f(\mathbf{y}(\ell))-\langle \nabla f(\mathbf{y}(\ell)), \mathbf{y}^\star-\mathbf{y}(\ell)\rangle.
\end{align*}
Adding the above two inequalities yields
\begin{align}
&\langle \nabla f(\mathbf{y}(\ell)), \mathbf{y}(\ell+1)-\mathbf{y}^\star\rangle\nonumber\displaybreak[0]\\
\ge& f(\mathbf{y}(\ell+1))-f(\mathbf{y}^\star)-\frac{L_f}{2}\|\mathbf{y}(\ell+1)-\mathbf{y}(\ell)\|^2.\label{eq:fk1minusfstar}
\end{align}
By incorporating \eqref{eq:fk1minusfstar} and \eqref{eq:L=nablafyyhsumtildeDelta} into \eqref{eq:L12rhoBycL12rhoByc<=}, we obtain
\begin{align}
& \Phi(\mathbf{y}(\ell+1))-\Phi(\mathbf{y}^\star)\le\sum_{i=1}^3\Bigl(\tilde{\Delta}_i^\ell(\mathbf{y}^\star)-\tilde{\Delta}_i^\ell(\mathbf{y}(\ell+1))\Bigr)\nonumber\displaybreak[0]\\
&+\frac{1}{2\rho}\|\mathbf{B}\mathbf{y}(\ell+1)-\mathbf{c}\|^2-\frac{1}{2\rho}\|\mathbf{B}\mathbf{y}^\star-\mathbf{c}\|^2\nonumber\displaybreak[0]\\
&+\frac{1}{\rho}\langle\mathbf{B}^T(\mathbf{B}\mathbf{y}(\ell+1)-\mathbf{c}), \mathbf{y}^\star-\mathbf{y}(\ell+1)\rangle\nonumber\displaybreak[0]\\
&-\frac{\alpha\!+\!\gamma\lambda^2}{2}\|\mathbf{y}^\star\!-\!\mathbf{y}(\ell+1)\|^2\!+\!\frac{L_f}{2}\|\mathbf{y}(\ell+1)-\mathbf{y}(\ell)\|^2.\label{eq:Phiyk1ystar}
\end{align}

Subsequently, we bound $\tilde{\Delta}_i^\ell(\mathbf{y}^\star)-\tilde{\Delta}_i^\ell(\mathbf{y}(\ell+1))$, $\forall i=1,2,3$ in \eqref{eq:Phiyk1ystar}. To do so, we introduce the following lemmas.

\begin{lemma}\label{lem:matrixlemma}
For any symmetric positive semidefinite matrix $M$ and vectors $\mathnormal{a}', \mathnormal{a}''$ of proper dimensions, 
\begin{equation}\label{eq:matrixlemma1}
\langle M(\mathnormal{a}'-\mathnormal{a}''),\mathnormal{a}''\rangle\le\frac{1}{2}(\|\mathnormal{a}'\|_{M}^2-\|\mathnormal{a}''\|_{M}^2).
\end{equation}
If, in addition, $\mathnormal{a}'\in\operatorname{Range}(M)$, then
\begin{equation}\label{eq:matrixlemma2}
\langle \mathnormal{a}',\mathnormal{a}''\rangle=\langle \mathnormal{a}',M^\dag M\mathnormal{a}''\rangle.
\end{equation}
\end{lemma}

\begin{proof}
Since $\langle M(\mathnormal{a}'-\mathnormal{a}''),\mathnormal{a}''\rangle=\frac{1}{2}(\|\mathnormal{a}'\|_{M}^2-\|\mathnormal{a}''\|_{M}^2-\|\mathnormal{a}''-\mathnormal{a}'\|_{M}^2)$, we obtain \eqref{eq:matrixlemma1}. Also, \eqref{eq:matrixlemma2} holds because of $\mathnormal{a}'=M^\dag M\mathnormal{a}'$ \cite{yanai11} and because $M^\dag M$ is symmetric.
\end{proof}

\begin{lemma}\label{lemma:qkproperties}
For each $\ell\ge0$,
\begin{align}
&\mathbf{q}(\ell)\ge\mathbf{0},\nonumber\displaybreak[0]\\
&\|\mathbf{q}(\ell+1)\|\ge\|\mathbf{G}(\mathbf{y}(\ell+1))\|,\label{eq:qkplus1normlargethang}\displaybreak[0]\\
&\langle\mathbf{q}(\ell),\mathbf{G}(\mathbf{y}(\ell+1))\rangle\ge\frac{1}{2}(\|\mathbf{q}(\ell+1)\|^2-\|\mathbf{q}(\ell)\|^2)\nonumber\displaybreak[0]\\
&\qquad\qquad\qquad\qquad\qquad-\|\mathbf{G}(\mathbf{y}(\ell+1))\|^2.\label{eq:successivepropertyofqk}
\end{align} 
\end{lemma}

\begin{proof}
The proof is very similar to the proofs of \cite[Lemmas~3 and~4]{YuH17}, and therefore is omitted.
\end{proof}

\subsubsection{Upper bound on $\tilde{\Delta}_1^\ell(\mathbf{y}^\star)-\tilde{\Delta}_1^\ell(\mathbf{y}(\ell+1))$}

Let $\mathnormal{a}'=\mathbf{v}(\ell)$, $\mathnormal{a}''=\mathbf{y}^\star-\mathbf{y}(\ell+1)$, and $M=(\mathbf{B}^s)^T\mathbf{B}^s$ in Lemma~\ref{lem:matrixlemma}. Recall from Section~\ref{ssec:proximalsparse} that $\mathbf{v}(\ell)\in\operatorname{Range}((\mathbf{B}^s)^T)$. Since $\operatorname{Range}((\mathbf{B}^s)^T)=\operatorname{Range}(M)$, we guarantee $\mathnormal{a}'\in\operatorname{Range}(M)$. Hence, from \eqref{eq:matrixlemma2}, 
\begin{align}
&\langle \mathbf{v}(\ell), \mathbf{y}^\star\rangle-\langle \mathbf{v}(\ell), \mathbf{y}(\ell+1)\rangle=\langle \mathbf{v}(\ell), \mathbf{y}^\star-\mathbf{y}(\ell+1)\rangle\nonumber\displaybreak[0]\\
&=\langle\mathbf{v}(\ell), ((\mathbf{B}^s)^T\mathbf{B}^s)^\dag(\mathbf{B}^s)^T\mathbf{B}^s(\mathbf{y}^\star-\mathbf{y}(\ell+1))\rangle.\label{eq:vyvy=BBvBByy}
\end{align}
Moreover, because of \eqref{eq:sparsetildevupdate} and because $\mathbf{B}^s\mathbf{y}^\star=\mathbf{c}^s$, 
\begin{align}
(\mathbf{B}^s)^T\mathbf{B}^s(\mathbf{y}^\star-\mathbf{y}(\ell+1))=\frac{1}{\gamma}(\mathbf{v}(\ell)-\mathbf{v}(\ell+1)).\label{eq:BByy=1gammavv}
\end{align}
By substituting \eqref{eq:BByy=1gammavv} into \eqref{eq:vyvy=BBvBByy}, we obtain
\begin{align}\label{eq:tildevystaryk1}
&\langle \mathbf{v}(\ell), \mathbf{y}^\star\rangle-\langle \mathbf{v}(\ell), \mathbf{y}(\ell+1)\rangle\nonumber\displaybreak[0]\\
=& \frac{1}{\gamma}\langle \mathbf{v}(\ell), ((\mathbf{B}^s)^T\mathbf{B}^s)^\dag(\mathbf{v}(\ell)-\mathbf{v}(\ell+1))\rangle\nonumber\displaybreak[0]\\
=&\frac{1}{2\gamma}(\|\mathbf{v}(\ell)\|_{((\mathbf{B}^s)^T\mathbf{B}^s)^\dag}^2-\|\mathbf{v}(\ell+1)\|_{((\mathbf{B}^s)^T\mathbf{B}^s)^\dag}^2)\nonumber\displaybreak[0]\\
&+\frac{1}{2\gamma}\|\mathbf{v}(\ell)-\mathbf{v}(\ell+1)\|_{((\mathbf{B}^s)^T\mathbf{B}^s)^\dag}^2\nonumber\displaybreak[0]\\
=&\frac{1}{2\gamma}(\|\mathbf{v}(\ell)\|_{((\mathbf{B}^s)^T\mathbf{B}^s)^\dag}^2-\|\mathbf{v}(\ell+1)\|_{((\mathbf{B}^s)^T\mathbf{B}^s)^\dag}^2)\nonumber\displaybreak[0]\\
&+\frac{\gamma}{2}\|\mathbf{y}(\ell+1)-\mathbf{y}^\star\|_{(\mathbf{B}^s)^T\mathbf{B}^s}^2,
\end{align}
where the last step is due again to \eqref{eq:BByy=1gammavv}. 

On the other hand, let $\tilde{\mathbf{B}}:=I_N-\frac{1}{\lambda^2}(\mathbf{B}^s)^T\mathbf{B}^s$ for convenience. Since $\lambda\ge \|\mathbf{B}^s\|_2$, $I_N\succeq\tilde{\mathbf{B}}\succeq \mathbf{O}$. Also, because $\mathbf{B}^s\mathbf{y}^\star=\mathbf{c}^s$, we have $\mathbf{y}^\star-\mathbf{y}(\ell)+\frac{1}{\lambda^2}(\mathbf{B}^s)^T(\mathbf{B}^s\mathbf{y}(\ell)-\mathbf{c}^s)=\tilde{\mathbf{B}}(\mathbf{y}^\star-\mathbf{y}(\ell))$. It follows that
\begin{align}
&\|\mathbf{y}^\star-\mathbf{y}(\ell)+\frac{1}{\lambda^2}(\mathbf{B}^s)^T(\mathbf{B}^s\mathbf{y}(\ell)-\mathbf{c}^s)\|^2\nonumber\displaybreak[0]\\
&-\|\mathbf{y}(\ell\!+\!1)\!-\!\mathbf{y}(\ell)\!+\!\frac{1}{\lambda^2}(\mathbf{B}^s)^T(\mathbf{B}^s\mathbf{y}(\ell)\!-\!\mathbf{c}^s)\|^2\nonumber\displaybreak[0]\\
=&-\|\mathbf{y}(\ell+1)-\mathbf{y}^\star\|^2-2\langle \mathbf{y}(\ell+1)-\mathbf{y}^\star,\tilde{\mathbf{B}}(\mathbf{y}^\star-\mathbf{y}(\ell))\rangle\nonumber\displaybreak[0]\\
\le&-\|\mathbf{y}(\ell+1)-\mathbf{y}^\star\|^2+\|\mathbf{y}(\ell+1)-\mathbf{y}^\star\|_{\tilde{\mathbf{B}}}^2+\|\mathbf{y}(\ell)-\mathbf{y}^\star\|_{\tilde{\mathbf{B}}}^2\nonumber\displaybreak[0]\\
\le&-\frac{1}{\lambda^2}\|\mathbf{y}(\ell+1)-\mathbf{y}^\star\|_{(\mathbf{B}^s)^T\mathbf{B}^s}^2+\|\mathbf{y}(\ell)-\mathbf{y}^\star\|^2.\label{eq:ystarykyk1yk}
\end{align}

Finally, combining \eqref{eq:tildevystaryk1} and \eqref{eq:ystarykyk1yk} gives 
\begin{align}
&\tilde{\Delta}_1^\ell(\mathbf{y}^\star)-\tilde{\Delta}_1^\ell(\mathbf{y}(\ell+1))\le \frac{\gamma\lambda^2}{2}\|\mathbf{y}(\ell)-\mathbf{y}^\star\|^2\nonumber\displaybreak[0]\\
&+\frac{1}{2\gamma}(\|\mathbf{v}(\ell)\|_{((\mathbf{B}^s)^T\mathbf{B}^s)^\dag}^2\!-\!\|\mathbf{v}(\ell+1)\|_{((\mathbf{B}^s)^T\mathbf{B}^s)^\dag}^2).\label{eq:M2starminusM2yk1upperbound}
\end{align}
		
\subsubsection{Upper bound on $\tilde{\Delta}_2^\ell(\mathbf{y}^\star)-\tilde{\Delta}_2^\ell(\mathbf{y}(\ell+1))$}		
From \eqref{eq:relationofdualandprimal},
\begin{align}
&\langle W\mathbf{u}(\ell)-\frac{1}{\rho}\mathbf{z}(\ell), \mathbf{B}\mathbf{y}^\star-\mathbf{c}\rangle\nonumber\displaybreak[0]\\
&-\langle W\mathbf{u}(\ell)-\frac{1}{\rho}\mathbf{z}(\ell), \mathbf{B}\mathbf{y}(\ell+1)-\mathbf{c}\rangle\nonumber\displaybreak[0]\\
=&\langle W\mathbf{u}(\ell)-\frac{1}{\rho}\mathbf{z}(\ell), \mathbf{B}(\mathbf{y}^\star-\mathbf{y}(\ell+1))\rangle\nonumber\displaybreak[0]\\
=&\langle \mathbf{u}(\ell+1)-\frac{1}{\rho}(\mathbf{B}\mathbf{y}(\ell+1)-\mathbf{c}), \mathbf{B}(\mathbf{y}^\star-\mathbf{y}(\ell+1))\rangle.\label{eq:WurhozByy=}
\end{align}
Below, we bound the term $\langle \mathbf{u}(\ell+1), \mathbf{B}(\mathbf{y}^\star-\mathbf{y}(\ell+1))\rangle$ on the right-hand side of \eqref{eq:WurhozByy=}. To this end, note from \eqref{eq:relationofdualandprimal}, \eqref{eq:pextrazupdate}, and $\mathbf{z}^\star=\mathbf{B}\mathbf{y}^\star-\mathbf{c}$ that
\begin{align}
&\mathbf{B}(\mathbf{y}^\star\!-\!\mathbf{y}(\ell+1)) \!=\! \mathbf{z}^\star-\mathbf{z}(\ell)+\rho(W\mathbf{u}(\ell)\!-\!\mathbf{u}(\ell+1))\nonumber\displaybreak[0]\\
&=\mathbf{z}^\star-\mathbf{z}(\ell+1)+\rho W(\mathbf{u}(\ell)-\mathbf{u}(\ell+1))\nonumber\displaybreak[0]\\
&\quad-\rho(I_{n(\tilde{m}+\tilde{p})}-W-H)\mathbf{u}(\ell+1).\nonumber
\end{align}
Due to this and \eqref{eq:WHsumsmallerthanid},
\begin{align}
&\langle \mathbf{u}(\ell+1), \mathbf{B}(\mathbf{y}^\star-\mathbf{y}(\ell+1))\rangle\nonumber\displaybreak[0]\\
\le &\langle \mathbf{u}(\ell+1), \mathbf{z}^\star-\mathbf{z}(\ell+1)\rangle\nonumber\displaybreak[0]\\
&+\rho\langle \mathbf{u}(\ell+1), W(\mathbf{u}(\ell)-\mathbf{u}(\ell+1))\rangle.\label{eq:uk1Bystark1}
\end{align}
Note from \eqref{eq:pextrazupdate} and $\mathbf{z}(0)=\rho H\mathbf{u}(0)$ that $\mathbf{z}(\ell+1)\in\operatorname{Range}(H)$. Moreover, because $\mathbf{y}^\star$ is feasible to problem~\eqref{eq:originalprobleminy}, we have $(\mathbf{1}\otimes I_{\tilde{m}+\tilde{p}})^T(\mathbf{B}\mathbf{y}^\star-\mathbf{c})=\mathbf{0}$, which, together with \eqref{eq:rangeH}, implies that $\mathbf{z}^\star\in\operatorname{Range}(H)$. Thus, by letting $\mathnormal{a}'=\mathbf{z}^\star-\mathbf{z}(\ell+1)\in \operatorname{Range}(H)$, $\mathnormal{a}''=\mathbf{u}(\ell+1)$, and $M=H$ satisfying \eqref{eq:WHsymmetricity} in Lemma~\ref{lem:matrixlemma} and by utilizing \eqref{eq:matrixlemma2} and \eqref{eq:pextrazupdate},
\begin{align}
&\langle\mathbf{u}(\ell\!+\!1), \mathbf{z}^\star\!-\!\mathbf{z}(\ell\!+\!1)\rangle\!=\!\langle \mathbf{z}^\star\!-\!\mathbf{z}(\ell\!+\!1), H^\dag H\mathbf{u}(\ell\!+\!1)\rangle\nonumber\displaybreak[0]\\
&=\frac{1}{\rho}\langle H^\dag(\mathbf{z}(\ell+1)-\mathbf{z}(\ell)), \mathbf{z}^\star-\mathbf{z}(\ell+1)\rangle.\label{eq:uk1zstarzk1}
\end{align}
Furthermore, due to \eqref{eq:WHsymmetricity}, $H^\dag=(H^\dag)^T\succeq \mathbf{O}$. Now let $\mathnormal{a}'=\mathbf{z}^\star-\mathbf{z}(\ell)$, $\mathnormal{a}''=\mathbf{z}^\star-\mathbf{z}(\ell+1)$, and $M=H^\dag$ in Lemma~\ref{lem:matrixlemma}. Thus, \eqref{eq:matrixlemma1} leads to
$\langle H^\dag(\mathbf{z}(\ell+1)-\mathbf{z}(\ell)), \mathbf{z}^\star-\mathbf{z}(\ell+1)\rangle\le\frac{1}{2}(\|\mathbf{z}^\star-\mathbf{z}(\ell)\|_{H^\dag}^2-\|\mathbf{z}^\star-\mathbf{z}(\ell+1)\|_{H^\dag}^2)$. By combining this with \eqref{eq:uk1zstarzk1}, we have
\begin{align}
&\langle\mathbf{u}(\ell+1), \mathbf{z}^\star-\mathbf{z}(\ell\!+\!1)\rangle\nonumber\displaybreak[0]\\
\le&\frac{1}{2\rho}(\|\mathbf{z}^\star-\mathbf{z}(\ell)\|_{H^\dag}^2-\|\mathbf{z}^\star-\mathbf{z}(\ell+1)\|_{H^\dag}^2).\label{eq:uzz<=12rhozzzz}
\end{align}
On the other hand, we set $\mathnormal{a}'=\mathbf{u}(\ell)$, $\mathnormal{a}''=\mathbf{u}(\ell+1)$, and $M=W$ which satisfies \eqref{eq:WHsymmetricity} in Lemma~\ref{lem:matrixlemma}. Then, according to \eqref{eq:matrixlemma1}, $\langle \mathbf{u}(\ell+1), W(\mathbf{u}(\ell)-\mathbf{u}(\ell+1))\rangle\le\frac{1}{2}(\|\mathbf{u}(\ell)\|_W^2-\|\mathbf{u}(\ell+1)\|_W^2)$. By substituting this and \eqref{eq:uzz<=12rhozzzz} into \eqref{eq:uk1Bystark1},
\begin{align*}
&\langle \mathbf{u}(\ell+1), \mathbf{B}(\mathbf{y}^\star-\mathbf{y}(\ell+1))\rangle\nonumber\displaybreak[0]\\
\le&\frac{1}{2\rho}(\|\mathbf{z}^\star-\mathbf{z}(\ell)\|_{H^\dag}^2-\|\mathbf{z}^\star-\mathbf{z}(\ell+1)\|_{H^\dag}^2)\nonumber\displaybreak[0]\\
&+\frac{\rho}{2}(\|\mathbf{u}(\ell)\|_W^2-\|\mathbf{u}(\ell+1)\|_W^2).
\end{align*}
This, along with \eqref{eq:WurhozByy=}, leads to
\begin{align}
&\tilde{\Delta}_2^\ell(\mathbf{y}^\star)-\tilde{\Delta}_2^\ell(\mathbf{y}(\ell+1))\nonumber\displaybreak[0]\\
\le&\frac{1}{2\rho}(\|\mathbf{z}^\star-\mathbf{z}(\ell)\|_{H^\dag}^2-\|\mathbf{z}^\star-\mathbf{z}(\ell+1)\|_{H^\dag}^2)\nonumber\displaybreak[0]\\
&+\frac{\rho}{2}(\|\mathbf{u}(\ell)\|_W^2-\|\mathbf{u}(\ell+1)\|_W^2)\nonumber\displaybreak[0]\\
&-\frac{1}{\rho}\langle \mathbf{B}\mathbf{y}(\ell+1)-\mathbf{c}, \mathbf{B}(\mathbf{y}^\star-\mathbf{y}(\ell+1))\rangle\nonumber\displaybreak[0]\\
&+\frac{1}{2\rho}\|\mathbf{B}\mathbf{y}^\star-\mathbf{c}\|^2-\frac{1}{2\rho}\|\mathbf{B}\mathbf{y}(\ell+1)-\mathbf{c}\|^2.\label{eq:M1starminusM1yk1upperbound}
\end{align}
		
\subsubsection{Upper bound on $\tilde{\Delta}_3^\ell(\mathbf{y}^\star)-\tilde{\Delta}_3^\ell(\mathbf{y}(\ell+1))$}

From \eqref{eq:virtualqueueq} and the initialization $\mathbf{q}(0)=\max\{\mathbf{0}, -\mathbf{G}(\mathbf{y}(0))\}$, we have $\mathbf{q}(\ell)\ge -\mathbf{G}(\mathbf{y}(\ell))$, which, together with $\mathbf{G}(\mathbf{y}^\star)\le \mathbf{0}$, implies that $\langle \mathbf{q}(\ell)+\mathbf{G}(\mathbf{y}(\ell)), \mathbf{G}(\mathbf{y}^\star)\rangle \le 0$. As a result,
\begin{align}
&\!\!\!\langle\mathbf{q}(\ell)+\mathbf{G}(\mathbf{y}(\ell)), \mathbf{G}(\mathbf{y}^\star)\rangle-\langle\mathbf{q}(\ell)+\mathbf{G}(\mathbf{y}(\ell)), \mathbf{G}(\mathbf{y}(\ell+1))\rangle\nonumber\displaybreak[0]\\
&\le-\langle\mathbf{q}(\ell), \mathbf{G}(\mathbf{y}(\ell+1))\rangle-\langle\mathbf{G}(\mathbf{y}(\ell)), \mathbf{G}(\mathbf{y}(\ell+1))\rangle.\label{eq:qGG-qGG<=}
\end{align}
From the Lipschitz continuity of $\mathbf{G}$ in Proposition \ref{prop:propertyofproby}(c), $\langle\mathbf{G}(\mathbf{y}(\ell)), \mathbf{G}(\mathbf{y}(\ell+1))\rangle=\frac{1}{2}(\|\mathbf{G}(\mathbf{y}(\ell))\|^2+\|\mathbf{G}(\mathbf{y}(\ell+1))\|^2)-\frac{1}{2}\|\mathbf{G}(\mathbf{y}(\ell+1))-\mathbf{G}(\mathbf{y}(\ell))\|^2\ge\frac{1}{2}(\|\mathbf{G}(\mathbf{y}(\ell))\|^2+\|\mathbf{G}(\mathbf{y}(\ell+1))\|^2)-\frac{L^2}{2}\|\mathbf{y}(\ell+1)-\mathbf{y}(\ell)\|^2$. Combining this and \eqref{eq:successivepropertyofqk} with \eqref{eq:qGG-qGG<=} results in
\begin{align}
&\tilde{\Delta}_3^\ell(\mathbf{y}^\star)\!-\!\tilde{\Delta}_3^\ell(\mathbf{y}(\ell+1))\!\le\! \frac{1}{2}(\|\mathbf{q}(\ell)\|^2\!-\!\|\mathbf{q}(\ell+1)\|^2)\nonumber\displaybreak[0]\\
&+\frac{\alpha}{2}\|\mathbf{y}^\star\!-\!\mathbf{y}(\ell)\|^2\!+\!\frac{1}{2}(\|\mathbf{G}(\mathbf{y}(\ell\!+\!1))\|^2\!-\!\|\mathbf{G}(\mathbf{y}(\ell))\|^2)\nonumber\displaybreak[0]\\
&+\frac{L^2-\alpha}{2}\|\mathbf{y}(\ell+1)\!-\!\mathbf{y}(\ell)\|^2.\label{eq:M3starminusM3yk1upperbound}
\end{align}

Now with the upper bounds on $\tilde{\Delta}_i^\ell(\mathbf{y}^\star)-\tilde{\Delta}_i^\ell(\mathbf{y}(\ell+1))$ $\forall i=1,2,3$ in \eqref{eq:M2starminusM2yk1upperbound}, \eqref{eq:M1starminusM1yk1upperbound}, and \eqref{eq:M3starminusM3yk1upperbound}, it can be shown that the right-hand side of \eqref{eq:Phiyk1ystar} is less than or equal to $R(\ell)-R(\ell+1)+\frac{L_f+L^2-\alpha}{2}\|\mathbf{y}(\ell+1)-\mathbf{y}(\ell)\|^2$. Moreover, due to the condition $\alpha\ge L_f+L^2$, \eqref{eq:FandRk} holds, which leads to \eqref{eq:sumPhiykystarupperbound}.

Finally, we show that $R(k)\ge0$ $\forall k\ge0$. Since $\mathbf{q}(0) = \max\{-\mathbf{G}(\mathbf{y}(0)),\mathbf{0}\}$ and $\mathbf{G}(\mathbf{y}^\star)\le \mathbf{0}$, it can be shown that
\begin{align}
\|\mathbf{G}(\mathbf{y}(0))\|^2\le\|\mathbf{q}(0)\|^2+\|\mathbf{G}(\mathbf{y}(0))-\mathbf{G}(\mathbf{y}^\star)\|^2.\label{eq:G<=q+GG}
\end{align} 
Because of Proposition~\ref{prop:propertyofproby}(c) and because $\alpha\ge L^2$, we have $\|\mathbf{G}(\mathbf{y}(0))-\mathbf{G}(\mathbf{y}^\star)\|^2\le\alpha\|\mathbf{y}(0)-\mathbf{y}^\star\|^2$. This, along with \eqref{eq:G<=q+GG}, implies $0\le\frac{\alpha}{2}\|\mathbf{y}(0)-\mathbf{y}^\star\|^2+\frac{1}{2}\|\mathbf{q}(0)\|^2-\frac{1}{2}\|\mathbf{G}(\mathbf{y}(0))\|^2\le R(0)$. For $k\ge 1$, $R(k)\ge0$ can be shown via \eqref{eq:qkplus1normlargethang}.

\subsection{Proof of Lemma \ref{lemma:feasibility}}\label{sec:prooflemmafeasibility} 

Let $k\ge 1$. To prove \eqref{eq:Gbaryklemma}, note from \eqref{eq:virtualqueueq} that $\mathbf{q}(\ell)\ge \mathbf{q}(\ell-1)+\mathbf{G}(\mathbf{y}(\ell))$ $\forall \ell\ge 1$. Adding this over $\ell=1,\ldots,k$ results in $\mathbf{q}(k)\ge \mathbf{q}(0)+\sum_{\ell=1}^k \mathbf{G}(\mathbf{y}(\ell))$. Moreover, since $\mathbf{q}(0)\ge \mathbf{0}$,
\begin{equation}\label{eq:sumGupperbound}
\sum_{\ell=1}^k \mathbf{G}(\mathbf{y}(\ell))\le \mathbf{q}(k).
\end{equation}
Furthermore, since (each row of) $\mathbf{G}(\mathbf{y})$ is convex, $\mathbf{G}(\bar{\mathbf{y}}(k))\le \frac{1}{k}\sum_{\ell=1}^k \mathbf{G}(\mathbf{y}(\ell)).$ This, along with \eqref{eq:sumGupperbound}, implies \eqref{eq:Gbaryklemma}.

To prove \eqref{eq:Bbaryklemma}, note from \eqref{eq:WHsymmetricity}, \eqref{eq:H1=0}, and \eqref{eq:pextrazupdate} that $(\mathbf{1}_n\otimes I_{\tilde{m}+\tilde{p}})^T \mathbf{z}(\ell+1) = (\mathbf{1}_n\otimes I_{\tilde{m}+\tilde{p}})^T\mathbf{z}(\ell)$ $\forall\ell\ge 0$. Also, $(\mathbf{1}_n\otimes I_{\tilde{m}+\tilde{p}})^T\mathbf{z}(0)=\rho (\mathbf{1}_n\otimes I_{\tilde{m}+\tilde{p}})^TH\mathbf{u}(0)=\mathbf{0}_{\tilde{m}+\tilde{p}}$. Therefore,
\begin{equation}\label{eq:sumzkissumz0}
(\mathbf{1}_n\otimes I_{\tilde{m}+\tilde{p}})^T\mathbf{z}(\ell)=\mathbf{0}_{\tilde{m}+\tilde{p}},\quad\forall \ell\ge 0.
\end{equation}
Additionally, from \eqref{eq:WHsymmetricity}, \eqref{eq:W1is1}, and \eqref{eq:relationofdualandprimal},
\begin{align}
&(\mathbf{1}_n\otimes I_{\tilde{m}+\tilde{p}})^T(\mathbf{u}(\ell+1)-\mathbf{u}(\ell))\nonumber\displaybreak[0]\\
=&\frac{1}{\rho}(\mathbf{1}_n\otimes I_{\tilde{m}+\tilde{p}})^T(\mathbf{B}\mathbf{y}(\ell+1)-\mathbf{c}-\mathbf{z}(\ell)),\quad\forall\ell\ge 0.\label{eq:sumusuccessive}
\end{align}
Now we substitute \eqref{eq:sumzkissumz0} into \eqref{eq:sumusuccessive}, and then add the resulting equation from $\ell=0$ to $\ell=k-1$, leading to 
\begin{align}
&(\mathbf{1}_n\otimes I_{\tilde{m}+\tilde{p}})^T(\mathbf{u}(k)-\mathbf{u}(0))\nonumber\displaybreak[0]\\
=&\frac{1}{\rho}\sum_{\ell=1}^k(\mathbf{1}_n\otimes I_{\tilde{m}+\tilde{p}})^T(\mathbf{B}\mathbf{y}(\ell)-\mathbf{c}).\label{eq:sumbylminusc}
\end{align}
This, together with $(\mathbf{1}_n\otimes I_{\tilde{m}+\tilde{p}})^T(\mathbf{B}\bar{\mathbf{y}}(k)-\mathbf{c})=\frac{1}{k}\sum_{\ell=1}^k(\mathbf{1}_n\otimes I_{\tilde{m}+\tilde{p}})^T(\mathbf{B}\mathbf{y}(\ell)-\mathbf{c})$, gives \eqref{eq:Bbaryklemma}.

Finally, to prove \eqref{eq:Bsbaryklemma}, note that 
$\|\mathbf{B}^s\bar{\mathbf{y}}(k)-\mathbf{c}^s\|^2=\|\mathbf{B}^s(\bar{\mathbf{y}}(k)-\mathbf{y}^\star)\|^2=(\bar{\mathbf{y}}(k)-\mathbf{y}^\star)^T(\mathbf{B}^s)^T\mathbf{B}^s((\mathbf{B}^s)^T\mathbf{B}^s)^\dag(\mathbf{B}^s)^T\mathbf{B}^s(\bar{\mathbf{y}}(k)-\mathbf{y}^\star)=\|(\mathbf{B}^s)^T\mathbf{B}^s(\bar{\mathbf{y}}(k)-\mathbf{y}^\star)\|_{((\mathbf{B}^s)^T\mathbf{B}^s)^\dag}^2=\|(\mathbf{B}^s)^T(\mathbf{B}^s\bar{\mathbf{y}}(k)-\mathbf{c}^s)\|_{((\mathbf{B}^s)^T\mathbf{B}^s)^\dag}^2$. In addition, due to \eqref{eq:sparsetildevupdate},
\begin{align}%\label{eq:sumvkupperbound}
&(\mathbf{B}^s)^T(\mathbf{B}^s\bar{\mathbf{y}}(k)-\mathbf{c}^s)=\frac{1}{k}\sum_{\ell=1}^k(\mathbf{B}^s)^T(\mathbf{B}^s\mathbf{y}(\ell)-\mathbf{c}^s)\nonumber\displaybreak[0]\\
&=\frac{1}{k}\sum_{\ell=1}^k\frac{1}{\gamma}(\mathbf{v}(\ell)-\mathbf{v}(\ell-1))=\frac{\mathbf{v}(k)-\mathbf{v}(0)}{\gamma k}.\nonumber
\end{align}
It follows that \eqref{eq:Bsbaryklemma} holds.

\subsection{Proof of Theorem \ref{theo:funcval}}\label{ssec:proofoffuncval}

Let $k\ge1$. We first show the boundedness of $\|\mathbf{q}(k)\|$, $\|(\mathbf{1}\otimes I_{\tilde{m}+\tilde{p}})^T(\mathbf{u}(k)-\mathbf{u}(0))\|$, and $\|\mathbf{v}(k)-\mathbf{v}(0)\|_{((\mathbf{B}^s)^T\mathbf{B}^s)^{\dag}}$ by virtue of Lemma~\ref{lemma:boundedsumF}, so that the constraint violations of \eqref{eq:originalprobleminy} in Lemma~\ref{lemma:feasibility} can be bounded by $O(1/k)$. 

To do so, note from \eqref{eq:ystarminimum} that $\Phi(\mathbf{y}^\star)\le\Phi(\mathbf{y}(\ell))+\langle\mathbf{q}^\star,\mathbf{G}(\mathbf{y}(\ell))\rangle+\langle u^\star,(\mathbf{1}_n\otimes I_{\tilde{m}+\tilde{p}})^T(\mathbf{B}\mathbf{y}(\ell)-\mathbf{c})\rangle+\langle\tilde{\mathbf{v}}^\star,\mathbf{B}^s\mathbf{y}(\ell)-\mathbf{c}^s\rangle$ $\forall\ell\ge1$. Adding this inequality over $\ell=1,\ldots,k$ yields
\begin{align}
&\sum_{\ell=1}^k(\Phi(\mathbf{y}^\star)-\Phi(\mathbf{y}(\ell)))\le\langle u^\star,\sum_{\ell=1}^k(\mathbf{1}_n\!\otimes\!I_{\tilde{m}+\tilde{p}})^T(\mathbf{B}\mathbf{y}(\ell)\!-\!\mathbf{c})\rangle\nonumber\displaybreak[0]\\
&+\langle\mathbf{q}^\star,\sum_{\ell=1}^k\mathbf{G}(\mathbf{y}(\ell))\rangle+\langle\tilde{\mathbf{v}}^\star,\sum_{\ell=1}^k(\mathbf{B}^s\mathbf{y}(\ell)-\mathbf{c}^s)\rangle.\label{eq:sumphiphi<=usum1BcqsumGvsumBc}
\end{align}
In addition, from \eqref{eq:Bsbaryklemma},
\begin{align}
\|\sum_{\ell=1}^k(\mathbf{B}^s\mathbf{y}(\ell)-\mathbf{c}^s)\|&=k\|\mathbf{B}^s\bar{\mathbf{y}}(k)-\mathbf{c}^s\|\nonumber\displaybreak[0]\\
&=\frac{1}{\gamma}\|\mathbf{v}(k)-\mathbf{v}(0)\|_{((\mathbf{B}^s)^T\mathbf{B}^s)^{\dag}}.\label{eq:sumbsyminusc}
\end{align}
Also, because $\mathbf{q}^\star\ge\mathbf{0}$ and because of \eqref{eq:sumGupperbound}, 
\begin{align}
\langle\mathbf{q}^\star,\sum_{\ell=1}^k\mathbf{G}(\mathbf{y}(\ell))\rangle\le\langle\mathbf{q}^\star,\mathbf{q}(k)\rangle.\label{eq:qsumG<=qq}
\end{align}
By incorporating \eqref{eq:sumbylminusc}, \eqref{eq:sumbsyminusc}, and \eqref{eq:qsumG<=qq} into \eqref{eq:sumphiphi<=usum1BcqsumGvsumBc}, we obtain
\begin{align}
&\!\!\!\!\!\sum_{\ell=1}^k\!(\Phi(\mathbf{y}^\star)\!-\!\Phi(\mathbf{y}(k)))\!\le\!\rho\|u^\star\|\!\cdot\!\|(\mathbf{1}_n\otimes I_{\tilde{m}+\tilde{p}})^T\!(\mathbf{u}(k)\!-\!\mathbf{u}(0))\|\nonumber\displaybreak[0]\\
&+\|\mathbf{q}^\star\|\cdot\|\mathbf{q}(k)\|+\frac{1}{\gamma}\|\tilde{\mathbf{v}}^\star\|\cdot\|\mathbf{v}(k)-\mathbf{v}(0)\|_{((\mathbf{B}^s)^T\mathbf{B}^s)^{\dag}}.\label{eq:sumPhiystaryk1upperbound}
\end{align}
Furthermore, since $P^W=(P^W)^T\succeq \mathbf{O}$ and $P^W\mathbf{1}_n=\mathbf{1}_n$, we have $P^W-\frac{\mathbf{1}_n\mathbf{1}_n^T}{n}\succeq \mathbf{O}$, which implies $W=P^W\otimes I_{\tilde{m}+\tilde{p}}\succeq \frac{(\mathbf{1}_n\otimes I_{\tilde{m}+\tilde{p}})(\mathbf{1}_n\otimes I_{\tilde{m}+\tilde{p}})^T}{n}$. Consequently,
\begin{align}
&\|(\mathbf{1}_n\otimes I_{\tilde{m}+\tilde{p}})^T(\mathbf{u}(k)-\mathbf{u}(0))\|\le\sqrt{n}\|\mathbf{u}(k)-\mathbf{u}(0)\|_W\nonumber\displaybreak[0]\\
&\le\sqrt{n}(\|\mathbf{u}(k)\|_W+\|\mathbf{u}(0)\|_W).\label{eq:1uku0upperbound}
\end{align}
By substituting \eqref{eq:1uku0upperbound} into \eqref{eq:sumPhiystaryk1upperbound} and then utilizing Lemma~\ref{lemma:boundedsumF}, 
\begin{align}
&R(k)\le R(0)+\|\mathbf{q}^\star\|\cdot\|\mathbf{q}(k)\|\nonumber\displaybreak[0]\\
&+\rho\sqrt{n}\|u^\star\|\Bigl(\|\mathbf{u}(k)\|_W+\|\mathbf{u}(0)\|_W\Bigr)\nonumber\displaybreak[0]\\
&+\frac{1}{\gamma}\|\tilde{\mathbf{v}}^\star\|\Bigl(\|\mathbf{v}(k)\|_{((\mathbf{B}^s)^T\mathbf{B}^s)^\dag}+\|\mathbf{v}(0)\|_{((\mathbf{B}^s)^T\mathbf{B}^s)^\dag}\Bigr).\label{eq:Rkupperboundcrossterms}
\end{align}
On the other hand, from \eqref{eq:qkplus1normlargethang}, Proposition~\ref{prop:propertyofproby}(c), and $L^2\le\alpha$,
\begin{align*}
&\|\mathbf{G}(\mathbf{y}(k))\|^2\le\frac{1}{2}\|\mathbf{q}(k)\|^2\!+\!\frac{1}{2}\|\mathbf{G}(\mathbf{y}(k))\!-\!\mathbf{G}(\mathbf{y}^\star)\!+\!\mathbf{G}(\mathbf{y}^\star)\|^2\displaybreak[0]\\
&\le\frac{1}{2}\|\mathbf{q}(k)\|^2+\|\mathbf{G}(\mathbf{y}(k))-\mathbf{G}(\mathbf{y}^\star)\|^2+\|\mathbf{G}(\mathbf{y}^\star)\|^2\displaybreak[0]\\
&\le\frac{1}{2}\|\mathbf{q}(k)\|^2+\alpha\|\mathbf{y}(k)-\mathbf{y}^\star\|^2+\|\mathbf{G}(\mathbf{y}^\star)\|^2.
\end{align*}
It follows from the definition of $R(k)$ in Lemma~\ref{lemma:boundedsumF} that 
\begin{align*}
R(k)\ge &~\frac{\rho}{2}\|\mathbf{u}(k)\|_W^2+\frac{1}{2\gamma}\|\mathbf{v}(k)\|_{((\mathbf{B}^s)^T\mathbf{B}^s)^\dag}^2\displaybreak[0]\\
&+\frac{1}{4}\|\mathbf{q}(k)\|^2-\frac{\|\mathbf{G}(\mathbf{y}^\star)\|^2}{2}.
\end{align*}
Since the right-hand side of the above inequality cannot exceed that of \eqref{eq:Rkupperboundcrossterms}, we are able to derive
\begin{align}
&\frac{1}{4}\Bigl(\|\mathbf{q}(k)\|-2\|\mathbf{q}^\star\|\Bigr)^2+\frac{\rho}{2}\Bigl(\|\mathbf{u}(k)\|_W\!-\!\sqrt{n}\|u^\star\|\Bigr)^2\nonumber\displaybreak[0]\\
&+\!\frac{1}{2\gamma}\Bigl(\|\mathbf{v}(k)\|_{((\mathbf{B}^s)^T\mathbf{B}^s)^\dag}\!-\!\|\tilde{\mathbf{v}}^\star\|\Bigr)^2\le C.\label{eq:upperboundedbyc}
\end{align}
From \eqref{eq:upperboundedbyc}, we obtain $\frac{1}{4}\Bigl(\|\mathbf{q}(k)\|-2\|\mathbf{q}^\star\|\Bigr)^2\le C$, which gives
\begin{align}
\|\mathbf{q}(k)\|\le V_{\mathbf{q}}\label{eq:qbound}.
\end{align}
Due again to \eqref{eq:upperboundedbyc}, $\|\mathbf{u}(k)\|_W\le\sqrt{n}\|u^\star\|+\sqrt{\frac{2C}{\rho}}$. It then follows from \eqref{eq:1uku0upperbound} that
\begin{align}
\|(\mathbf{1}_n\otimes I_{\tilde{m}+\tilde{p}})^T(\mathbf{u}(k)-\mathbf{u}(0))\|\le V_{\mathbf{u}}.\label{eq:ubound}
\end{align}
Similarly, we derive from \eqref{eq:upperboundedbyc} that 
\begin{align}
\|\mathbf{v}(k)-\mathbf{v}(0)\|_{((\mathbf{B}^s)^T\mathbf{B}^s)^{\dag}}\le V_{\mathbf{v}}.\label{eq:vbound}
\end{align}

As a result, \eqref{eq:Gerror}--\eqref{eq:Bserror} can be directly derived from Lemma~\ref{lemma:feasibility} and \eqref{eq:qbound}--\eqref{eq:vbound}. To show \eqref{eq:theofuncvalbound}, note from the convexity of $\Phi$, Lemma~\ref{lemma:boundedsumF}, and $R(k)\ge0$ that
\begin{align*}
\Phi(\bar{\mathbf{y}}(k))-\Phi(\mathbf{y}^\star)\le\frac{1}{k}\sum_{\ell=1}^k(\Phi(\mathbf{y}(\ell))-\Phi(\mathbf{y}^\star))\le \frac{R(0)}{k}.
\end{align*}
Moreover, using $\mathbf{q}^\star\ge\mathbf{0}$, \eqref{eq:ystarminimum}--\eqref{eq:Bserror}, and the Cauchy-Schwarz inequality, we obtain
\begin{align*}
&\Phi(\bar{\mathbf{y}}(k))-\Phi(\mathbf{y}^\star)\ge-\langle \mathbf{q}^\star,\mathbf{G}(\bar{\mathbf{y}}(k))\rangle-\langle\tilde{\mathbf{v}}^\star,\mathbf{B}^s\bar{\mathbf{y}}(k)-\mathbf{c}^s\rangle\displaybreak[0]\\
&-\langle u^\star,(\mathbf{1}_n\otimes I_{\tilde{m}+\tilde{p}})(\mathbf{B}\bar{\mathbf{y}}(k)\!-\!\mathbf{c})\rangle\ge-\frac{R_{\Phi}}{k}.
\end{align*}
Therefore, \eqref{eq:theofuncvalbound} holds.

\subsection{Proof of Corollary~\ref{cor:originalprobconv}}\label{ssec:proofofcororiginprob}

First, from the definition of $\Phi$ in Section~\ref{ssec:probtrans}, \eqref{eq:funcvalbound} is equivalent to \eqref{eq:theofuncvalbound}. Then, from the definitions of $\mathbf{B}^s$ and $\mathbf{c}^s$, we have $\|\mathbf{B}^s\bar{\mathbf{y}}(k)-\mathbf{c}^s\|=\sqrt{\sum_{i\in\mathcal{V}^{\operatorname{eq}}}\|(\sum_{j\in S_i^{\operatorname{eq}}}A_{ij}^s\bar{x}_j(k))-b_i^s\|^2}\ge\|(\sum_{j\in S_i^{\operatorname{eq}}}A_{ij}^s\bar{x}_j(k))-b_i^s\|$ $\forall i\in\mathcal{V}^{\operatorname{eq}}$, which, together with \eqref{eq:Bserror}, leads to \eqref{eq:sparseeqconv}. From the definitions of $\mathbf{B}$ and $\mathbf{c}$, 
\begin{align}
&\|(\mathbf{1}\otimes I_{\tilde{m}+\tilde{p}})^T(\mathbf{B}\bar{\mathbf{y}}(k)-\mathbf{c})\|^2\nonumber\displaybreak[0]\\
=&\|\sum_{i\in \mathcal{V}} \bar{t}_i(k)\|^2+\|\sum_{i\in \mathcal{V}} (A_i\bar{x}_i(k)\!-\!b_i)\|^2,\label{eq:1IByc=sqrtsumt+sumAxb}
\end{align}
where $\bar{t}_i(k)=\frac{1}{k}\sum_{\ell=1}^k t_i(\ell)$ with each $t_i(\ell)$ generated by Algorithm~\ref{alg:IPLUX}. Due to \eqref{eq:1IByc=sqrtsumt+sumAxb}, $\|\sum_{i\in \mathcal{V}} (A_i\bar{x}_i(k)-b_i)\|\le \|(\mathbf{1}\otimes I_{\tilde{m}+\tilde{p}})^T(\mathbf{B}\bar{\mathbf{y}}(k)-\mathbf{c})\|$. Because of this and \eqref{eq:Berror}, we can see that \eqref{eq:globeqconv} is satisfied. Due to the definition of $\mathbf{G}$ and \eqref{eq:Gerror}, we know that \eqref{eq:sparseineqconv} holds and
\begin{equation}\label{eq:gbarxikminusti}
g_i(\bar{x}_i(k))-\bar{t}_i(k)\le \frac{V_{\mathbf{q}}}{k}\mathbf{1}_{\tilde{p}},\quad\forall i\in \mathcal{V}.
\end{equation}
Moreover, \eqref{eq:1IByc=sqrtsumt+sumAxb} also suggests
\begin{align}
\|\sum_{i\in \mathcal{V}} \bar{t}_i(k)\|\le \|(\mathbf{1}_n\otimes I_{\tilde{m}+\tilde{p}})^T(\mathbf{B}\bar{\mathbf{y}}(k)-\mathbf{c})\|.\label{eq:sumt<=1IByc}
\end{align}
Furthermore, $\sum_{i\in \mathcal{V}} g_i(\bar{x}_i(k))=\sum_{i\in\mathcal{V}}(g_i(\bar{x}_i(k))-\bar{t}_i(k))+\sum_{i\in \mathcal{V}} \bar{t}_i(k)\le\sum_{i\in \mathcal{V}}(g_i(\bar{x}_i(k))-\bar{t}_i(k))+\|\sum_{i\in \mathcal{V}} \bar{t}_i(k)\|\mathbf{1}_{\tilde{p}}$. It follows from \eqref{eq:gbarxikminusti}, \eqref{eq:sumt<=1IByc}, and \eqref{eq:Berror} that \eqref{eq:globineqconv} holds. 

\bibliographystyle{IEEEtran}
\bibliography{reference}

\begin{thebibliography}{10}
\providecommand{\url}[1]{#1}
\csname url@samestyle\endcsname
\providecommand{\newblock}{\relax}
\providecommand{\bibinfo}[2]{#2}
\providecommand{\BIBentrySTDinterwordspacing}{\spaceskip=0pt\relax}
\providecommand{\BIBentryALTinterwordstretchfactor}{4}
\providecommand{\BIBentryALTinterwordspacing}{\spaceskip=\fontdimen2\font plus
\BIBentryALTinterwordstretchfactor\fontdimen3\font minus
  \fontdimen4\font\relax}
\providecommand{\BIBforeignlanguage}[2]{{%
\expandafter\ifx\csname l@#1\endcsname\relax
\typeout{** WARNING: IEEEtran.bst: No hyphenation pattern has been}%
\typeout{** loaded for the language `#1'. Using the pattern for}%
\typeout{** the default language instead.}%
\else
\language=\csname l@#1\endcsname
\fi
#2}}
\providecommand{\BIBdecl}{\relax}
\BIBdecl

\bibitem{Nedic15}
A.~Nedi\'{c} and A.~Olshevsky, ``Distributed optimization over time-varying
  directed graphs,'' \emph{IEEE Transactions on Automatic Control}, vol.~60,
  no.~3, pp. 601-- 615, 2015.

\bibitem{Giselsson13}
P.~Giselsson, M.~D. Doan, T.~Keviczky, B.~Schutter, and A.~Rantzer,
  ``Accelerated gradient methods and dual decomposition in distributed model
  predictive control,'' \emph{Automatica}, vol.~49, no.~3, pp. 829--833, 2013.

\bibitem{Yang2016}
T.~Yang, J.~Lu, D.~Wu, J.~Wu, G.~Shi, Z.~Meng, and K.~H. Johansson, ``A
  distributed algorithm for economic dispatch over time-varying directed
  networks with delays,'' \emph{IEEE Transactions on Industrial Electronics},
  vol.~64, no.~6, pp. 5095--5106, 2016.

\bibitem{Nedic10}
A.~Nedi\'{c}, A.~Ozdaglar, and P.~A. Parrilo, ``Constrained consensus and
  optimization in multi-agent networks,'' \emph{IEEE Transactions on Automatic
  Control}, vol.~55, no.~4, pp. 922--938, 2010.

\bibitem{Nedic17}
A.~Nedi\'{c}, A.~Olshevsky, and W.~Shi, ``Achieving geometric convergence for
  distributed optimization over time-varying graphs,'' \emph{SIAM Journal on
  Optimization}, vol.~27, no.~4, pp. 2597--2633, 2017.

\bibitem{ShiW15}
W.~Shi, Q.~Ling, G.~Wu, and W.~Yin, ``{EXTRA}: an exact first-order algorithm
  for decentralized consensus optimization,'' \emph{SIAM Journal on
  Optimization}, vol.~25, no.~2, pp. 944--966, 2015.

\bibitem{ShiW15a}
------, ``A proximal gradient algorithm for decentralized composite
  optimization,'' \emph{IEEE Transactions on Signal Processing}, vol.~63,
  no.~22, pp. 6013--6023, 2015.

\bibitem{QuG19}
G.~Qu and N.~Li, ``Accelerated distributed {N}esterov gradient descent,''
  accepted to {\em IEEE Transactions on Automatic Control}, 2019.

\bibitem{Koshal11}
J.~Koshal, A.~Nedi\'{c}, and U.~V. Shanbhag, ``Multiuser optimization:
  Distributed algorithms and error analysis,'' \emph{SIAM Journal on
  Optimization}, vol.~21, no.~3, pp. 1046--1081, 2011.

\bibitem{Wu19}
X.~Wu and J.~Lu, ``Fenchel dual gradient methods for distributed convex
  optimization over time-varying networks,'' \emph{IEEE Transactions on
  Automatic Control}, vol.~64, no.~11, pp. 4629--4636, 2019.

\bibitem{XiaoL06b}
L.~Xiao and S.~Boyd, ``Optimal scaling of a gradient method for distributed
  resource allocation,'' \emph{Journal of Optimization Theory and
  Applications}, vol. 129, no.~3, pp. 469--488, 2006.

\bibitem{Lakshmanan08}
H.~Lakshmanan and D.~P. de~Farias, ``Decentralized resource allocation in
  dynamic networks of agents,'' \emph{SIAM Journal on Optimization}, vol.~19,
  no.~2, pp. 911--940, 2008.

\bibitem{Nedic17a}
A.~Nedi\'{c}, A.~Olshevsky, and W.~Shi, ``Improved convergence rates for
  distributed resource allocation,'' \emph{arXiv preprint arXiv:1706.05441},
  2017.

\bibitem{ChangTH14}
T.~Chang, A.~Nedi\'{c}, and A.~Scaglione, ``Distributed constrained
  optimization by consensus-based primal-dual perturbation method,'' \emph{IEEE
  Transactions on Automatic Control}, vol.~59, no.~6, pp. 1524--1538, 2014.

\bibitem{Falsone17}
A.~Falsone, K.~Margellos, S.~Garatti, and M.~Prandini, ``Dual decomposition for
  multi-agent distributed optimization with coupling constraints,''
  \emph{Automatica}, vol.~84, pp. 149--158, 2017.

\bibitem{Notarnicola20}
I.~Notarnicola and G.~Notarstefano, ``Constraint-coupled distributed
  optimization: A relaxation and duality approach,'' \emph{IEEE Transactions on
  Control of Network Systems}, vol.~7, no.~1, pp. 483 -- 492, 2020.

\bibitem{Liang19}
S.~Liang, L.~Wang, and G.~Yin, ``Distributed smooth convex optimization with
  coupled constraints,'' accepted to {\em IEEE Transactions on Automatic
  Control}, 2019.

\bibitem{Liang19a}
------, ``Distributed dual subgradient algorithms with iterate-averaging
  feedback for convex optimization with coupled constraints,'' accepted to {\em
  IEEE Transactions on Cybernetics}, 2019.

\bibitem{Wu19a}
X.~Wu and J.~Lu, ``Improved convergence rates of {P-EXTRA} for non-smooth
  distributed optimization,'' in \emph{Proc. IEEE International Conference on
  Control and Automation}, Edinburgh, UK, 2019, pp. 55--60.

\bibitem{Boyd11}
S.~Boyd, N.~Parikh, E.~Chu, B.~Peleato, and J.~Eckstein, ``Distributed
  optimization and statistical learning via the alternating direction method of
  multipliers,'' \emph{Foundations and Trends in Machine Learning}, vol.~3,
  no.~1, pp. 1--122, 2011.

\bibitem{YuH17}
H.~Yu and M.~J. Neely, ``A simple parallel algorithm with an ${O}(1/t)$
  convergence rate for general convex programs,'' \emph{SIAM Journal on
  Optimization}, vol.~27, no.~2, pp. 759--783, 2017.

\bibitem{WuX20}
X.~Wu, H.~Wang, and J.~Lu, ``A distributed proximal primal-dual algorithm for
  nonsmooth optimization with coupling constraints,'' in \emph{Proc. IEEE
  Conference on Decision and Control}, Jeju Island, Republic of Korea, 2020,
  pp. 3657--3662.

\bibitem{Bertsekas99}
D.~P. Bertsekas, \emph{Nonlinear Programming}.\hskip 1em plus 0.5em minus
  0.4em\relax Belmont, MA: Athena Scientific, 1999.

\bibitem{Diamond16}
S.~Diamond and S.~Boyd, ``{CVXPY}: A {P}ython-embedded modeling language for
  convex optimization,'' \emph{The Journal of Machine Learning Research},
  vol.~17, no.~1, pp. 2909--2913, 2016.

\bibitem{yanai11}
H.~Yanai, K.~Takeuchi, and Y.~Takane, \emph{Projection Matrices, Generalized
  Inverse Matrices, and Singular Value Decomposition}.\hskip 1em plus 0.5em
  minus 0.4em\relax New York, NY: Springer, 2011.

\end{thebibliography}

\end{document}